\documentclass[10pt]{amsart}

\usepackage{graphicx}
\usepackage{geometry}
\usepackage{appendix}
\usepackage{color}
\usepackage{tikz-cd}

\definecolor{refblue}{RGB}{0, 0, 153}
\definecolor{citegreen}{RGB}{0, 115, 0}
\definecolor{linkred}{RGB}{191, 26, 61}

\usepackage{leftidx}
\usepackage[all]{xy}
\usepackage[colorlinks,linkcolor=refblue,citecolor=citegreen,urlcolor=linkred,bookmarks=false,hypertexnames=true]{hyperref} 
\usepackage[nameinlink]{cleveref}
\usepackage{amssymb}
\usepackage{setspace}
\usepackage{stmaryrd}
\usepackage{amsthm}
\usepackage{colonequals}

\usepackage{titletoc}
\titlecontents{section}[4em]{\normalfont\bfseries}
    {\contentslabel{2em}}{\hspace{-2em}}{\titlerule*[1pc]{.}\contentspage}
\titlecontents{subsection}[7em]{\normalfont}
    {\contentslabel{3.2em}}{}{\titlerule*[1pc]{.}\contentspage}
\usepackage{titlesec}
\titleformat{\section}{\normalfont\scshape\center}{\thesection}{1em}{}
\titleformat{\subsection}{\normalfont\bfseries}{\thesubsection}{1em}{}

\newtheorem{theorem}{Theorem}[section]
\newtheorem{lemma}[theorem]{Lemma}

\newtheorem{proposition}[theorem]{Proposition}
\newtheorem{corollary}[theorem]{Corollary}
\newtheorem{maintheorem}{Theorem}

\theoremstyle{definition}
\newtheorem{definition}[theorem]{Definition}
\newtheorem*{definition*}{Definition}
\newtheorem{example}[theorem]{Example}

\theoremstyle{remark}
\newtheorem{remark}[theorem]{Remark}

\numberwithin{equation}{section}

\newcommand{\stackcite}[1]{\cite[\href{https://stacks.math.columbia.edu/tag/#1}{#1}]{stk}}

\usepackage{tikz}
\usetikzlibrary{shapes.geometric, arrows}

\newcommand{\bbC}{\mathbb C}

\newcommand{\bbG}{\mathbb G}

\newcommand{\bbN}{\mathbb N}

\newcommand{\bbQ}{\mathbb Q}

\newcommand{\bbX}{\mathbb X}

\newcommand{\bbZ}{\mathbb Z}

\newcommand{\calA}{\mathcal A}
\newcommand{\calB}{\mathcal B}
\newcommand{\calC}{\mathcal C}

\newcommand{\calE}{\mathcal E}
\newcommand{\calF}{\mathcal F}
\newcommand{\calG}{\mathcal G}

\newcommand{\calI}{\mathcal I}

\newcommand{\calL}{\mathcal L}
\newcommand{\calM}{\mathcal M}

\newcommand{\calO}{\mathcal O}
\newcommand{\calP}{\mathcal P}

\newcommand{\calT}{\mathcal T}

\newcommand{\calV}{\mathcal V}

\newcommand{\frakm}{\mathfrak m}

\newcommand{\frakp}{\mathfrak p}

\renewcommand{\sl}{\mathfrak{sl}}

\newcommand{\GL}{\operatorname{GL}} 
\newcommand{\SL}{\operatorname{SL}} 
\newcommand{\Sp}{\operatorname{Sp}} 
\newcommand{\GSp}{\operatorname{GSp}} 
\newcommand{\OO}{\operatorname{O}} 
\newcommand{\Gm}{\operatorname{\bbG}_{\mathrm m}} 

\newcommand{\Gpd}{\operatorname{Gpd}} 
\newcommand{\Grp}{\operatorname{Grp}} 
\newcommand{\Rep}{\operatorname{Rep}} 
\newcommand{\Sch}{\operatorname{Sch}} 
\newcommand{\Set}{\operatorname{\mathrm{Set}}} 

\newcommand{\Quot}{\operatorname{Quot}} 
\newcommand{\Spec}{\operatorname{Spec}} 

\newcommand{\Alg}{\operatorname{Alg}} 
\newcommand{\CAlg}{\operatorname{CAlg}}
\newcommand{\Mod}{\operatorname{Mod}} 

\newcommand{\ab}{\operatorname{ab}} 
\newcommand{\Ad}{\operatorname{Ad}} 
\newcommand{\Aut}{\operatorname{Aut}} 
\newcommand{\chara}{\operatorname{char}} 
\newcommand{\Det}{\operatorname{Det}} 
\newcommand{\diag}{\operatorname{diag}} 
\newcommand{\End}{\operatorname{End}} 
\newcommand{\Hom}{\operatorname{Hom}} 
\newcommand{\id}{\operatorname{id}} 
\newcommand{\map}{\operatorname{Map}} 
\newcommand{\Nat}{\operatorname{\mathrm{Nat}}} 
\newcommand{\Nrd}{\mathrm{Nrd}} 
\newcommand{\Prd}{\mathrm{Prd}} 
\newcommand{\Norm}{\mathrm{N}} 
\newcommand{\ob}{\operatorname{ob}} 
\newcommand{\op}{\operatorname{\mathrm{op}}} 
\newcommand{\PC}{\operatorname{PC}} 
\newcommand{\PGL}{\operatorname{PGL}} 
\newcommand{\pr}{\operatorname{pr}} 
\newcommand{\Rad}{\operatorname{rad}} 
\newcommand{\rank}{\operatorname{rank}} 
\newcommand{\sep}{\operatorname{sep}} 
\newcommand{\sgn}{\operatorname{sgn}} 
\newcommand{\SpDet}{\operatorname{SpDet}} 
\newcommand{\wSpDet}{\operatorname{w-SpDet}} 
\newcommand{\GSpDet}{\operatorname{GSpDet}} 
\newcommand{\wGSpDet}{\operatorname{w-GSpDet}} 
\newcommand{\Sym}{\operatorname{Sym}} 
\newcommand{\tr}{\operatorname{tr}} 
\newcommand{\Pf}{\operatorname{Pf}} 
\newcommand{\pf}{\operatorname{pf}} 
\newcommand{\CH}{\operatorname{CH}} 
\newcommand{\SpRep}{\operatorname{SpRep}} 
\newcommand{\GO}{\operatorname{GO}} 
\newcommand{\PGSp}{\operatorname{PGSp}} 
\newcommand{\PSp}{\operatorname{PSp}} 

\newcommand{\Inner}{\operatorname{Int}} 
\newcommand{\red}{\operatorname{red}} 
\newcommand{\et}{\operatorname{\acute{e}t}} 
\newcommand{\sh}{\operatorname{sh}} 
\newcommand{\eqto}{\xrightarrow{\sim}} 

\newenvironment{enum}{
\begin{enumerate}
  \setlength\itemsep{6pt}
  \setlength{\parskip}{0pt}
  \setlength{\parsep}{0pt}
}{\end{enumerate}}

\usepackage{microtype}

\usepackage{mathrsfs}

\begin{document}

\title{Symplectic determinant laws and invariant theory}

\author{Mohamed Moakher}
\email{moakher@math.univ-paris13.fr}

\author{Julian Quast}
\email{julian.quast@uni-due.de, me@julianquast.de}
\thanks{The second author is funded by the Deutsche Forschungsgemeinschaft (DFG, German Research Foundation) – Project-ID 444845124 – TRR 326 and Project-ID 517234220}


\date{\today}



\begin{abstract} We introduce the notion of \emph{symplectic determinant laws} by analogy with Chenevier's definition of determinant laws. Symplectic determinant laws are a way to define pseudorepresentations for symplectic representations of algebras with involution over arbitrary $\bbZ[\tfrac{1}{2}]$-algebras. We prove that this notion satisfies the properties expected from a good theory of pseudorepresentations, and we compare it to V. Lafforgue's $\Sp_{2d}$-pseudocharacters. In the process, we compute generators of the invariant algebras $A[M_d^m]^{G}$ and $A[G^m]^G$ over an arbitrary commutative ring $A$ when $G \in \{\Sp_d, \mathrm O_d, \GSp_d, \GO_d\}$,  generalizing results of Zubkov.
\end{abstract}

\maketitle

\tableofcontents



\section{Introduction}

Let $R$ be an algebra over an algebraically closed field $k$ of characteristic $p\nmid d!$, and let $\rho : R \to M_d(k)$ be a $k$-algebra homomorphism into the algebra of $d \times d$-matrices $M_d(k)$ over $k$. By the Brauer-Nesbitt theorem, the trace map
$$ \tr\rho : R \to M_d(k), ~r \mapsto \tr(\rho(r)) $$
determines $\rho$ up to semisimplification. Motivated by this fact, Wiles (in \cite{Wiles88} for $d=2$) and Taylor (in \cite{Tay91} for general $d$) introduced the notion of a \emph{$d$-dimensional pseudocharacter}, which is a function $T\colon R\rightarrow k$ that satisfies certain properties that mimic the trace function $\tr(\rho)$. They used this notion to construct Galois representations associated to certain automorphic forms and to study their properties. In fact, relying on results of Procesi (\cite{Pro87}), Taylor showed that over an algebraically closed field $k$ of characteristic zero, a $d$-dimensional pseudocharacter is always the trace of a semisimple representation. This result was extended to characteristic $p\nmid d!$ by Rouquier (see \cite{Rouquier}).

If $p \mid d!$, then the semisimplification of $\rho$ is no longer determined by the trace. However, it was apparent from the work of Procesi that for a characteristic free approach, the determinant should replace the trace. Building upon this perspective, Chenevier \cite{MR3444227} introduced the notion of a \emph{$d$-dimensional determinant law} over a general commutative ring $A$. For a commutative $A$-algebra $R$, a determinant law $D\colon R\rightarrow A$ is a homogenious of degree $d$ multiplicative polynomial law from $R$ to $A$. In essence, it is the datum of a characteristic polynomial $\chi^D(r,t)\in A[t]$ for each $r\in R$, that behave analogously to the characteristic polynomials of a representation $\rho\colon R\rightarrow \GL_d(A)$. 
This theory has proved to be fruitful in the study of Galois deformation rings (\cite{BPI}) and Hecke algebras at Eisenstein primes (\cite{WWE20} and \cite{WWE21}). Furthermore, it plays a significant role in the proof of modularity lifting theorems in the residually reducible case (\cite{ANT} and \cite{Thorne15}), and in the proof of the Bloch-Kato conjecture for the adjoint of automorphic Galois representations (\cite{NT23}).

The goal of this paper is to define and study \emph{symplectic determinant laws} of involutive algebras $(R,*)$ over arbitrary commutative $\bbZ[\tfrac{1}{2}]$-algebras $A$. These are the analogue of determinant laws with respect to symplectic representations of algebras with involution. Inspired by the  invariant theory of tuples of matrices under the action of the symplectic group (see \cite[\S 10]{Pro}), this is done by adding to the determinant law a square root, defined on the set of symmetric elements $R^+$ of the algebra. The main definition is the following:
\begin{definition*}[\Cref{defsympldetlaw}]
    Let $(R,*)$ be an involutive $A$-algebra. A \emph{$2d$-dimensional symplectic determinant law} on $(R,*)$ is the datum of a pair $(D,P)$, where $D \colon R \rightarrow A$ is a $2d$-dimensional determinant law which is invariant under the involution, and $P\colon R^+\rightarrow A$ is a $d$-dimensional homogeneous polynomial law such that $P^2=D|_{R^+}$, $P(1)=1$, and $\CH(P) \subseteq \ker(D)$.
\end{definition*}
The last condition, that the Cayley-Hamilton ideal $\CH(P)$ associated to $P$ lies inside the kernel of $D$, is added for technical reasons.
We distinguish between \emph{weak symplectic determinant laws}, where this condition is omitted and symplectic determinant laws as above.
We conjecture that $\CH(P) \subseteq \ker(D)$ holds for weak symplectic determinant laws.

An important special case is that of a group algebra $R=A[\Gamma]$ equipped with an appropriate involution, where our theory of symplectic determinant laws leads to a notion of "pseudorepresentation" for the symplectic similitude group $\GSp_{2d}$ (see \Cref{gspdet}). In this case, all of our results have an analogue where we replace $\Sp_{2d}$ by $\GSp_{2d}$. 

We also introduce the notion of symplectic Cayley-Hamilton algebras, which are involutive $A$-algebras equipped with a symplectic determinant law $(D,P)$ such that every symmetric element of $R$ satisfies its characteristic polynomial associated to $P$. Matrix algebras equipped with a symplectic involution are a basic example (see \Cref{charpf}). It is worth noting that the properties of Cayley-Hamilton algebras proved in \cite{MR3444227} and \cite{MR3167286} are easily transferred to this context.

Naturally, a symplectic representation gives rise to a symplectic determinant law. Alternatively, we can ask when a given symplectic determinant law arises from a symplectic representation. An important property that one would expect is that over algebraically closed fields, this association provides a bijection between symplectic determinant laws and conjugacy classes of semisimple representations. We verify this for weak symplectic determinant laws, in which case the corresponding statement for symplectic determinant laws is immediate.

\begin{maintheorem}[\Cref{reconsalgcl}] \label{introductionreconsalgcl} Let $k$ be an algebraically closed field with $2 \in k^{\times}$, let $(R,*)$ be an involutive $k$-algebra and let $(D,P)\colon R \to k$ be a weak symplectic determinant law of dimension $2d$. Then there exists a semisimple symplectic representation $(R,*) \to (M_{2d}(k), \mathrm j)$, unique up to conjugation by $\Sp_{2d}(k)$, whose associated symplectic determinant law is $(D,P)$.
\end{maintheorem}

Another important property to expect is that a symplectic determinant law over a Henselian local ring, which is residually absolutely irreducible, should arise from a symplectic representation. We prove this result in \Cref{henselianlemma}.

More generally, we introduce the notion of symplectic GMAs (Generalized Matrix Algebras) following \cite{BC}, which we equip with a canonical symplectic determinant law. This allows us to characterize residually multiplicity-free symplectic Cayley-Hamilton algebras over Henselian local rings, as we show in the following theorem.

\begin{maintheorem}[\Cref{hensCH}]\label{introductionhensCH}
   Suppose that $A$ is a Henselian local ring, and let $(R,*,D,P)$ be a finitely generated $2d$-dimensional symplectic Cayley-Hamilton $A$-algebra with involution. If $D$ is residually  multiplicity-free, then $(R,*)$ admits a symplectic $\text{GMA}$ structure.
\end{maintheorem}
Following \cite{CWE} and \cite{MR3167286}, we undertake the study of the moduli space of symplectic representations of a finitely generated involutive $A$-algebra $(R,*)$, when $A$ is Noetherian. We introduce the space $\SpRep_{(R,*)}^{\square,2d}$ of $2d$-dimensional symplectic representations  of conjugacy classes of $(R,*)$, and we compare it to the space $\SpDet^{2d}_{(R,*)}$ of $2d$-dimensional symplectic determinant laws on $(R,*)$. Given that two conjugate symplectic representations give rise to the same symplectic determinant law, we have the following diagram
\begin{align*}
    \xymatrix{
        \left[\SpRep^{\square,2d}_{(R,*)}/\Sp_{2d}\right]\ar[d] \ar[dr] \\
        \SpDet^{2d}_{(R,*)} & \SpRep^{\square, 2d}_{(R,*)} \sslash \Sp_{2d} \ar[l]^{\nu}
    }
\end{align*}
We prove that the map $\nu$ is close to being an isomorphism. In fact, the coordinate ring map is power-surjective and its kernel consists of $\bbZ$-torsion nilpotent elements. We also show that it is an isomorphism on neighborhoods of points corresponding to multiplicity free representations. Our result is the following.

\begin{maintheorem}[\Cref{adequate}, \Cref{isoonmultiplicityfreelocus}]\label{introductionadequate}The map $
\nu$ is a finite adequate homeomorphism. Moreover, there exsists a Zariski open subset of $\SpRep^{\square, 2d}_{(R,*)} \sslash \Sp_{2d}$, containing the points corresponding to multiplicity free symplectic representations, on which $\nu$ is an isomorphism.
\end{maintheorem}
This shows that the symplectic determinant space is, in some sense, a good approximation of the GIT quotient of $\SpRep^{\square, 2d}_{(R,*)}$ by $\Sp_{2d}$. To prove this theorem, we need to extend the main result of \cite{Pro87}, which states that every Cayley-Hamilton algebra over a characteristic zero field embeds in a matrix algebra compatibly with the trace. This is our modification of the result.

\begin{maintheorem}[\Cref{converseCH}]\label{introconverseCH}
Suppose that $A$ is a commutative $\bbQ$-algebra, and let $(R,*,D\colon R\to A,P\colon R^+\to A)$ be a symplectic Cayley-Hamilton $A$-algebra of degree $2d$. Then there is a commutative $A$-algebra $B$, and an injective symplectic $A$-linear representation
\begin{equation*}
    \rho\colon (R,*)\hookrightarrow (M_{2d}(B),\mathrm{j})
\end{equation*}
inducing $(D,P)$.
\end{maintheorem}
Finally, we compare our notion of symplectic determinant laws for a group algebra $A[\Gamma]$ to Lafforgue's pseudocharacters (see\cite{Laf}). 
To begin, we provide an explicit description of Lafforgue pseudocharacters for the groups $\Sp_{2d}$ and $\GSp_{2d}$ over an arbitrary commutative ring. This is derived from the following result in invariant theory, which may be of interest beyond the scope of this paper.
\begin{maintheorem}[\Cref{invZ}, \Cref{gspZ}] \label{introductioninvtheorem} \phantom{a}
\begin{enum}
    \item $\bbZ[\Sp_{d}^m]^{\Sp_{d}}$ is essentially generated by coefficients of characteristic polynomials.
    \item $\bbZ[\GSp_{d}^m]^{\GSp_{d}}$ is essentially generated by coefficients of characteristic polynomials and the inverse symplectic similitude character.
\end{enum}
\end{maintheorem}
The proof of this theorem relies on the work of Donkin \cite{Don}, but involves a new idea which generalizes to arbitrary semisimple groups once we know the result over algebraically closed fields. Notably, the same theorem for the groups $\OO_d$ and $\GO_d$ ($d \geq 1$) is true thanks to the results in \cite{Zubkov}.

In \cite{emerson2023comparison}, the authors show that the notion of determinant laws is equivalent to Lafforgue's pseudocharacters for $\GL_d$ over arbitrary rings. In our case, we show that this bijection restricts to an injection of the space of Lafforgue's pseudocharacters for $\Sp_{2d}$ (and $\GSp_{2d}$) into the space of symplectic determinant laws. Moreover, we obtain a that this injection is a bijection in some cases.   

\begin{maintheorem}[\Cref{lafforguecomp}]\label{introductionlafforguecomp} If $A$ is a reduced commutative $\bbZ[\frac{1}{2}]$-algebra or if it is an arbitrary $\bbQ$-algebra, then there is a bijection $\PC^\Gamma_{\Sp_{2d}}(A)\xrightarrow{\sim}\SpDet^{(A[\Gamma],*)}_{2d}(A)$ between the space of Lafforgue's pseudocharacters for $\Sp_{2d}$ and the space of $2d$-dimensional symplectic determinant laws on $(A[\Gamma],*)$. 
\end{maintheorem}

Even though Lafforgue's theory of $G$-pseudocharacters works in arbitrary characteristic, our theory of symplectic determinant laws is more explicit, making it more amenable to computations. For instance, the structure of Cayley-Hamilton algebras for $\GL_d$ and GMAs had been used by Bellaïche and Chenevier (\cite{BC}) to produce extensions of Galois representations, and by Wake and Wang-Erickson (\cite{WWE18}) to study the geometry of the eigencurve. Since our theory admits analogs of these constructions, unlike Lafforgue's theory, we expect that our results can be used in a similar way for the symplectic similitude groups $\GSp_{2d}$.

\noindent\textbf{Organization of the paper.} In \Cref{secsymplrep}, we recall basic results on symplectic representations of algebras with involution. In \Cref{secdefsympldetlaw}, we give the definition and some basic properties of symplectic determinant laws. \Cref{secsympldetlawfields} offers a characterization of symplectic determinant laws over fields which leads to \Cref{introductionreconsalgcl}. The notion of symplectic Cayley-Hamilton algebras is introduced in  \Cref{subsecsymplCHAlg}. The subsequent
\Cref{CH0} is dedicated to the proof of \Cref{introconverseCH}. In 
\Cref{SpDetHenselian}, we study residually multiplicity free symplectic Cayley-Hamilton algebras over Henselian local rings using the theory of symplectic GMAs and prove \Cref{introductionhensCH}. 
In \Cref{GITcomp}, we compare the space of symplectic determinant laws with the GIT quotient of the space of symplectic representation by the conjugate action of the symplectic group. This analysis is consolidated in \Cref{introductionadequate}, which requires results from \Cref{modulisection} and \Cref{CH0}.
\Cref{secinvthy} being self-contained, is dedicated to the proof of \Cref{introductioninvtheorem}. Finally, the comparison between Lafforgue's pseudocharacters for $\Sp_{2d}$ and $\GSp_{2d}$ with symplectic determinant laws, as explicated in \Cref{introductionlafforguecomp}, is done in \Cref{secLafforgue}.

\noindent\textbf{Acknowledgements.} The first author would like to express his heartfelt gratitude to Stefano Morra and James Newton for their invaluable guidance and constant support throughout the entirety of this project. Their insights and constant encouragement greatly shaped the direction of this work. He would also like to thank Claudio Procesi for inviting him to the Sapienza University of Rome in March 2022, for the inspiring discussions they had, and for providing him with the key idea behind the proof of \Cref{reconsalgcl}. The second author wants to thank his advisor Gebhard Böckle for continuous support and advice during the writing of this paper. He would also like to thank Ariel Weiss and Stephen Donkin for helpful conversations about classical invariant theory and good filtrations. The debt that this paper owes to the contributions of Joël Bellaïche, Gaëtan Chenevier, Claudio Procesi, and Carl Wang-Erickson is deeply acknowledged and evident throughout the text.

\noindent\textbf{Notation.} Throughout this paper, $A$ will denote a commutative ring with $2\in A^\times$ unless stated otherwise, and $d$ will denote an integer $\ge 1$. We use the following notation:

\begin{enum}
    \item 
    $J \colonequals \left(\begin{array}{cc}
        0 & \id_d \\
        -\id_d & 0
    \end{array}\right) \in M_{2d}(A)$.
    \item Transposition of matrices in $M_d(A)$ is $(-)^{\top}$. It is also called the \emph{orthogonal standard involution} of $M_d(A)$.
    \item If $A=(a_{i,j})_{1\le i,j \le 2d}\in M_{2d}(A)$ is an alternating matrix (i.e. $A^{\top}=-A$), we define its \emph{Pfaffian} by the formula
    $$ \Pf(A)\colonequals\frac{1}{2^d d!}\sum_{\sigma\in \mathfrak{S}_{2d}}\sgn(\sigma)\prod_{i=1}^d a_{\sigma(2i-1),\sigma(2i)}, $$
    where $\mathfrak{S}_{2d}$ is the symmetric group on $2d$ elements. The same definition applies, when $A$ is a quasi-coherent commutative $\calO_S$-algebra on a scheme $S$.
    \item The \emph{symplectic standard involution} $(-)^{\mathrm{j}} \colon M_{2d}(A) \to M_{2d}(A)$ is defined by $M^{\mathrm{j}} \colonequals  JM^{\top}J^{-1}$.
    \item We define the following linear algebraic groups:
    \begin{itemize}
        \item[-] The \emph{symplectic group} $\Sp_{2d}(A) \colonequals  \{M \in \GL_{2d}(A) \mid M^{\mathrm{j}}M = 1\}$.
        \item[-] The \emph{general symplectic group} $\GSp_{2d}(A) \colonequals  \{M \in \GL_{2d}(A)\ | \ M^{\mathrm{j}}M = \lambda(M)\cdot \id, \ \lambda(M)\in A^\times\}$.
        \item[-] The \emph{projective symplectic group} $\PSp_{2d}(A)\colonequals \Sp_{2d}(A)/\{\pm \id\}$.
    \end{itemize}
    \item If $(R,*)$ is an involutive ring, let $R^+ \colonequals  \{x \in R \mid x^* = x\}$ and $R^- \colonequals  \{x \in R \mid x^* = -x\}$. We say, that the elements of $R^+$ are \emph{symmetric} and the elements of $R^-$ are \emph{antisymmetric}.
    \item The \emph{swap involution} is defined as
    $$\mathrm{swap} \colon M_d(A) \times M_d(A) \to M_d(A) \times M_d(A), ~ (a,b) \mapsto (b^{\top}, a^{\top}).$$
    \item If $S$ is a set, we write $(A\langle S\rangle,*)$ for the free (non-commutative) $A$- algebra with involution generated by the symbols $x_s$ and $x_s^*$ for $s\in S$.
    \item $\CAlg_A$ is the category of commutative $A$-algebras.
    \item $\Grp$ is the category of groups.
    \item $\Gpd$ is the category of groupoids.
    \item $\Sch$ is the category of schemes.
    \item If $\calC$ is a category and $X$ is an object of $\calC$, we write $\calC_{/X}$ for the slice category over $X$.
    \item $\Sym_A(M)\colonequals \oplus_{k\ge 0} \Sym^k_A(M)$ is the symmetric $A$-algebra of an $A$-module $M$, on which we denote the product by $\odot$.
    \item If $G$ is an affine $A$-group scheme and $M$ is a rational $G$-module, as defined in \cite[\S I.2.7]{Jantzen2003}, we define $$M^G := \{x \in M \mid \forall B \in \CAlg_A : \forall x \in M \otimes_A B : \Delta_M(x) = x \otimes 1\}$$ 
    where $\Delta_M : M \to M \otimes_A A[G]$ is the coaction map (see \cite[§I.2.10]{Jantzen2003}).
    \item If $G$ is an affine $A$-group scheme acting on an affine $\Spec(A)$-scheme $X=\Spec(B)$, we will write $X\sslash G= \Spec(B^G)$ for the $\emph{ GIT quotient}$ of $X$ by $G$.
\end{enum}

\section{Symplectic representations and their moduli spaces}\label{secsymplrep}
\subsection{Symplectic representations}\label{subsymprep} Let $(R,*)$ be an $A$-algebra with involution.
\begin{definition}\label{defsymplrep} Let $B$ a commutative $A$-algebra. A \emph{symplectic representation} of $(R,*)$ is a homomorphism of involutive $A$-algebras $(R,*) \to (M_{2d}(B), \mathrm{j})$. We denote by $ \SpRep_{(R,*)}^{\square,2d}$ the functor 
\begin{align*}
    \SpRep_{(R,*)}^{\square,2d}\colon \CAlg_A &\rightarrow \Set
     \\ B &\mapsto \{\text{symplectic representations }(R,*) \to (M_{2d}(B), \mathrm{j})\}
\end{align*}
of symplectic representations of $(R,*)$.
\end{definition}

We are particularly interested in the case where $R$ is a group algebra over $A$ of a group $\Gamma$. In fact, a representation $\Gamma \to \Sp_{2d}(B)$ can be identified with a symplectic representations $(A[\Gamma], *) \to (M_{2d}(B), \mathrm j)$, where $\gamma^* = \gamma^{-1}$ for $\gamma\in \Gamma$. More generally, for a fixed character $\lambda : \Gamma \to A^{\times}$, a representation $\Gamma \to \GSp_{2d}(B)$ with similitude character $\lambda$ can be identified with a symplectic representation $(A[\Gamma], *) \to (M_{2d}(B), \mathrm j)$, where $\gamma^* = \lambda(\gamma)\gamma^{-1}$ for $\gamma\in \Gamma$.

\begin{lemma}\label{spreprepr}
The functor $\SpRep_{(R,*)}^{\square,2d}$ is representable by a commutative $A$-algebra $A[\SpRep_{(R,*)}^{\square,2d}]$. We let $\rho ^{u}\colon  (R,*)\rightarrow (M_{2d}(A[\SpRep_{(R,*)}^{\square,2d}]),\mathrm{j})$ be the universal representation. If $R$ is a finitely generated $A$-algebra, then $A[\SpRep_{(R,*)}^{\square,2d}]$ is a finitely generated $A$-algebra.
\end{lemma}

\begin{proof}
If $R=A\langle S\rangle=A\langle x_s,x_s^* \ | \ s\in S\rangle$ is the free $A$-algebra with involution on a set $S$, then clearly $A[\SpRep_{(R,*)}^{\square,2d}]$ is equal to the polynomial algebra $A[M_{2d}^S]=A[\bbX_{h,k}^{(s)}\ | \ 1\le h,k\le 2d,\ s\in S]$ and $\rho^{u}(x_s)=\bbX^{(s)}=(\bbX_{h,k}^{(s)})_{h,k}$.

For a general $A$-algebra with involution $R$, there is a presentation $R=A\langle S\rangle/I$ for some involution-stable two-sided ideal $I$ of $A\langle S\rangle$. Then $\rho^{u}(I)$ generates a two-sided ideal in $M_{2d}(A[M_{2d}^S])$ which, as any two-sided ideal in a matrix algebra, is of the form $M_{2d}(J)$ with $J$ an ideal of $A[M_{2d}^S]$. Thus, the universal map for $R$ is given by
\begin{center}
    \begin{tikzcd}
A\langle S\rangle \arrow[r] \arrow[d] &  M_{2d}(A[M_{2d}^S]) \arrow[d]
\\ R  \arrow[r, "\rho^{u}"] & M_{2d}(A[M_{2d}^S]/J)
\end{tikzcd}
\end{center}
By the universal property, $M_{2d}(A[M_{2d}^S]/J)$ is independent of the presentation of $R$.
\end{proof}

Consider the action of $\Sp_{2d,A}$ on the matrices by conjugation, given for every commutative $A$-algebra $B$ by
\begin{align*}
    \Sp_{2d}(B)\times M_{2d}(B)&\to M_{2d}(B)
    \\ (g,M)&\mapsto i_g(M)\colonequals g M g^{-1}.
\end{align*}
This induces an action of $\Sp_{2d,A}$ on $\SpRep^{\square,2d}_{(R,*)}$ such that for a commutative $A$-algebra $B$, we have
\begin{align*}
    \Sp_{2d}(B)\times \SpRep_{(R,*)}^{\square,2d}(B)&\rightarrow \SpRep_{(R,*)}^{\square,2d}(B)
    \\ (g, u)&\mapsto i_g\circ u
\end{align*}
The action of $\Sp_{2d}(B)$ on $\SpRep_{(R,*)}^{\square,2d}$ is by scheme-automorphisms over $B$. In particular, for every $g \in \Sp_{2d}(B)$ we get an induced $B$-algebra automorphism $\widehat g \colon  B[\SpRep^{\square,2d}_{(R,*)}] \to C[\SpRep^{\square,2d}_{(R,*)}]$.

Using this action, we can  equip $M_{2d}(A[\SpRep_{(R,*)}^{\square,2d}])$ with the structure of a $\Sp_{2d,A}$-module by setting the action for each commutative $A$-algebra $B$ to be
\begin{align*}
    \Sp_{2d}(B) \times M_{2d}(B[\SpRep_{(R,*)}^{\square,2d}])&\rightarrow M_{2d}(B[\SpRep_{(R,*)}^{\square,2d}])= B[\SpRep_{(R,*)}^{\square,2d}] \otimes_B M_{2d}(B)
    \\ (g,M) &\mapsto ({\widehat{g}}^{-1}\otimes i_g)(M).
\end{align*}
Unraveling the definition of this action, we get the following lemma.
\begin{lemma}\label{imageofR1}
The image of $R$ by universal representation
$$ \rho^{u} \colon (R,*)\longrightarrow (M_{2d}(A[\SpRep_{(R,*)}),\mathrm{j})$$
lies inside $M_{2d}(A[\SpRep_{(R,*)}^{\square,2d}])^{\Sp_{2d}}$.
\end{lemma}

\subsection{Azumaya algebras with involution}\label{secazumaya} In this subsection, we review a few facts about Azumaya algebras with involution, which will be useful in subsequent discussions.  Recall that an \emph{Azumaya algebra} over a scheme $Y$ is a quasi-coherent unital $\calO_Y$-algebra $\calA$ such that there is an étale covering $\{f_i \colon Y_i \to Y\}_{i \in I}$ of $Y$ so that all $i \in I$, the $\calO_{Y_i}$-algebra $f_i^* \calA$ is isomorphic to a matrix algebra of positive rank over $\calO_{Y_i}$. We see the rank of $\calA$ as a locally constant function $\rank(\calA) \colon |Y| \to \bbZ_{\geq 1}$, where for $y \in Y$, we define $\rank(\calA)(y) \colonequals  \dim_{\kappa(y)}(\calA_y)$.
\begin{definition}
    Let $Y$ be a scheme and let $\calA$ be an Azumaya algebra over $Y$.
    We say, that an $\calO_Y$-linear involution $\sigma \colon \calA \to \calA$ is an \emph{involution of the first kind}.
    If $y \in Y$ and $2 \in \kappa(y)^{\times}$, we say that $\calA$ is \emph{symplectic} (\emph{orthogonal}) at $y$, if the involution on $\calA_y$ is induced by a symplectic (orthogonal) bilinear form.
\end{definition}
From the theory of central simple algebras over fields, it is known that $\calA_y$ is either symplectic or orthogonal when $2 \in \kappa(y)^{\times}$.
In characteristic $2$ every alternating bilinear form is also symmetric, and we won't apply the terminology in these cases.

Let $\calA$ be an Azumaya algebra of constant rank $d^2$ with an involution $\sigma$ over a scheme $S$ with $2 \in \Gamma(Y, \calO_Y)^{\times}$.
Assume for the moment that $Y = \Spec(A)$ is affine. Then $\calA$ is associated to an $A$-algebra $R$. Let $B$ be a faithfully flat $A$-algebra, such that we have a splitting $\alpha\colon B \otimes_A R\xrightarrow{\sim} M_{d}(B)$ of $R$ over $B$ and let $\widetilde{\sigma}=\alpha(1\otimes \sigma)\alpha^{-1}$ be the induced involution on $M_{d}(B)$. The map $x\mapsto\widetilde{\sigma}(x^{\top})$ is an automorphism of $M_d(B)$. We can choose $B$ so that $\widetilde{\sigma}(x)=u(x^{\top})u^{-1}$ for some suitably chosen $u\in \GL_d(B)$ and all $x \in M_d(B)$. The fact that $\widetilde{\sigma}^2=\id$ implies that $u^{\top} = \epsilon u$ for some $\epsilon\in \mu_2(B)$. By \cite[III. Lemma 8.1.1]{knus}, one can choose $B$ so that $\epsilon\in \mu_2(A)$ and this element is independent of the choice of $B$. By descent, we obtain for general schemes $S$ with $2 \in \Gamma(S, \calO_S)^{\times}$ an element $\epsilon \in \mu_2(S)$. We call it the \emph{type} of the involution $\sigma$ on $\calA$.

By our assumption that $2 \in \Gamma(Y, \calO_Y)^{\times}$, $\mu_2$ is a constant group scheme over $Y$ and $\mu_2(Y)$ identifies with the set of locally constant maps $|Y| \to \{\pm 1\}$. In particular, the type of an Azumaya algebra with involution is Zariski-locally constant.
An involution of constant type $1$ is called an \emph{orthogonal involution}, and an involution of constant type $-1$ is called a \emph{symplectic involution}. 
Equivalently an orthogonal (symplectic) involution on $\calA$ is an involution such that for all $y \in Y$, the involution is orthogonal (symplectic) on $\calA_y$.
\begin{lemma}\label{standardformoverschemes} Let $(\mathcal{A}, \sigma)$ be an Azumaya algebra of constant rank $d^2$ over $A$ with involution of the first kind. Then étale locally over $A$, $(\mathcal{A},\sigma)$ has one of the following two forms.
\begin{enum}
    \item $(M_d(A), \mathrm j)$, if $\sigma$ is symplectic.
    \item $(M_d(A), \top)$, if $\sigma$ is orthogonal.
\end{enum}
\end{lemma}
\begin{proof}
    We know that Azumaya algebras of symplectic type are classified by $\check H^1_{\et}(A, \PGSp_d)$ (see \cite[III. §8.5]{knus}).
    It is then sufficient to find an étale trivialization of the cohomology class associated to $(\mathcal{A}, \sigma)$.
    The orthogonal case is treated similarly.
\end{proof}

\subsection{Moduli of symplectic representations}\label{modulisection}
In this subsection, we suppose that $A$ is Noetherian, and write $Y \colonequals \Spec(A)$. We let $(R,*)$ be an $A$-algebra with involution. In analogy to \cite[Definition 2.1]{CWE}, we define the following functors on $Y$-schemes.

\begin{definition}\label{defstacks} \phantom{a}
\begin{enum}
    \item $\SpRep_{(R,*)}^{\square,2d} \colon \Sch_{/Y}^{\op} \to \Set$ is defined by
    \begin{align*}
        \SpRep_{(R,*)}^{\square,2d} (X) \colonequals \left\{
            \begin{array}{l}
                A\text{-algebra morphisms } (R,*)\rightarrow (M_{2d}(\Gamma(X,\mathcal{O}_X)),\mathrm{j}) \\
                \text{respecting the involution}
            \end{array}
        \right\}
    \end{align*}
    \item $\SpRep_{(R,*)}^{2d} \colon \Sch_{/Y}^{\op} \to \mathrm{Gpd}$ is defined by
    \begin{align*}
        \ob \SpRep_{(R,*)}^{2d}(X) \colonequals \left\{ 
            \begin{array}{l} 
                V_{/X} \text{ a rank }2d\text{ vector bundle,} \\
                b\colon  V\times V \rightarrow \calO_X \text{ a non-singular skew-symmetric $\calO_X$-bilinear form,} \\
                \text{and an }  A\text{-algebra morphism } \rho \colon (R,*)\rightarrow (\Gamma(X,\End_{\mathcal{O}_X}(V)),\sigma_b) \\
                \text{respecting the involution}
            \end{array} 
        \right\}
    \end{align*}
    An isomorphism of two objects $(V, b, \rho)$ and $(V', b', \rho')$ is an isomorphism $\alpha \colon V \to V'$, such that $b' \circ (\alpha \times \alpha) = b$ and $\Gamma(X, \End_{\calO_X}(\alpha)) \circ \rho = \rho'$.
    \item $\overline{\SpRep}_{(R,*)}^{2d} \colon \Sch_{/Y}^{\op} \to \mathrm{Gpd}$ is defined by
    \begin{align*}
        \ob \overline{\SpRep}_{(R,*)}^{2d}(X) \colonequals \left\{
            \begin{array}{l}
                 (\mathcal{E},\sigma) \text{ a rank }4d^2 \text{ Azumaya algebra over $X$}  \\
                 \text{equipped with a symplectic involution,} \\
                 \text{and an } A\text{-algebra morphism } \rho \colon (R,*)\rightarrow (\Gamma(X,\mathcal{E}),\sigma) \\
                 \text{respecting the involution}
            \end{array}
        \right\}
    \end{align*}
    An isomorphism of two objects $(\calE, \sigma, \rho)$ and $(\calE', \sigma', \rho')$ is an isomorphism $\alpha \colon \calE \to \calE'$ of Azumaya algebras over $\calO_X$, such that $\alpha \circ \rho = \rho'$.
\end{enum}
\end{definition}

The functor $\SpRep_{(R,*)}^{\square,2d}$ is representable by an affine scheme, which is of finite type over $Y$ if $R$ is finitely generated over $A$.
The functors $\SpRep_{(R,*)}^{2d}$ and $\overline{\SpRep}_{(R,*)}^{2d}$ are (2-)representable by categories fibered in groupoids over $\Sch_{/Y}$.

\begin{lemma}\label{SpVBandSptorsors} Let $X$ be a scheme.
\begin{enum}
    \item There is a natural bijection of pointed sets between the set of symplectic vector bundles of rank $2d$ on $\Sch_{/X}$ up to isomorphism and the set of étale $\Sp_{2d}$-torsors on $\Sch_{/X}$ up to isomorphism.
    \item There is a natural bijection of pointed sets between the set of Azumaya algebras of rank $4d^2$ equipped with a symplectic involution on $\Sch_{/X}$ up to isomorphism and the set of étale $\PGSp_{2d}$-torsors on $\Sch_{/X}$ up to isomorphism.
\end{enum}
\end{lemma}

\begin{proof} We first observe, that symplectic vector bundles are the same in the Zariski and in the étale topology.
This follows from the equivalence of categories \stackcite{03DX}, which is also used in the proof of Hilbert's Theorem 90 \stackcite{03P7} in the case of line bundles.

The bijection between étale symplectic vector bundles and étale $\Sp_{2d}$-torsors is now the standard one:
take an étale symplectic vector bundle $(\calV,\sigma)$ to the étale $\calI som$-sheaf
$$ \calI som((\calV,\sigma), (\calO_X^{2d}, \mathrm{std}))(U) \colonequals  \mathrm{Isom}((\calV,\sigma)|_U, (\calO_X^{2d}, \mathrm{std})|_U) $$
with $\Sp_{2d}$-action induced by the standard action on $\calO_X^{2d}$.
It follows directly from local triviality of $(\calV,\sigma)$, that $\calI som((\calV,\sigma), (\calO_X^{2d}, \mathrm{std}))$ is an $\Sp_{2d}$-torsor.

Take an étale $\Sp_{2d}$-torsor $\calT$ to the étale sheaf quotient $\calT \times^{\Sp_{2d}} \calO_X^{2d} \colonequals  (\calT \times^{\Sp_{2d}} \calO_X^{2d})/\Sp_{2d}$, which by local triviality of $\calT$ is again easily seen to be an étale symplectic vector bundle.

By the same argument using \Cref{standardformoverschemes}, we see that the groupoid of Azumaya algebras with symplectic involution is equivalent to the groupoid of étale $\PGSp_{2d}$-torsors.
\end{proof}

\begin{theorem}\label{equivofstk}
The canonical functors
$$ [\SpRep_{(R,*)}^{\square,2d}/\Sp_{2d}] \eqto \SpRep_{(R,*)}^{2d} \quad \text{ and } \quad [\SpRep_{(R,*)}^{\square,2d}/\PGSp_{2d}] \eqto \overline{\SpRep}_{(R,*)}^{2d} $$ 
are equivalences of étale stacks on $\Sch_{/Y}$. On the left hand sides we take the étale stack quotient.
\end{theorem}

The proof follows closely the proof of \cite[Theorem 1.4.1.4]{MR3167286}.
We remark, that the result is a version of \cite[Theorem 1.4.4.6]{MR3167286} for representations of algebras instead of groups.

\begin{proof}
By \stackcite{003Z}, it is enough to show that the functors induce equivalences of fiber groupoids.
For each $Y$-scheme $t \colon T \to Y$, the stack $[\SpRep_{(R,*)}^{\square,2d}/\Sp_{2d}]$ parametrizes pairs
\begin{equation*}
    (f \colon \calG \to T, ~\calG \to \SpRep_{(R,*)}^{\square,2d}) \in [\SpRep_{(R,*)}^{\square,2d}/\Sp_{2d}](T),
\end{equation*}
where $\calG$ is an étale $\Sp_{2d}$-torsor over $T$ and $\calG \to \SpRep_{(R,*)}^{\square,2d}$ is an $\Sp_{2d}$-equivariant map of $S$-schemes.

Using \Cref{SpVBandSptorsors}, we attach to $\calG$ a symplectic vector bundle $(V, b)$ on $T$.
Since $\calG(\calG)$ contains $\id_{\calG}$, $(V,b)$ is canonically trivialized over $\calG$. The composition
$$ f^* t^*R \to (M_{2d}(\calO_{\calG}), \mathrm j) \to \End_{\calO_{\calG}}(f^*V, \sigma_b) $$
can be descended to a map $t^*R \to \End_{\calO_{\calG}}(V, \sigma_b)$ using $\Sp_{2d}$-equivariance of $\calG \to \SpRep_{(R,*)}^{\square,2d}$.
The functor $\calG \mapsto (V,b)$ realizes the identification \Cref{SpVBandSptorsors} between symplectic vector bundles and $\Sp_{2d}$-torsors.
In particular, it induces an equivalence between the groupoid of symplectic vector bundles and the groupoid of $\Sp_{2d}$-torsors.

To show that the functor $[\SpRep_{(R,*)}^{\square,2d}/\Sp_{2d}](T) \to \SpRep_{(R,*)}^{2d}(T)$ is an equivalence, we give a functor in the other direction. It is then formal to verify that this realizes an equivalence of groupoids.

An object of $\SpRep_{(R,*)}^{2d}(T)$ is a triple $(V, b, \rho)$ as in \Cref{defstacks}. We define an $\Sp_{2d}$-torsor $\calG$ over $T$ by setting
$$ \calG(X) \colonequals  \mathrm{Isom}_{\calO_X}((x^*V, b), ~(\calO_X^{\oplus 2d}, b_{\mathrm{std}})) $$
for all $T$-schemes $x \colon X \to T$. Here $b_{\mathrm{std}}$ is the standard symplectic form and isomorphisms shall preserve the bilinear forms.
The functor $\calG$ is represetable by a flat scheme $f \colon \calG \to T$ of finite presentation over $T$ \cite[Theorem 3.24]{youcis}.
The identity map in $\calG(\calG)$ corresponds to an isomorphism $f^* V \eqto \calO_{\calG}^{\oplus 2d}$ compatible with $b$ and $b_{\mathrm{std}}$. The composition
$$ (f^*t^* R, *) \overset{\rho}{\to} (\End_{\calO_{\calG}}(f^*V), \sigma_b) \eqto (\End_{\calO_{\calG}}(\calO_{\calG}^{\oplus 2d}), \mathrm j) $$
defines a representation in $\SpRep_{(R,*)}^{\square,2d} (\calG)$, so we obtain a map $\calG \to \SpRep_{(R,*)}^{\square,2d}$. The latter is $\Sp_{2d}$-equivariant, for the action of $\Sp_{2d}$ realizes a change of basis. We have constructed an object of $[\SpRep_{(R,*)}^{\square,2d}/\Sp_{2d}](T)$.
The equivalence $[\SpRep_{(R,*)}^{\square,2d}/\PGSp_{2d}] \eqto \overline{\SpRep}_{(R,*)}^{2d}$ follows by an analogous argument.
\end{proof}

\section{Symplectic determinant laws}\label{Spdetlaw}
\subsection{Polynomial laws}

Chenevier's definition \cite[§1.5 Definition]{MR3444227} of determinant laws is based on the notion of polynomial laws.
The main references are \cite{Roby, BC, MR3444227, MR3167286}.
We give the basic definitions and explain how to introduce the structure of an algebra with involution on the graded pieces of a divided power algebra.
In this subsection, we suppose that $A$ is an arbitrary commutative ring.

\begin{definition}\label{defpolynomiallaws} Let $M$ and $N$ be any $A$-modules and let $R$ and $S$ be not necessarily commutative $A$-algebras.
\begin{enum}
    \item An \emph{$A$-polynomial law} $P \colon M \to N$ is a collection of maps $P_B \colon M \otimes_A B \to N \otimes_A B$ for each commutative $A$-algebra $B$, such that for each homomorphism $f \colon B \to B'$ of commutative $A$-algebras, the diagram
    \begin{align*}
        \xymatrix{
            M \otimes_A B \ar[d]^{\id \otimes f} \ar[r]^{P_B} & N \otimes_A B \ar[d]^{\id \otimes f} \\
            M \otimes_A B' \ar[r]^{P_{B'}} & N \otimes_A B'
        }
    \end{align*}
    commutes.
    In other words, an $A$-polynomial law is a natural transformation $\underline{M} \to \underline{N}$, where $\underline{M}(B) \colonequals  M \otimes_A B$ is the \emph{functor of points} of $M$. We denote the set of $A$-polynomial laws from $M$ to $N$ by $\calP_A(M,N)$.
    \item A polynomial law $P \colon M \to N$ is called \emph{homogeneous of degree $d \in \bbN_0$} or \emph{$d$-homogeneous}, if for all commutative $A$-algebras $B$, all $b \in B$ and all $x \in M \otimes_A B$ we have $P_B(bx) = b^d P_B(x)$. We denote the set of $d$-homogeneous $A$-polynomial laws from $M$ to $N$ by $\calP_A^d(M,N)$.
    \item A polynomial law $P \colon R \to S$ is called \emph{multiplicative}, if for all commutative $A$-algebras $B$, we have $P_B(1_{R \otimes_A B}) = 1_{S \otimes_A B}$ and for all $x, y \in R \otimes_A B$, we have $P_B(xy) = P_B(x)P_B(y)$. We denote the set of $d$-homogeneous multiplicative $A$-polynomial laws from $R$ to $S$ by $\calM_A^d(R,S)$.
    \item If $R$ and $S$ are equipped with $A$-linear involutions, both denoted by $*$, we say that a polynomial law $P\colon R\to S$ \emph{preserves the involution} if $P_B(x^*)=P_B(x)^*$ for every commutative $A$-algebra $B$, and all $x\in R\otimes_A B$. Here the $A$-linear involution on $R$ (resp. S) is canonically extended to a $B$-linear involution on $R \otimes_A B$ (resp. $S \otimes_A B$).
    \item A \emph{$d$-dimensional determinant law} on $R$ is a $d$-dimensional homogeneous multiplicative polynomial law $D \colon R \to A$.
    \item If $* \colon R \to R$ is an $A$-linear involution, a $d$-dimensional \emph{$*$-determinant law} on $(R,*)$ is a $d$-dimensional determinant law $D \colon R \to A$, which preserves the involution 
 (here $A$ is equipped with the trivial involution).
    \item Let $P\colon M\rightarrow N$ be an $A$-polynomial law. We define the \emph{kernel of $P$} as the $A$-submodule $\ker(P) \subseteq M$ consisting of the elements $m\in M$ such that for every commutative $A$-algebra $B$, $b\in B$ and $m'\in M\otimes_A B$, we have $P(m\otimes b + m')=P(m')$.
\end{enum}
\end{definition}

\begin{remark}\label{polylawsoverschemes}
    \Cref{defpolynomiallaws} (1) naturally extends to the case when $A = \calO_S$ is the structure sheaf of a scheme $S$, $M$ and $N$ are quasi-coherent $\calO_S$-modules and $B$ varies over the category of commutative quasi-coherent $\calO_S$-algebras. The properties defined in (2), (3) and (4) can be defined by evaluation on open subsets of $S$, e.g. $P : M \to N$ is \emph{homogeneous of degree $d$}, if for every open subset $U \subseteq S$, every commutative quasi-coherent $\calO_S$-algebra $B$, every $b \in B(U)$ and every $x \in (M \otimes_{\calO_S} B)(U)$, we have $P_B(U)(bx) = b^dP_B(U)(x)$, where $P_B(U) : (M \otimes_{\calO_S} B)(U) \to (N \otimes_{\calO_S} B)(U)$ and $b \in B(U)$ acts on $(M \otimes_{\calO_S} B)(U)$ through the sheafification map $M(U) \otimes_{\calO_S(U)} B(U) \to (M \otimes_{\calO_S} B)(U)$. Definitions (5) and (6) evidently extend to schemes. It follows, that these properties are local for the Zariski topology.
\end{remark}

\begin{definition}\label{charpfaffianalpha}
Let $R$ be an $A$-algebra and $P\colon R\rightarrow A$ be a $d$-homogeneous $A$-polynomial law. 
\begin{enum}
    \item For a commutative $A$-algebra $B$ and an element $r\in R\otimes_A B$, we define its \emph{characteristic polynomial} by
    $$ \chi^P(r,t)\colonequals P_{B[t]}(t-r)\in B[t]. $$
    We understand $\chi^P(\cdot, t)$ as an $A$-polynomial law $R \to A[t]$.
    \item For an integer $n\ge 1$, $r_1,\dots,r_n\in R$ and an ordered tuple of integers $\alpha=(\alpha_1,\dots,\alpha_n)$, we consider the function $\chi^{P}_\alpha\colon R^n\to R$ defined by
    $$ \chi^P(t_1r_1+\cdots+t_nr_n,t_1r_1+\cdots t_nr_n)= \sum_{\alpha} \chi^{P}_\alpha(r_1,\dots,r_n)t^{\alpha}\in R[t]$$
    where $t^\alpha=\prod_{i=1}^n t^{\alpha_i}$. Note that by homogeneity, we have $\chi^{P}_\alpha\equiv 0$ if $\sum_i \alpha_i\neq d$.
    \item We let $\CH(P)\subseteq R$ be the two-sided ideal of $R$ generated by the set  
    $$\left\{\chi^P_\alpha(r_1,\dots,r_n) \ \middle| \ n\in \mathbb{N}_{\ge 1}, \ r_1,\dots,r_n\in R,  \ \alpha=(\alpha_1,\dots,\alpha_n)\in \mathbb{N}^n \text{ with }\sum_{i}\alpha_i=d \right\}$$
    We refer to $\CH(P)$ as the \emph{Cayley-Hamilton ideal} of $P$.
\end{enum}
\end{definition}

\begin{lemma}
If $D\colon R\to A$ is a determinant law, then $\CH(D)\subseteq \ker(D)$.
\end{lemma}

\begin{proof} See \cite[Lemma 1.21]{MR3444227}.
\end{proof}

We will now describe a few representability results for polynomial laws, that are already explained in \cite{MR3444227}.
Recall that for any commutative ring $A$ and any $A$-module $M$, the \emph{divided power algebra} $\Gamma_A(M)$ is the commutative graded $A$-algebra generated by the symbols $m^{[i]}$ in degree $i$ for $m\in M$, $i\in \bbN_0$ subject to the following relations:
\begin{enum}
    \item $m^{[0]}=1$ for all $m\in M$.
    \item $(am)^{[i]}=a^i m^{[i]}$ for all $a\in A$, $m\in M$.
    \item $m^{[i]}m^{[j]}=\frac{(i+j)!}{i!j!}m^{[i+j]}$ for all $i,j\in \mathbb{N}_0$, $m\in M$.
    \item $(m+m')^{[i]}= \sum_{p+q=i}m^{[p]}{m'}^{[q]}$ for all $i\in \mathbb{N}_0$, $m,m'\in M$.
\end{enum}
We denote by $\Gamma_A^d(M)$ the $d$-th graded piece of $\Gamma_A(M)$. It represents the functor $\mathcal{P}^d_A(M,-) \colon \Mod_A \to \Set$ with the universal $d$-homogeneous polynomial law given by $P^{u}\colon M \rightarrow \Gamma^d_A(M), \ m \mapsto m^{[d]}$.
We have $\mathcal{P}^d_A(M,N) \cong \Hom_A(\Gamma_A^d(M), N)$ for any $A$-module $N$.

For an $A$-algebra $R$, we can equip $\Gamma_A^d(R)$ with the structure of an $A$-algebra as follows:
the map $R\oplus R \rightarrow R\otimes_A R, \ (r,r') \mapsto r\otimes r'$ is homogeneous of degree $2$ and is compatible with $-\otimes_A B$ for any $B\in \CAlg_A$. Thus it gives rise to a $2$-homogeneous $A$-polynomial law. Composing this map with the universal $d$-homogeneous polynomial law $R\otimes_A R \rightarrow \Gamma_A^d(R\otimes_A R)$, we obtain a $2d$-homogeneous polynomial law $R \oplus R \to \Gamma_A^d(R\otimes_A R)$. By the universal property of $\Gamma_A^{2d}(R\oplus R)$, we get a morphism of $A$-modules
$$ \eta\colon \Gamma_A^{2d}(R\oplus R) \to \Gamma_A^d(R\otimes_A R) $$
There is canonical isomorphism $\Gamma_A^{2d}(R \oplus R) \cong \bigoplus_{p+q=2d}\Gamma_A^p(R)\otimes_A \Gamma_A^q(R)$ (see \cite[§1.1.11]{MR3167286}) and $\eta$ kills $\Gamma_A^p(R)\otimes_A \Gamma_A^q(R)$ for $p\neq q$. From the multiplication map $\theta\colon R\otimes_A R \rightarrow R $, we obtain an $A$-linear map
$$ \Gamma_A^d(R)\otimes_A \Gamma_A^d(R) \xrightarrow{\eta} \Gamma_A^d(R\otimes_A R) \xrightarrow{\Gamma_A^d(\theta)} \Gamma_A^d(R) $$
defining the structure of an $A$-algebra on $\Gamma_A^d(R)$. In fact, we have a natural isomorphism $\mathcal{M}_A^d(R,S)\cong\Hom_{\Alg_A}(\Gamma_A^d(R),S)$ for any commutative $A$-algebra $S$.

When $R$ is equipped with an $A$-linear involution $*$, we want to equip $\Gamma_A^d(R)$ with an induced involution. For this, let $R^\text{op}$ be the opposite algebra of $R$. Then $*$ induces an isomorphism $R\cong R^\text{op}$. We define the $A$-linear maps $s\colon R\oplus R \to R\oplus R, \ (a,b)\mapsto (b,a)$ and $s'\colon R\otimes_A R \to R \otimes_A R, \ a\otimes b \to b\otimes a$, and we have a commutative diagram
\begin{center}
    \begin{tikzcd}
    \Gamma_A^d(R) \otimes \Gamma_A^d(R) \arrow[r, "\eta"] \arrow[d,"\Gamma^d_A(s)"] & \Gamma_A^d(R\otimes_A R) \arrow[d," \Gamma_A(s') "]  \arrow[r, " \Gamma^d(\theta^\text{op})"] & \Gamma_A^d(R) \arrow[d,"\text{id}"] 
    \\ \Gamma_A^d(R) \otimes \Gamma_A^d(R) \arrow[r, "\eta"]  & \Gamma_A^d(R\otimes_A R) \arrow[r, " \Gamma^d(\theta)"] & \Gamma_A^d(R) 
    \end{tikzcd}
\end{center}
which shows that we have a canonical isomorphism $\Gamma_A^d(R^\text{op})\cong \Gamma_A^d(R)^\text{op}$. Here $\theta^{\op} \colon R \otimes_A R \to R, a \otimes b \mapsto ba$ is the multiplication of $R^{\op}$ and $\Gamma_A^d(R) \otimes \Gamma_A^d(R)$ is identified with a submodule of $\Gamma_A^{2d}(R\oplus R)$.

\begin{definition}
Let $(R,*)$ be an $A$-algebra with involution. We define the involution $*$ on $\Gamma_A^d(R)$ by the isomorphism $$\Gamma_A^d(R)\xrightarrow{\Gamma_A^d(*)} \Gamma_A^d(R^\text{op})\cong \Gamma_A^d(R)^\text{op}$$
\end{definition}
Since the above diagram is compatible with tensoring with any $B\in \CAlg_A$, the isomorphism $\Gamma_A^d(R)\otimes_A B \cong \Gamma_B^d(R\otimes_A B)$ is compatible with the involution.

\subsection{Symplectic determinant laws}
\label{secdefsympldetlaw}

The definition of symplectic determinant laws is based on the following observation. Let $M \in M_{2d}(A)$ be a matrix with $M^{\mathsf j} = M$. Then we can write
$$ M = \begin{pmatrix} D & B \\ C & D^{\top} \end{pmatrix}, $$
where $D \in M_d(A)$ is arbitrary and $B, C \in M_d(A)$ are antisymmetric. The matrix
$$ MJ = \begin{pmatrix} -B & D \\ -D^{\top} & C \end{pmatrix} = JM^{\top} = -J^{\top}M^{\top} = -(MJ)^{\top}  $$
is alternating and therefore the Pfaffian $\Pf(MJ)$ is well defined. We have $\det(M) = \det(MJ) = \Pf(MJ)^2$.

\begin{definition}\label{defsympldetlaw} Let $(R,*)$ be an involutive $A$-algebra.
\begin{enum}
    \item A \emph{weak $2d$-dimensional symplectic determinant law} on $(R,*)$ is a pair $(D,P)$, where $D \colon R \rightarrow A$ is a $2d$-dimensional $*$-determinant law and $P\colon R^+\rightarrow A$ is a $d$-dimensional homogeneous polynomial law, such that $P^2=D|_{R^+}$ and $P(1)=1$.
    \item A \emph{$2d$-dimensional symplectic determinant law} on $(R,*)$ is a weak  $2d$-dimensional symplectic determinant law $(D,P)\colon (R,*)\to A$, that satisfies $\CH(P) \subseteq \ker(D)$.
\end{enum}
\end{definition}

\begin{remark} We believe that definitions (1) and (2) of \Cref{defsympldetlaw} are equivalent, but we are not able to prove it at the moment.
To prove that the condition $\CH(P) \subseteq \ker(D)$ holds for weak symplectic determinant laws, we would need to prove a version of Amitsur's formulae for the Pfaffian (compare with \cite[§1.10]{MR3444227}), or to prove a symplectic version of Vaccarino's comparison theorem between determinant laws and invariants of generic matrices (see \cite{vac08}, \cite{VACCARINO20091283}, and \cite{DeConPro17} for a detailed exposition of the proof). This would require knowledge of a second fundamental theorem of invariant theory over $\bbZ[\tfrac{1}{2}]$ for the action of the symplectic group by conjugation on tuples of matrices.
\end{remark}

\begin{remark}\label{stabilityCH}
    The conditions in (1) are formal and easily seen to be stable under base extension of polynomial laws.
    For (2), we have to see that the condition $\CH(P) \subseteq \ker(D)$ is stable under base extension. Actually, we have for every commutative $A$-algebra $B$, a surjection $\CH(P)\otimes_A B \twoheadrightarrow \CH(P\otimes_A B)$ (see \cite[Lemma 1.1.8.6]{CWE}).
\end{remark}

\begin{example}\label{reptodet}Let $\rho \colon (R,*) \to (M_{2d}(A), \mathsf j)$ be symplectic representation. For any commutative $A$-algebra $B$, we get a representation $\rho_B\colon (R\otimes_A B,*)\to(M_{2d}(B),\mathsf j)$. We define:
\begin{enum}
    \item $D_B \colon R \otimes_A B \to M_{2d}(B)$ by $D_B(r) \colonequals  \det (\rho_B(r))$,
    \item $P_B \colon R^+ \otimes_A B \to M_{2d}(B)$ by $P_B(r) \colonequals  \Pf (\rho_B(r)J)$.
\end{enum}
Then $(D,P)$ is a symplectic determinant law of $(R,*)$ over $A$. The fact that the condition $\CH(P)\subseteq \ker(D)$ holds follows from the Pfaffian Cayley-Hamilton theorem, which is the content of the following lemma. 
\end{example}
\begin{lemma}\label{charpf}
Suppose that $A$ is an arbitrary commutative ring. Let $M\in M_{2d}(A)$, such that $M^\mathrm{j}=M$. The \emph{Pfaffian characteristic polynomial} of $M$
$$ \Pf_M(\lambda)=\Pf((\lambda \cdot \id- M)J) \in A[\lambda] $$
is annihilated by $M$, i.e. $\Pf_M(M)=0$.
\end{lemma}

\begin{proof}
Using the ring homomorphism $\mathbb{Z}[x_{ij} \ | \ 1\le i,j\le 2d]\rightarrow A, ~x_{ij}\mapsto m_{ij}$, where $M=(m_{i,j})_{1\le i,j\le 2d}$ is as in the statement, and an injection $\mathbb{Z}[x_{ij} \ | \ 1\le i,j\le 2d] \hookrightarrow \bbC$,
we can assume that $A=\mathbb{C}$. By density, we can assume that $M$ is diagonalizable. Since $\Pf_M(M)$ is a polynomial in $M$, it is diagonalizable, and by the Cayley-Hamilton theorem we have that $\Pf_M(M)^2=0$. Therefore, $\Pf_M(M)=0$.
\end{proof}
\begin{example}
Let $\Gamma$ be a group. Let $\lambda\colon\Gamma\to A^\times$ be a group homomorphism, and let $R\colonequals A[\Gamma]$ be the group algebra of $\Gamma$ over $A$. We equip $R$ with the $A$-linear involution $*$ defined for $\gamma\in \Gamma$ as follows:
$$ \gamma^*=\lambda(\gamma)\gamma^{-1}. $$
A representation $\rho\colon \Gamma\to \GSp_{2d}(A)$ with similitude character $\lambda$ induces a symplectic representation $\rho\colon (R,*)\to (M_{2d}(A),\mathrm{j})$.
By \Cref{reptodet}, we get a symplectic determinant law $(D,P)\colon (R,*)\to A$.
\end{example}

This example leads us to consider the following definition.
\begin{definition}\label{gspdet}
    Let $\Gamma$ be a group. A \emph{$2d$-dimensional (weak) general symplectic determinant law} of $\Gamma$ over $A$ is a triple $(D,P,\lambda)$, where $\lambda : \Gamma \to A^{\times}$ is a character and $(D,P)$ is a $2d$-dimensional (weak) symplectic determinant law $(D,P) : A[\Gamma] \to A$, where $A[\Gamma]$ carries the $A$-linear involution given by $\gamma^* = \lambda(\gamma) \gamma^{-1}$ for $\gamma\in \Gamma$.
\end{definition}

The following Proposition ensures existence and uniqueness of the reduced Pfaffian on symplectic Azumaya algebras of constant rank over general base schemes.

\begin{proposition}\label{reducedpfaffian}
    Let $Y$ be a scheme with $2 \in \calO_Y(Y)^{\times}$ and let $(\calA, \sigma)$ be a symplectic Azumaya algebra over $Y$ of constant rank $4d^2$.
    Then there is a unique $\calO_S$-linear map $\Prd \colon \calA^+ \to \calO_Y$, such that for every morphism of schemes $T \to Y$ with $(f^* \calA, \sigma)$ isomorphic to $(M_{2d}(\calO_Y), \mathrm j)$, the induced map $(M_{2d}(\calO_Y), \mathrm j)^+ \cong f^* \calA^+ \overset{\Prd}{\to} \calO_T$ is equal to $M \mapsto \Pf(MJ)$. We call $\Prd$ the reduced Pfaffian of $(\calA, \sigma)$.
\end{proposition}

\begin{proof}
    From $(\calA, \sigma)$ we obtain a descent datum of quasi-coherent involutive $\calO_Y$-algebras on the small étale site of $Y$: by \Cref{standardformoverschemes} there is an étale covering $f \colon \tilde Y \to Y$ and an isomorphism $\psi\colon (M_{2d}(\calO_{\tilde S}), \mathrm j) \cong (f^* \calA, \sigma)$. So we obtain an isomorphism $\varphi \colon (M_{2d}(\calO_{\tilde Y \times_Y \tilde Y}), \mathrm j) \to (M_{2d}(\calO_{\tilde Y \times_Y \tilde Y}), \mathrm j)$ satisfying the axioms of a descent datum. To define $\Prd$ by effectivity of étale descent for maps of vector bundles, we have to a give a morphism of descent data from $\varphi$ to the descent datum of $\calO_Y$ induced by the covering $\tilde Y \to Y$. We define $\Prd_{\tilde Y} \colon M_{2d}(\calO_{\tilde Y})^+ \to \calO_{\tilde Y}$ by $M \mapsto \Pf(MJ)$ using the Leibniz formula for $\Pf$. To check that $\Prd_{\tilde Y}$ gives a morphism of descent data and to prove at the same time that $\Prd$ satisfies the desired property stated in the Proposition (which also implies uniqueness), we have to show, that for an arbitrary scheme $T$ and an arbitrary isomorphism of Azumaya algebras $\alpha \colon (M_{2d}(\calO_T), \mathrm j) \to (M_{2d}(\calO_T), \mathrm j)$ over $\calO_T$, we have $\Prd_T \circ \alpha = \Prd_T$ where $\Prd_T \colon M_{2d}(\calO_T)^+ \to \calO_T$ is defined by $M \mapsto \Pf(MJ)$. For that we may assume that $T = \Spec(A)$ for a local ring $A$. By the Skolem-Noether theorem (which holds over local rings, see e.g. \cite[Remark 3.4.19]{AGPR}), $\alpha$ is given by conjugation by a matrix in $\GSp_{2d}(A)$. To show $\Prd_T \circ \alpha = \Prd_T$ is now entirely formal, hence we may also assume that $A$ is an integral domain. We have $(\Prd_T \circ \alpha)^2 = \det \circ \alpha = \det = \Prd_T^2$, hence $\Prd_T \circ \alpha = \pm \Prd_T$. It follows, that $\Prd_T \circ \alpha = \Prd_T$, since $\GSp_{2d}$ is a connected group scheme over $A$.
\end{proof}

\begin{corollary}
    Let $(\calA, \sigma)$ be a symplectic Azumaya algebra over $A$ of constant rank $4d^2$.
    For every commutative $A$-algebra $B$, we define $\det_B : \calA \otimes_{A} B \to B$ to be the reduced determinant of the Azumaya algebra $\calA \otimes_{A} B$ and we define $\pf_B : \calA^+ \otimes_{A} B \to B$ to be its reduced Pfaffian. Then $(\det, \pf) : \calA \to A$ is a symplectic determinant law.
\end{corollary}

\begin{proof}
    All desired properties follow from \Cref{reducedpfaffian} and étale descent.
\end{proof}

For an $A$-algebra with involution $(R,*)$, and a weak symplectic determinant $(D,P)\colon (R,*)\to A$, we introduce the polynomial laws
\begin{align*}
    \Lambda_i^D&\colon R\to A \quad \text{ for } 1\le i \le 2d,
    \\ \mathcal{T}_j^P&\colon R^+\to A \quad \text{ for }1\le j \le d,
\end{align*}
defined for any $A$-algebra $B$ by the formulas
\begin{align*}
    \chi^D(r,t)\colonequals D_B(t-r)= \sum_{i=0}^{2d}(-1)^i\Lambda_{i,B}^D(r)t^{2d-i}, \quad r\in R\otimes_A B,
    \\ \chi^P(r,t)\colonequals P_B(t-r)= \sum_{i=0}^{d}(-1)^i\mathcal{T}_{i,B}^P(r)t^{d-i}, \quad r\in R^+\otimes_A B.
\end{align*}
The following result explains how the characteristic polynomial of $P$ is related to the characteristic polynomial of $D$ when restricted to symmetric elements.
In particular, we see that a weak symplectic determinant law $(D,P)$ is determined by $D$.
\begin{proposition}\label{pfaffianunique} If $D \colon R \to A$ and $P, P' \colon R^+ \to A$ are polynomial laws, such that $(D,P)$ and $(D,P')$ are weak symplectic determinant laws, then $P = P'$. Further, we have the recursion formula
$$ \Lambda_{i}^D|_{R^+} = \sum_{j=0}^i \calT^P_{j} \calT^P_{i-j} $$
for $1 \leq i \leq 2d$ with $\calT^P_i = 0$ for $i > d$.
\end{proposition}

\begin{proof} Since $1 = P(1) = P'(1)$, we have $\calT^P_0 = \calT^{P'}_0$. By comparing the coefficients $\calT_i^P$ and $\calT_i^{P'}$ and the coefficients $\Lambda_i^D$ using $\chi^D(\cdot,t)|_{R^+} = \chi^P(\cdot, t)^2 = \chi^{P'}(\cdot, t)^2$ we obtain
$$ \Lambda_{i}^D|_{R^+} = \sum_{j=0}^i \calT_{j}^P \calT_{i-j}^P = \sum_{j=0}^i \calT^{P'}_{j} \calT^{P'}_{i-j} $$
for $1 \leq i \leq 2d$ and $\calT^P_{d} = P$ and $\calT^{P'}_{d} = P'$. For $i=0$, we know, that $1 = \Lambda_0 = {\calT^P_0}^2 = {\calT^{P'}_0}^2$.

By induction over the above equations and using $2 \in A^{\times}$, we obtain $\calT_i' = \calT_i$ for all $0 \leq i \leq d$, in particular $P' = P$.
\end{proof}

Taking $r=1$, we see that $\calT^P_i(1) = \pm \binom{d}{i}$ for $0 \leq i \leq d$ and the assumption $P(1) = \calT^P_d(1) = 1$ implies that $\calT^P_i(1) = \binom{d}{i}$ by downward induction.

\begin{example}\label{exd=2}
 For $d=4$, one finds that
$$ P= \frac{1}{2}\Lambda_4^D-\frac{1}{4}\Lambda_1^D\Lambda_3^D+\frac{1}{16}(\Lambda_1^D)^2\Lambda_2^D+\frac{1}{8}(\Lambda_2^D)^2-\frac{3}{128}(\Lambda_1^D)^4. $$
In particular, we see that the recursion formulas of \Cref{pfaffianunique} provide us with a way to define $P$ as a $d$-homogeneous $A$-polynomial law on the entire algebra $R$ for every $2d$-dimensional determinant law $D$ when $2 \in A^{\times}$.
\end{example}

\begin{example}\label{SL2example}
Let $\Gamma$ be a group. By \cite[Lemma 1.9]{MR3444227}, the datum of a $2$-dimensional determinant law $D \colon A[\Gamma]\to A$ is equivalent to the datum of a pair of functions $(d,t)\colon\Gamma\to A$ such that $d \colon \Gamma\to A^\times$ is a group homomorphism, and $t$ is a function satisfying $t(1)=2$ and for all $\gamma,\gamma'\in \Gamma$ the following two equations:
\begin{itemize}
    \item[(a)] $t(\gamma\gamma')=t(\gamma'\gamma)$,
    \item[(b)] $d(\gamma)t(\gamma^{-1}\gamma')-t(\gamma)t(\gamma')+t(\gamma\gamma')=0$.
\end{itemize}
Here the functions $t$ and $d$ are obtained from the determinant law $D$ by considering the characteristic polynomial $\chi^D(x,\gamma)=x^2-t(\gamma)x+d(\gamma) \in A[x]$ for all $\gamma \in \Gamma$. In particular, they are defined as functions $t,d \colon A[\Gamma] \to A$, and we have the usual polarization formula
\begin{align}
    d(r)=\frac{t(r)^2-t(r^2)}{2},\label{polDT}
\end{align}
for all $r \in A[\Gamma]$.

We are interested in the case, that $D$ is a symplectic determinant law in the sense of \Cref{defsympldetlaw}.
Note, that this means that $D$ is a determinant law for $\Sp_2 = \SL_2$.
So we require that there exists a $1$-homogeneous $A$-polynomial law $P \colon A[\Gamma]^+ \to A$ with $P^2 = D|_{A[\Gamma]^+}$ and $P(1)=1$. So let us assume such a $P$ exists.
By \cite[Example 1.2 (i)]{MR3444227} $P$ is determined by the $A$-linear map $P_A \colon A[\Gamma]^+ \to A$.
By \Cref{pfaffianunique}, we have $P_A(r)=\frac{1}{2}t(r)$ for all $r\in A[\Gamma]^+$.
Evaluating the equation $P_A^2 = d|_{A[\Gamma]^+}$ at $\gamma+\gamma^{-1}$ for some $\gamma \in \Gamma$, we thus obtain
\begin{align}
    \frac{1}{4}t(\gamma+\gamma^{-1})^2=d(\gamma+\gamma^{-1}). \label{TDeqn}
\end{align} 
\Cref{polDT} gives
\begin{align}
    d(\gamma+\gamma^{-1})=d(\gamma)+d(\gamma^{-1})+t(\gamma)t(\gamma^{-1})-2. \label{deqn}
\end{align}
Combining \Cref{deqn} with \Cref{TDeqn} we get
$$ \frac{1}{4}t(\gamma+\gamma^{-1})^2 = d(\gamma)+d(\gamma^{-1})+t(\gamma)t(\gamma^{-1})-2,$$
and thus
$$ t(\gamma)^2+2t(\gamma)t(\gamma^{-1})+t(\gamma^{-1})^2-2t(\gamma^2)-2t(\gamma^{-2})-8=0. $$
In \Cref{defsympldetlaw}, we also require that the determinant law $D$ is invariant for the $A$-linear involution on $A[\Gamma]$ extending inversion $\Gamma \to \Gamma, ~\gamma \mapsto \gamma^{-1}$. This implies that the functions $t,d$ are invariant under the inversion map. So we have
$$ 4t(\gamma)^2-4t(\gamma^{2})-8=0,  $$
and hence $d(\gamma)=1$ by \Cref{polDT}. 

On another note, for $\gamma,\gamma'\in \Gamma$  we have
$$ t(\gamma'\cdot (\gamma+\gamma^{-1}-t(\gamma)))=t(\gamma'\gamma)+t(\gamma'\gamma^{-1})-t(\gamma)t(\gamma')=0, $$
by (b) since $d(\gamma)=1$. If $A$ is an infinite domain, this is saying that $\CH(P)\subseteq\ker(D)$.
\end{example}

\begin{lemma}\label{propertyofP}
Let $(R,*)$ be an $A$-algebra with involution equipped with a weak symplectic determinant $(D,P)$. Then for every commutative $A$-algebra $B$, any $x\in R\otimes_A B$, and any $y\in R^+\otimes_A B$ such that  $P_B(y)$ is a non-zero divisor,  we have that
$$ P_B(xyx^*)=D_B(x)P_B(y). $$
\end{lemma}

\begin{proof}
For a fixed $y$ as in the statement, consider the polynomial laws $Q_1\colon R\otimes_A B\to B, \ x\mapsto P(xyx^*)$ and $Q_2\colon  R\otimes_A B\to B, \ x\mapsto D(x)P(y)$. Then, it is clear that $Q_1^2=Q_2^2$, and so evaluating at the formal power series ring $B[[t]]$, we have  $$(Q_{1}(tx-t+1)-Q_{2}(tx-t+1))(Q_{1}(tx-t+1)+Q_{2}(tx-t+1))=0.$$
The evaluation of the second summand at $t=0$ gives $2P_B(y)$, thus $Q_{1}(tx-t+1)+Q_{2}(tx-t+1)$ is a non zero divisor. And so, $Q_{1}(tx-t+1)=Q_{2}(tx-t+1)$ whose evaluation at $t=1$ gives the result.
\end{proof}

If $x,y\in R^+$, then we do not generally have that $xy\in R^+$. In fact, this happens if and only if $x$ and $y$ commute. In this case, $P$ turns out to be multiplicative as recorded by the following Lemma (which was discovered in \cite[Proposition 3.1]{THC&NBC}). We will make use of it in \Cref{SCHidempotent}.
\begin{lemma}\label{Pcommutes} Let $(R,*)$ be an $A$-algebra with involution equipped with a weak  $2d$-dimentional symplectic determinant law $(D,P)$. Then, for any commutative $A$-algebra $B$ any commuting elements $x,y\in R^+\otimes_A B$, we have that $xy\in R^+\otimes_A B$ and
$$ P_B(xy)=P_B(x)P_B(y). $$
\end{lemma}
\begin{proof}
The fact that $xy\in R^+\otimes_A B$ is immediate. Now we introduce the commuting elements $1+t_1x,1+t_2y\in R^+\otimes_A B[t_1,t_2]$, and the polynomials
$$ Q_x=P_B(1+t_1x), \quad Q_y=P_B(1+t_2y), \quad Q_{xy}=P_B((1+t_1x)(1+t_2y))$$
in $B[t_1,t_2]$. The $Q_x$ is a polynomial in $t_1$ of degree at most $d$ whose coefficient of $t^d$ is $P_B(x)$. Similarly $Q_y$ is a polynomial in $t_2$ of degree at most $d$ whose coefficient of $t^d$ is $P_B(y)$, and $Q_{xy}$ is a polynomial in $t_1,t_2$ whose coefficient of $t_1^dt_2^d$ is $P_{B}(xy)$. Thus, to prove the statement, it suffices to show the equality $Q_xQ_y=Q_{xy}$, which can be checked inside the power series ring $B[[t_1,t_2]]$. 
\\ Note that for every power series $g\in B[[t_1,t_2]]^\times$ with $g(0,0)\in B^\times$ and every square root $f_0\in B^\times$ of $g(0,0)$, there exists a unique power series $f\in B[[t_1,t_2]]^\times$ with $f(0,0)=f_0$ such that $f^2=g$. This can be seen by considering the power series expansion of the square root function at $1$. Using this fact, the equality $Q_{xy}^2=Q_{x}^2Q_{y}^2$ (coming from multiplicativity of $D$), and $Q_{xy}(0,0)=Q_{y}(0,0)=Q_{x}(0,0)$, we ge that $Q_xQ_y=Q_{xy}$ as desired.
\end{proof}

\begin{lemma}\label{strongHomTheorem} Let $(R,*)$ be an $A$-algebra with involution equipped with a (weak) symplectic determinant law $(D,P)$. Then $\ker(D)$ is stable under $*$ and $\ker(D) \cap R^+ \subseteq \ker(P)$. In particular for every $*$-ideal $I \subseteq \ker(D)$, $(D,P)$ factors uniquely through a (weak) symplectic determinant law $(\overline{D},\overline{P})\colon (R/I,*) \to A$.
\end{lemma}

\begin{proof} Since $D$ is $*$-invariant, it follows that $\ker(D)$ is a $*$-ideal. 
Using  \cite[Lemma 1.19]{MR3444227} we have that
$$ \ker(D)= \left\{r\in R \ | \ \forall B \in \CAlg_A, \ \forall m\in R\otimes_A B, \ \forall i \ge 1, \ \Lambda_i(rm)=0 \right\} $$
By \Cref{pfaffianunique}, we know that $P$ can be expressed as a polynomial in the $\Lambda_i$, thus to show that $r\in\ker(D)\cap R^+$ is in $\ker(P)$, it suffices to show that $\Lambda_i(r\otimes b + m) = \Lambda_i(m)$ for all commutative $A$-algebras $B$, $b\in B$ and $m\in R^+\otimes_A B$. But this follows from the definition of the $\Lambda_i$ and the definition of $\ker(D)$.

Since $2 \in R^{\times}$, we have a surjection $R^+ \twoheadrightarrow (R/I)^+$ and $(R/I)^+$ is identified with $R^+/(I \cap R^+)$. Since $I \cap R^+ \subseteq \ker(D) \cap R^+ \subseteq \ker(P)$, $P$ descends to a well-defined $A$-polynomial law $\overline P \colon (R/I)^+ \to A$ satisfying the desired properties.
\end{proof}
We can define direct sums of (weak) symplectic determinant laws.
On the level of representations, it corresponds to the orthogonal direct sum of symplectic spaces carrying an equivariant group action.
We will use the direct sum to state the structure theorem \Cref{symplecticreconstructionthmoverfields} for weak symplectic determinant laws over fields.

\begin{lemma} Let $A$ be a commutative ring, let $(R,*)$ be an involutive $A$-algebra and let $(D_1,P_1)$ and $(D_2, P_2)$ be (weak) symplectic determinant laws of $(R,*)$ over $A$ of dimension $2d_1$ and $2d_2$ respectively. Then $(D_1D_2, P_1P_2)$ is a (weak) symplectic determinant law of dimension $2(d_1+d_2)$.
\end{lemma}

\begin{proof} As in \cite[§2.1]{MR3444227}, $D_1D_2$ is a determinant law of dimension $2(d_1+d_2)$ and one checks, that it is a $*$-determinant law. Similarly $P_1P_2 \colon R^+ \to A$ is homogeneous of degree $d_1+d_2$. Further $(P_1P_2)^2 = D_1|_{R^+}D_2|_{R^+}$ and $(P_1P_2)(1) = 1$. This proves the claim for weak symplectic determinant laws.

Now suppose that $\CH(P_i) \subseteq \ker(D_i)$. We will show that $\CH(P_1P_2) \subseteq \ker(D_1D_2)$. Let $P \colonequals  P_1P_2$ and $D \colonequals  D_1D_2$. Recall \Cref{charpfaffianalpha} of the functions $\chi_{\alpha}^{P_i}$, where $\alpha \in \bbN_0^n$ with $\sum\nolimits_{j=1}^n \alpha_j = d_i$. For $r_1, \dots, r_n \in R$, the equation
$$ \chi^P(r_1t_1 + \dots + r_nt_n) = \chi^{P_1}(r_1t_1 + \dots + r_nt_n) \chi^{P_2}(r_1t_1 + \dots + r_nt_n) $$
in $R[t_1, \dots, t_n]$ implies
$$ \chi_{\alpha}^P(r_1, \dots, r_n) = \sum_{\alpha' + \alpha'' = \alpha} \chi_{\alpha'}^{P_1}(r_1, \dots, r_n) \chi_{\alpha''}^{P_2}(r_1, \dots, r_n) $$
by comparing the coefficients of $t^{\alpha}$. To check that $D(1 + \chi_{\alpha}^P(r_1, \dots, r_n)r) = 1$
for all $r \in R\otimes_A B$, it suffices to check that $D_i(1 + \chi_{\alpha}^P(r_1, \dots, r_n)r) = 1$ for all $r \in R\otimes_A B$. This follows from \cite[Lemma 1.19]{MR3444227}, since
$$ \sum_{\alpha' + \alpha'' = \alpha} \chi_{\alpha'}^{P_1}(r_1, \dots, r_n) \chi_{\alpha''}^{P_2}(r_1, \dots, r_n)  \in \ker(D_1)\cdot\ker(D_2)\subset\ker(D_1) \cap \ker(D_2). $$
\end{proof}

\begin{remark} Let $A$ be a commutative ring and $(R,*)$ be an involutive $A$-algebra. If $(D_1,P_1)$ and $(D_2,P_2)$ are the symplectic determinant laws attached respectively to the symplectic representations $\rho_1 \colon (R,*) \to (M_{2d_1}(A), \mathsf j)$ and $\rho_2 \colon (R,*) \to (M_{2d_2}(A), \mathsf j)$ then $(D_1D_2,P_1P_2)$ is the symplectic determinant law attached to $\rho_1 \oplus \rho_2$.
\end{remark}

We now introduce the space of (weak) symplectic determinant laws and show its representability.

\begin{proposition}\label{representability}
Let $(R,*)$ be an $A$-algebra with involution and assume $2 \in A^{\times}$. Then the functor
\begin{align*}
    \wSpDet_{(R,*)}^{2d}\colon \CAlg_A &\to \Set
    \\ B &\mapsto \{\text{weak symplectic determinant laws } (D,P)\colon R \to B\}
\end{align*}
resp.,
\begin{align*}
    \SpDet_{(R,*)}^{2d}\colon \CAlg_A &\to \Set
    \\ B &\mapsto \{\text{symplectic determinant laws } (D,P)\colon R \to B\}
\end{align*}
is represented by a commutative $A$-algebra denoted by $A[\wSpDet_{(R,*)}^{2d}]$ (resp. $A[\SpDet_{(R,*)}^{2d}]$) .
If $R$ is a finitely generated $A$-algebra, then $A[\wSpDet_{(R,*)}^{2d}]$ (resp. $A[\SpDet_{(R,*)}^{2d}]$) is a finitely generated $A$-algebra.
\end{proposition}

\begin{proof}
Let $I$ be the ideal of $\Sym_A(\Gamma_A^d(R^+))$ generated by the element $[1]^d-1$. Then the ring $\Sym_A(\Gamma_A^d(R^+))/I$ represents the functor which associates to a commutative $A$-algebra $B$ the set of homogeneous polynomial laws $P$ of degree $d$ such that $P(1)=1$. 
Using the isomorphism
\begin{align*}
    \Gamma_A(R^+\times R^+) &\xrightarrow{\sim} \Gamma_A(R^+)\otimes_A \Gamma_A(R^+)
    \\ [(r_1,r_2)]^i &\mapsto \sum_{p+q=i} [r_1]^p\otimes [r_2]^q
\end{align*}
we get a morphism of $A$-modules
$$ \widetilde{\varphi}\colon \Gamma_A^{2d}(R^+)\xrightarrow{\Gamma_A(\Delta)} \Gamma_A(R^+\times R^+)\twoheadrightarrow \Gamma_A^d(R^+)\otimes_A \Gamma_A^d(R^+)\rightarrow \Sym_A(\Gamma_A^{d}(R^+))/I. $$
For $[r_1]^{i_1}\cdots [r_m]^{i_m}\in\Gamma_A^{2d}(R^+)$ with  $i_1+\cdots+i_m=d$, it is given by 
$$\widetilde{\varphi}([r_1]^{i_1}\cdots [r_m]^{i_m})=\sum_{p_1,q_1,\dots,p_m,q_m} ([r_1]^{p_1}\cdots[r_m]^{p_m})\odot ([r_1]^{q_1}\cdots [r_m]^{q_m}),$$
where the sum runs over the integers $p_j,q_j$ satisfying $p_j+q_j=i_j$ and $p_1+\cdots+p_m=q_1+\cdots+q_m=d$. Therefore, we get a morphism of $A$-algebras $\varphi\colon \Sym_A(\Gamma_A^{2d}(R^+))\to \Sym_A(\Gamma_A^{d}(R^+))/I$.
\\ Aside from that, the canonical map $\Gamma_A^{2d}(R^+)\to \Gamma_A^{2d}(R)$ induces a morphism of commutative $A$-algebras $\Sym_A(\Gamma_A^{2d}(R^+))\to \Gamma_A^{2d}(R)^{\text{ab}}$. Then we can take the representing ring for weak symplectic determinant laws to be
$$ A[\wSpDet_{(R,*)}^{2d}]=(\Gamma_A^{2d}(R)^{\text{ab}}/*)\otimes_{\Sym_A(\Gamma_A^{2d}(R^+)),\varphi} \Sym_A(\Gamma_A^{d}(R^+))/I.$$
Here $\Gamma_A^{2d}(R)^{\text{ab}}/*$ is the quotient of $\Gamma_A^{2d}(R)^{\text{ab}}$ by the ideal generated by $\gamma-\gamma^*$ for $ \gamma\in \Gamma_A^{2d}(R)^{\text{ab}}$. 
\\ Now let $(D^u,P^u)$ be the universal weak symplectic determinant law on $(R,*)$. The universal ring $A[\SpDet_{(R,*)}^{2d}]$ is the quotient of $A[\wSpDet_{(R,*)}^{2d}]$ by $\chi_\alpha^{D^u}(rr_1,\dots,rr_n)$ for every $r\in \CH(P^u)$, every $r_1,\dots,r_n\in R$, and every ordered tuple of integers $\alpha$ (see \Cref{charpfaffianalpha}).
\end{proof}

\begin{remark}\label{representabilityGSp} Let $\Gamma$ be a group. Then the functor
\begin{align*}
    \wGSpDet_{\Gamma}^{2d}\colon \CAlg_A &\to \Set
    \\ B &\mapsto \left\{\text{weak general symplectic determinant laws } (D,P,\lambda)\colon A[\Gamma] \to B \right\}
\end{align*}
resp.,
\begin{align*}
    \GSpDet_{\Gamma}^{2d}\colon \CAlg_A &\to \Set
    \\ B &\mapsto \left\{\text{general symplectic determinant laws } (D,P, \lambda)\colon A[\Gamma] \to B \right\}
\end{align*}
is represented by a commutative $A$-algebra denoted by $A[\wGSpDet_{\Gamma}^{2d}]$ (resp. $A[\GSpDet_{\Gamma}^{2d}]$).
\\ Indeed, the functor $\CAlg_A \to \Set, ~B \mapsto \Hom(\Gamma, B^{\times})$ is representable by $A[\Gamma^{\ab}]$ and the universal element is the universal character $\lambda^u : \Gamma \to A[\Gamma^{\ab}]^{\times}$. The map $\wGSpDet_{\Gamma}^{2d} \to \Spec(A[\Gamma^{\ab}]), ~(D,P,\lambda) \mapsto \lambda$ is relatively representable by \Cref{representability}. This implies that $\wGSpDet_{\Gamma}^{2d}$ is representable by the $A$-algebra $A[\wSpDet_{A[\Gamma^{\ab}][\Gamma]}^{2d}]$, where $A[\Gamma^{\ab}][\Gamma]$ carries the involution defined by $\gamma^* := \lambda^u(\gamma) \gamma^{-1}$. The same argument applies to the functor $\GSpDet_{\Gamma}^{2d}$.
\end{remark}

We end this subsection by introducing the following terminology, which will be useful in \Cref{CH0} and \Cref{subsecisommultfree}. 
\begin{definition}\label{spdetalg}
A symplectic determinant $A$-algebra of dimension $2d$ is a tuple $(B,R,*,D,P)$, where $B$ is a commutative $A$-algebra, $(R,*)$ is $B$-algebra with involution, and $(D,P)\colon R\to B$ is symplectic determinant law. We equip $R$ with the trace map
\begin{align*}
    \tr\colonequals \Lambda^{D}_{1,B}\colon R\longrightarrow B.
\end{align*}
A morphism between two symplectic determinant $A$-algebras $(B,R,*,D,P)$ and $(B',R',*,D',P')$ is the data of a morphism of commutative $A$-algebras $f\colon B\to B'$, and a morphism of involutive $B$-algebras $g\colon R \to R'$, such that $R'$ is seen as a $B$-algebra via $f$, and  $f\circ D=D'\circ g$ (from which it follows that $f\circ P=P'\circ g$).
\end{definition}
Given a symplectic determinant $A$-algebra $(R,*,D,P)$ with $(D,P)$ taking values in a commutative $A$-algebra $B$, we can consider the functor $\SpRep_{(R,*,D,P)}^{\square,2d}\colon \CAlg_A \rightarrow \Set$ given by
\begin{align*}
    \SpRep_{(R,*,D,P)}^{\square}(C)\colonequals \left\{ \begin{array}{l}
         \text{symplectic representations } \rho\colon (R,*) \to (M_{2d}(B\otimes_A C), \mathrm{j}) \\
         \text{such that } (\det\circ \rho, \Pf\circ(\rho\cdot J))=(D,P)
    \end{array}  \right\}.
\end{align*}
It is representable by a commutative $A$-algebra $A[\SpRep_{(R,*,D,P)}^{\square,2d}]$. In fact, we have that $$\SpRep_{(R,*,D,P)}^{\square}=\SpRep_{(R,*)}^{\square,2d}\times_{\SpDet_{(R,*)}^{2d}}\Spec(B),$$
where the map $\Spec(B)\to \SpDet_{(R,*)}^{2d}$ is the one induced by $(D,P)$. Similarly to \Cref{subsymprep}, we can define an action of $\Sp_{2d,A}$ on $M_{2d}(A[\SpRep_{(R,*,D,P)}^{\square}])$, and we the same statement as in \Cref{imageofR1}.
\begin{lemma}\label{imageofR2}
The image of $R$ by universal representation
$$ \rho^{u} \colon (R,*)\longrightarrow (M_{2d}(A[\SpRep_{(R,*,D,P)}^{\square}),\mathrm{j})$$
lies inside $M_{2d}(A[\SpRep_{(R,*,D,P)}^{\square}])^{\Sp_{2d}}$.
\end{lemma}
We can also define the algebraic stack $\Rep_{(R,*,D,P)} : \Sch_{/A}^{\op} \to \Gpd$, such that for a $\Spec(A)$-scheme $X$, the groupoid $\Rep_{(R,*,D,P)}(X)$ is the subgroupoid of $ \Rep_{(R,*)}^{2d}(X\times_{\Spec(A)}\Spec(B))$ consisting of tuples $(V,b,\rho)$ that satisfy $(\det \circ \rho, \Pf \circ (\rho\cdot J)) = (D, P)$. It is the fibre product $\Rep_{(R,*)}^{2d} \times_{\SpDet_{(R,*)}^{2d}} \Spec(A)$, where the map $\Spec(A) \to \SpDet_{(R,*)}^{2d}$ is determined by $(D,P)$. 

\subsection{Symplectic determinant laws over fields}
\label{secsympldetlawfields}

Fix a field $k$ with $2 \in k^{\times}$ and an algebraic closure $\overline k$ of $k$ throughout \Cref{secsympldetlawfields}.
The goal of this section is to give a precise structure theorem for symplectic determinant laws over $k$.
This is the content of \Cref{symplecticreconstructionthmoverfields}, which is the symplectic analog of \cite[Thm. 2.16]{MR3444227}.
An important ingredient in the $\GL_n$-case is the Artin-Wedderburn theorem. Here we need a version of the Artin-Wedderburn theorem for semisimple $k$-algebras with involution.

\begin{proposition}\label{invAW} Let $(R,*)$ be a semisimple $k$-algebra equipped with a $k$-linear involution, such that every simple factor of $R$ is finite-dimensional over its center. Then $(R,*)$ is isomorphic as an involutive $k$-algebra to a product
$$ (R,*) \cong \prod_{i=1}^t (R_i, \sigma_i) $$
for some $t \in \bbN_{\ge 1}$, where the involutive rings $ (R_i, \sigma_i)$ have one of the following three forms:
\begin{enum}
    \item[(I)] $R_i$ is a central simple algebra over a field $K_i$, and $\sigma_i$ is a $K_i$-linear symplectic involution.
    \item[(II)] $R_i$ is a central simple algebra over a field $K_i$, and $\sigma_i$ is a $K_i$-linear orthogonal involution.
    \item[(IIIa)] $R_i$ is a central simple algebra over a field $L_i$, and $\sigma_i$ is a $K_i$-linear involution of the second kind for some index $2$ subfield $K_i$ of $L_i$ with $L_i/K_i$ separable.
    \item[(IIIb)] $R_i = T_i \times T_i^{\op}$ for some central simple algebra $T_i$ over a field $K_i$, and $\sigma_i(a,b^{\op})=(b,a^{\op})$ for $a,b\in T_i$.
\end{enum}
\end{proposition}

\begin{proof} Applying the Artin-Wedderburn theorem to $R$, we see that $R\cong\prod_{i=1}^s R_i'$, where $R'_i$ is a central simple algebra over a field $K_i$. 
This product decomposition corresponds to a unique set of orthogonal central primitive idempotents $e_1, \dots, e_s \in R$ with $e_1 + \dots + e_s = 1$. The involution $*$ defines a bijection $* \colon \{e_1, \dots, e_s\} \to \{e_1, \dots, e_s\}$. This defines a partition $\{1, \dots, s\}=I_0\sqcup I_1\sqcup I_2$ such that $e_i^* = e_i$ if and only if $i\in I_0$, and $e_i^*= e_{i'}$ for $i\in I_1$ if and only if  $i'\in I_2$. 
Since $e_iR$ is $*$-stable for all $i\in I_0$, and $(e_i+e_i^*)R$ is $*$-stable for all $i\in I_1$, we obtain $*$-stable $k$-algebras $R_i$ with $R_i = R'_i$ if $i\in I_0$ and $R_i = R_i' \times R_{i'}'$ if $i\in I_1$ with $e_i^*=e_{i'}$. In the latter case, the involution $*$ induces an isomorphism $R_{i'}'\cong (R_i')^{\mathrm{op}}$. This provides us with the desired decomposition. The rest of the proposition is deduced from the discussion in \Cref{secazumaya} (see also \cite[Chapter I, Proposition 2.20]{Inv}) for the cases (I-II-IIIa), and the case (IIIb) is immediate.
\end{proof}

\begin{example}\label{exampleforms}
Let $K/k$ be an algebraic extension and let $k^s\subseteq K$ be the maximal separable extension of $k$ inside $K$. We assume that $f\colonequals [k^s:k]$ is finite. 
If $\chara(k)=p>0$, we assume there is an integer $q\in p^ \mathbb{N}$ such that $K^q\subseteq k^s$. We take $q$ minimal with this property. If $p=0$ we take $q=1$.

Let $(R,\sigma)$ be a $K$-algebra with involution and let $(D,P)$ be a weak symplectic determinant law of $(R,\sigma)$ over $k$. We consider the following cases:

\begin{enum}
    \item[(I)] $(R,\sigma)$ is a central simple algebra over $K$ with a symplectic involution.
    Then $(D,P)$ is power of weak symplectic determinant law given by the pair
    \begin{align*}
        \Norm_{k^s/k}\circ F^q\circ \Nrd_R &\colon R\to k
        \\ \Norm_{k^s/k}\circ F^q\circ \Prd_R & \colon R^+\to k
    \end{align*}
    This follows from \cite[Lemma 2.17]{MR3444227} and the existence and uniqueness \Cref{pfaffianunique} of the Pfaffian. 
    \item[(II)] $(R,\sigma)$ is a central simple algebra over $K$ with an orthogonal involution. By \cite[Lemma 2.17]{MR3444227}, we know that $D$ is a power $m \geq 0$ of
    \begin{align*}
    \Norm_{k^s/k}\circ F^q\circ \Nrd_R &\colon R\to k
    \end{align*}
    Let $\tilde k$ be a separably closed extension of $k$, then
    \begin{align*}
        K\otimes_{k}\tilde k\cong K\otimes_{k^s}k^s \otimes_k \tilde k\cong K \otimes_{k^s}\prod \tilde k\cong \prod K\otimes_{k^s}\tilde k
    \end{align*}
    Since the extension $K/k^s$ is purely inseparable, $K\otimes_{k^s}\tilde k$ over its nilradical is a domain. But since $\tilde k/k^s$ is separable, $K\otimes_{k^s}\tilde k$ has a trivial nilradical and so it must be a field. Note that it is even a separably closed field, therefore
    \begin{align*}
        (R\otimes_k \tilde k,\sigma)\cong (R\otimes_K (K\otimes_{k}\tilde k),\sigma) \cong \prod (R\otimes_K(K\otimes_{k^s}\tilde k),\sigma)
        \cong \prod (M_d(K\otimes_{k^s}\tilde k),\top)
    \end{align*}
    Restricting $(D,P)$ to one of the summands, we find that $D=\det^{mq}$. On the symmetric element $\diag(t,1,\dots,1)\in M_d(K\otimes_{k^s}\tilde k[t])^+$, we have $D(\diag(t,1,\dots,1))=t^{qm}=P(\diag(t,1,\dots,1))^2$. This forces $m$ to be even, and so by uniqueness of the Pfaffian, $(D,P)$ is equal to
    \begin{align*}
                (\Norm_{k^s/k}\circ F^q\circ \Nrd_R)^2 &\colon  R\to k
                \\         \Norm_{k^s/k}\circ F^q\circ \Nrd_R &\colon R\to  k
    \end{align*}
    to the $\frac{m}{2}$-th power. 
    \item[(III)] $(R,\sigma)$ is a central simple algebra over an étale $K$-algebra $L$ of degree $2$ equipped with a unitary involution over $L/K$.
    In other words $L$ is either $K \times K$ and $R=E\times E^{\mathrm{op}}$ with $E$ a central simple algebra over $K$, or $L$ is a separable field extension of $K$ and $R$ is a central simple algebra over $L$. Also $\sigma$ is $K$-linear and restricts to the nontrivial element of $\Aut_K(L)$. Then $(D,P)$ is a power of
    \begin{align*}
        \Norm_{k^s/k}\circ F^q\circ \Norm_{L/K} \circ \Nrd_R &\colon R\to k \\
        \Norm_{k^s/k}\circ F^q\circ \Nrd_R & \colon R^+ \to k
    \end{align*}
    This is because $\Nrd_R$ on $R^+$ takes values in $K$. Indeed in the first case, we have that $\sigma$ is given by
    $\sigma(a,b) = (\iota(b), \iota(b))$ with $\iota \colon E \to E^{\mathrm{op}}$
    an isomorphism of central simple algebras over $K$.
    So for $(a, \iota(a)) \in R^+$ with $a \in E$, we have
    $$ \Nrd_R(a,\iota(a)) = (\Nrd_E(a), \Nrd_{E^{\mathrm{op}}}(\iota(a))) = (\Nrd_E(a), \Nrd_E(a)) $$
    The second case follows from the first case by base change \cite[§2, Proposition 2.15]{Inv}.
 \end{enum}
\end{example}

\begin{proposition}\label{symplecticreconstructionthmoverfields} Let $(D,P) \colon (R, *) \to k$ be a $2d$-dimensional weak symplectic determinant law. Then there is an isomorphism
\begin{align}\label{isoAW}
    (R/\ker(D), *) \cong \prod_{i=1}^t (R_i, \sigma_i)
\end{align}
of involutive $k$-algebras, where each $(R_i,\sigma_i)$ is equipped with a symplectic determinant law $(D_i,P_i)$ which takes one of the forms (I)-(III) described in \Cref{exampleforms}, such that
$$ (D,P) = \left(\prod_{i=1}^t D_i\circ \pi_i, \prod_{i=1}^t P_i\circ \pi_i \right) $$
where $\pi_i\colon R \twoheadrightarrow R_i$ are the projections induced by the isomorphism (\ref{isoAW}). In particular $\CH(P) \subseteq \ker(D)$, and $(D,P)$ is a symplectic determinant law.
\end{proposition}

\begin{proof} This Proposition follows from \cite[Theorem 2.16]{MR3444227} and \Cref{invAW}.
\end{proof}

\begin{theorem}\label{reconsalgcl} Let $(R,*)$ be an involutive $\overline k$-algebra. There is a bijection between isomorphism classes of semisimple $2d$-dimensional symplectic representations of $(R,*)$ over $\overline k$ and $2d$-dimensional symplectic determinant laws of $(R,*)$ over $\overline k$ given by sending $\rho \colon (R,*) \to (M_{2d}(\overline k), \mathrm j)$ to $(\det\circ{\rho}, \Pf\circ(\rho\cdot J))$.
\end{theorem}

\begin{proof} Let $(D,P)$ be a symplectic determinant of $(R,*)$ over $\overline k$. By \Cref{symplecticreconstructionthmoverfields} there is a
decomposition
$$ (R/\ker(D), *) \cong \prod_{i=1}^s (R_i, \sigma_i) $$
where the $R_i$ are $K_i$-algebras of the form described in \Cref{exampleforms} for some extension field $K_i/\overline{k}$.

From the description of $K_i$, we see that every element of $K_i$ is algebraic over $\overline{k}$, therefore $K_i = \overline k$ for all $i$ and we have the following three cases:
\begin{enum}
    \item[(I)] $(R_i, \sigma_i) \cong (M_{2n_i}(\overline k), \mathrm j)$. We let $\rho_i \colon (R, *) \to (M_{2n_i}(\overline k), \mathrm j)$ be the corresponding symplectic representation.
    \item[(II)] $(R_i, \sigma_i) \cong (M_{n_i}(\overline k), \top)$. We let
    \begin{align*}
        \rho_i \colon (R,*) &\to (M_{2n_i}(\overline k), \mathrm j) \\
        r &\mapsto \begin{pmatrix} \pi_i(r) & 0 \\ 0 & \pi_i(r) \end{pmatrix}
    \end{align*}
    \item[(III)] $(R_i, \sigma_i) \cong (M_{n_i}(\overline k) \times M_{n_i}(\overline k), \mathrm{swap})$. We let
    \begin{align*}
        \rho_i \colon (R, *) &\to (M_{2n_i}(\overline k), \mathrm j) \\
        r &\mapsto \begin{pmatrix} \pr_1(\pi_i(r)) & 0 \\ 0 & \pr_2(\pi_i(r))^{\top} \end{pmatrix}
    \end{align*}
\end{enum}
In these three cases $(D_i, P_i)$ is of the form $(\det \circ \rho_i, \Pf \circ( \rho_i\cdot J)$.
In particular $(D,P)$ is of the form $(\det \circ \rho, \Pf \circ (\rho\cdot J)$, where $\rho = \bigoplus_{i=1}^s \rho_i$. 
Since $R$ surjects onto the $R_i$, the $\rho_i$ are semisimple and thus $\rho$ is semisimple.

To prove that the map is injective, let us consider two semisimple representations $\rho_1$ and $\rho_2$ of $R$ over $\overline k$ of dimension $2d$ that have the same symplectic determinant. By \cite[Theorem 2.12]{MR3444227}, $\rho_1$ and $\rho_2$ are conjugated by an element $g\in \GL_{2d}(\overline{k})$. We need to show that we can take $g\in \Sp_{2d}(\overline{k})$. 

Since the product of copies of the symplectic group embeds diagonally in a symplectic group up to conjugation, it suffices to check this for direct summands of $\rho_1$ and $\rho_2$.
We can match the irreducible symplectic subrepresentations of $\rho_1$ and $\rho_2$.
An irreducible subrepresentation of $\rho_1$, which is contained in an indecomposable symplectic subrepresentation of $\rho_1$ that is not irreducible, is mapped into an indecomposable symplectic subrepresentation of $\rho_2$ that is also not irreducible.
Thus, we can assume that $\rho_1$ and $\rho_2$ are indecomposable as symplectic representations.

We distinguish two cases:
\begin{enum}
    \item[(a)] $\rho_1$ and $\rho_2$ are irreducible as representations. In this case, they are both surjective onto $M_{2d}(\overline{k})$, so that $\Inner(g)\in \Aut((M_{2d}(\overline{k}),\mathrm{j}))=\PGSp_{2d}(\overline{k})$.
    \item[(b)] The representations are of the form $\rho_i=\rho_{i,1}\oplus {\rho_{i,2}}$ with $\rho_i(r^\sigma)=(\rho_{i,2}(r)^{*},\rho_{i,1}(r)^{\top})$. There exist $g_1,g_2\in \GL_{d}(\overline{k})$ such tht $\rho_{1,1}=g_1\rho_{2,1}g_1^{-1}$ and $\rho_{1,2}=g_2\rho_{2,2}g_2^{-1}$. The compatibility of the representations with the involution implies that $g_2=(g_1^{\top})^{-1}$, and so $\diag(g_1,g_2)=\diag(g_1,(g_1^{\top})^{-1})\in \Sp_{2d}(\overline{k})$.
\end{enum}
\end{proof}
\begin{corollary}\label{exfiniteext}
Let $(R,*)$ be an involutive $k$-algebra equipped with a $2d$-dimensional symplectic determinant law $(D,P)$ over $k$ . Assume, that $R/\ker(D)$ is a finitely generated $k$-algebra. Then there exists a finite field extension $k'/k$ and a symplectic representation $\rho\colon  (R\otimes_k k',*)\to (M_{2d}(k'),\mathrm{j})$ such that $(D \otimes_k k',P \otimes_k k')=(\det\circ \rho),\Pf\circ(\rho\cdot J))$.
\end{corollary}

\begin{proof} By \Cref{strongHomTheorem}, we may assume that $\ker(D) = 0$ and that $R$ is a finitely generated $k$-algebra. We know by \Cref{reconsalgcl} that there is a symplectic representation $\rho_{\overline k} \colon (R \otimes_k {\overline k},*) \to (M_d(\overline k), \mathrm j)$ with $D \otimes_k {\overline k} = \det\circ\rho_{\overline k}$ and $P \otimes_k {\overline k} = \Pf\circ(\rho_{\overline k}|_{(R \otimes_k {\overline k})^+}\cdot J)$. Then the image $\rho_{\overline k}(R) \subseteq M_d(\overline k)$ is as a $k$-subalgebra generated by finitely many matrices in $M_d(\overline k)$. Hence, there is a finite field extension $k'/k$ such that $\rho_{\overline k}(R) \subseteq M_d(k')$, and so we get a symplectic representation $\rho \colon (R \otimes_k {k'}, \sigma \otimes \id_{k'}) \to (M_d(k'), \mathrm j)$.

For every commutative $k'$-algebra $B$ we obtain a diagram
\begin{align*}
    \xymatrix{
        R \otimes_{k'} B \ar[dr]_{\rho \otimes \id} \ar[rr]^{D_B} \ar[dd] && B \ar[dd] \\
        & M_d(B) \ar[ur]_{\det} \ar[dd] \\
        R \otimes_{k'} (B \otimes_{k'} \overline k) \ar[dr]_{\rho \otimes \id} \ar[rr]^{D_{B \otimes_{k'} \overline k}} && B \otimes_{k'} \overline k \\
        & M_d(B \otimes_{k'} \overline k) \ar[ur]_{\det}
    }
\end{align*}
By the functorialities of $D$, $\det$ and the base changes of $\rho$, we know that every square commutes.
The bottom triangle commutes by our choice of $\rho$.
The vertical maps are all injective and so it follows that the top triangle commutes, hence the equality $D \otimes_k k'=\det\circ\rho$. The other equality follows from \Cref{pfaffianunique}.
\end{proof}

\section{Symplectic Cayley-Hamilton algebras}
\label{secsymplCHAlg}

\subsection{Definition and first properties}
\label{subsecsymplCHAlg} Let $(R,*)$ be an algebra with involution over $A$ which is equipped with a $2d$-dimensional symplectic determinant law $(D,P)$. For a commutative $A$-algebra $B$, and $r\in R^+\otimes_A B$, recall that we can associate a characteristic polynomial
$$ \chi^P(r,t) = P_{B[t]}(t-r)=\sum_{i=0}^d(-1)^i \mathcal{T}^P_{i,B}(r)t^{d-i} 
\in B[t].$$
Motivated by the Cayley-Hamilton theorem for the Pfaffian the matrix algebra case (\Cref{charpf}), we give the following definition.
\begin{definition}
We say that the tuple $(R,*,D,P)$ is a \emph{symplectic Cayley-Hamilton $A$-algebra of degree $2d$} if $\chi^P(r,r)=0$ for all $r\in R^+\otimes_A B$ and all $B \in \CAlg_A$, equivalently if $\CH(P)=0$.
\end{definition}

\begin{example}
     We provide an example of a symplectic determinant law that is Cayley-Hamilton but not symplectic Cayley-Hamilton. In other words, we do not always have that $\CH(P)\subseteq \CH(D)$. On the other hand, we will see later that if $A$ is an algebra over a characteristic zero field or if it is local Henselian with $(D,P)$ residually multiplicity-free, then $\CH(D)\subseteq \CH(P)$, although we do not know if this holds in general. Note that the following example is in the content of \Cref{sCHGMA}.
\\ Consider the $A$-algebra
    $$ R=\left\{ \begin{pmatrix}
        (a,a) & (b,0) \\ (0,c) & (d,d) \end{pmatrix}  \ | \ a,b,c,d \in A\right\} \subset M_2(A\times A) $$
equipped with the involution
$$  \begin{pmatrix}
        (a,a) & (b,0) \\ (0,c) & (d,d) \end{pmatrix}^*= \begin{pmatrix}
        (d,d) & (b,0) \\ (0,c) & (a,a) \end{pmatrix},  $$
so that $R^+= \left\{ \begin{pmatrix}
        (a,a) & (b,0) \\ (0,c) & (a,a) \end{pmatrix}\ | \ a,b,c\in A\right\}$. The polynomial laws $D\colon R\to A$ and $P\colon R^+\to A$ defined for every commutative $A$-algebra $B$ by
        \begin{align*}
            D_B\left(\begin{pmatrix}
        (a,a) & (b,0) \\ (0,c) & (d,d) \end{pmatrix}\right) &=ad \quad (a,b,c,d\in B),
        \\ P_B\left(\begin{pmatrix}
        (a,a) & (b,0) \\ (0,c) & (a,a) \end{pmatrix}\right) &=a \quad (a,b,c\in B )
        \end{align*}
        define a $2$-dimensional symplectic determinant law such that $(R,D)$ is Cayley-Hamilton. Note however that the ideal $\CH(P)$ is non zero and is generated by elements of the form $\begin{pmatrix}
            0 & (b,0) \\ 0 & 0 
        \end{pmatrix}$ and  $\begin{pmatrix}
            0 & 0 \\ (0,c) & 0 
        \end{pmatrix}$ for $b,c\in A$, and that we indeed have $\CH(P)\subseteq \ker(D)$. 
\end{example}

\begin{proposition}\label{SCHisfinite}
    If $(R,*,D,P)$ is a finitely generated symplectic Cayley Hamilton $A$-algebra of degree $2d$, then $R$ is a finite $A$-module.
\end{proposition}
\begin{proof}
    Since every element $r\in R$ can be written as $r=r_1+r_2$ with $r_1^*=r_1$ and $r_2^*=-r_2$ and that $r_1$ and $r_2$ are roots of $\chi^P(r_1,t)$ and $\chi^P(r_2^2,t^2)$, we get that $r$ is integral over $A$ of degree $\le 2d^2$. Therefore by \cite[Proposition 3.22]{Procesi73}, we get that $R$ is a PI (Polynomial Identity) $A$-algebra (we invite the reader to consult \cite[1.2.2]{MR3167286} for the definition of a PI algebra). Hence by \cite[Theorem 2.7]{Procesi73}, we get that $R$ is a finite $A$-module.
\end{proof}
The following two lemmas are the variants of \cite[Lemma 2.4]{MR3444227} and \cite[Lemma 2.10 (i)]{MR3444227} in our setting. Both are needed for the proof of \Cref{henselianlemma}.
\begin{lemma}\label{SCHidempotent}
Assume that $\Spec(A)$ is connected, that $(R,*)$ is an involutive $A$-algebra equipped with a (weak) symplectic determinant law $(D,P)$ of degree $2d$, and that $e\in R^+$ is a symmetric idempotent element. Then the polynomial laws given for any commutative $A$-algebra $B$ by
\begin{align*}
    &D_{e,B}\colon eRe\otimes_A B\to B, \ r \mapsto D_B(r+1-e)
    \\ &P_{e,B}\colon (eRe)^+\otimes_A B\to B, \ r\mapsto P_B(r+1-e)
\end{align*}
define a (weak) symplectic determinant law of degree $2d_e\le 2d$. Moreover, if $(D,P)$ is symplectic Cayley-Hamilton, then so is $(D_e,P_e)$. 
\end{lemma}
\begin{proof}
    It is straightforward that $(D_e,P_e)$ define a weak symplectic determinant law. So let us show that if $\CH(P)\subseteq \ker(D)$, then $\CH(P_e)\subseteq \ker(D_e)$. For this, we note that if $r_1,\dots,r_n\in (eRe)^+$, then $(r_1t_1+\cdots+r_nt_n+1-e)^k=(r_1t_1+\cdots+r_nt_n)^k+1-e$ in $eRe[t_1,\dots,t_n]$. Hence for $\alpha=(\alpha_1,\dots,\alpha_n)\in \mathbb{N}^n$, if we write $r_\alpha$ (resp. $r'_\alpha)$ for the coefficient in front of $t_1^{\alpha_1}\cdots t_n^{\alpha_n}$ of  $\chi^{P_e}(r_1t_1+\cdots+r_nt_n,r_1t_1+\cdots+r_nt_n)$ (resp. $\chi^{P}(r_1t_1+\cdots+r_nt_n+1-e,r_1t_1+\cdots+r_nt_n+1-e))$, then $r_\alpha=er'_\alpha$. So for every commutative $A$-algebra $B$, and $s\in eRe\otimes_A B$,
    $$ D_e(e+ sr_\alpha)=D(e+sr_{\alpha}+1-e)=D(1+sr'_\alpha) $$
    We get the desired implication by applying \cite[Lemma 1.19 (i)]{MR3444227}.
    \\ Now suppose that $(D,P)$ is symplectic Cayley-Hamilton. Note that for a commutative $A$-algebra $B$, and $r\in (eR^+e \oplus (1-e)R^+(1-e))\otimes_A B$, $\chi^{P_e}(t,er)=P_{e,B[t]}(te-r+1-e)$. Since $te-er+1-e$ and $t(1-e)-(1-e)r+e$ commute and their product is equal to $t-r$, we get by \Cref{Pcommutes} that
    $$\chi^P(r,t)=P_{B[t]}(r-t)=P_{B[t]}(te-r+1-e)P_{B[t]}(t(1-e)-(1-e)r+e)=\chi^{P_e}(er,t)\chi^{P_{1-e}}((1-e)r,t) $$ 
    For $r\in eR^+e\otimes_A B$, we apply the Cayley-Hamilton identity to $r$ and $r+1-e$ to get
    $$\chi^{P_e}(r,r)r^{d_{1-e}}=0 \quad \text{ and }\quad \chi^{P_e}(r,r+1-e)(r-e)^{d_{1-e}}=0 $$
    By the same argument as in the first part of the proof, we have that $\chi^{P_e}(r,r+1-e)(r-e)^{d_{1-e}}=\chi^{P_e}(r,r)(r-e)^{d_{1-e}}$. Now since the ideal generated by $t^{d_{1-e}}$ and $(t-1)^{d_{1-e}}$ in $B[t]$ is $B[t]$ itself, we get that $\chi^{P_e}(r,r)=0$ (note that $e$ is the unit element in $eRe$).
\end{proof}

\begin{lemma}\label{radR}
    Assume that $A$ is a local ring with maximal ideal $\mathfrak{m}_A$ and residue field $k=A/\mathfrak{m}_A$. Let $(R,*,D,P)$ be a symplectic Cayley-Hamilton $A$-algebra of degree $2d$, $\overline{R}=R/\mathfrak{m}_A R$, and denote by $\overline{D}\colon \overline{R}\to k$ the reduction of $D$ modulo $\mathfrak{m}_A$. Then $\ker(R\rightarrow \overline{R}/\ker(\overline{D})) $ is equal to the Jacobson radical $\Rad(R)$ of $R$. In particular, if $A=k$ is a field, then $\ker(D)=\Rad(R)$.
\end{lemma}

\begin{proof}
    First note that
    $$ \Rad(R)=\left\{r\in R\ | \ D_A(1+rr')=D_A(1+r'r)\in A^\times, \ \forall r'\in R\right\} $$
    Indeed since $D$ is multiplicative, we have that $D_A(R^\times)\subseteq A^\times$. Conversely if for $r\in R$, $D_A(r)\in A^\times$, then by the symplectic Cayley-Hamilton property applied to $rr^*$, we have
    \begin{equation*}
   r(r^*(rr^*)^{d-1}-\mathcal{T}^{P}_{1,A}(rr^*)r^*(rr^*)^{d-2}+\cdots +(-1)^{d-1}\mathcal{T}^{P}_{d-1,A}(rr^*)r^*)=(-1)^{d+1}P_A(rr^*)= (-1)^{d+1}D_A(r)\in A^\times.
    \end{equation*}
    Here the last equality follows from \Cref{propertyofP}.
    Therefore, we have the above equality of sets since $\Rad(R)$ is the set of elements $r\in R$ such that $1+rr'$ and $1+r'r$ are invertible for all $r'\in R$ and by \cite[Lemma 1.12 (i)]{MR3444227}.
    
    Now let us write $I=\ker(R\rightarrow \overline{R}/\ker(\overline{D}))$. By definition and \cite[Lemma 1.19 (i)]{MR3444227}, we have  $D(1+I)\subset 1+\mathfrak{m}_A$. Consequently, we have that $I\subseteq \Rad(R)$. To show the reverse inclusion, note that $\mathfrak{m}_A R \subseteq I \subseteq \Rad(R)$, so we can suppose that $A=k$. First let us assume that $k$ is an infinite field, then the paragraph after \cite[Lemma 1.19 (i)]{MR3444227} tells us that
    $$ \ker(D)=\left\{r\in R \ | \ D_A(1+rr')\in A^\times, \ \forall r'\in R\right\}.$$
    In other words, we have that $\Rad(R)=\ker(D)$. \\ Finally, the case where $k$ is finite follows from \cite[Lemma 2.8 (i)]{MR3444227}.
\end{proof}
\subsection{Symplectic Cayley-Hamilton algebras in characteristic zero}\label{CH0}
The embedding problem consists of asking under which conditions can a noncommutative ring $R$ be embedded into the ring of matrices $M_{d}(B)$ over a commutative ring $B$. Procesi proposed a modification to this problem, by adding the structure of a trace to the algebra $R$ and asking whether this embedding can be made compatibly with the trace.
In \cite{Pro87}, he gives a solution to the modified problem when $R$ is a trace algebra over a characteristic zero field. Specifically, he shows that in this case, $R$ embeds into a matrix algebra $M_d(B)$ compatibly with the trace if and only if it is Cayley-Hamilton. In this subsection, our goal is to establish \Cref{converseCH}, which serves as the symplectic counterpart of the main theorem of \cite{Pro87}. In fact, the proof follows the same lines as \cite{Pro87}.

We assume that $A$ is a commutative $\mathbb{Q}$-algebra. In this setting, the theory of pseudocharacters and determinant laws are equivalent \cite[Proposition 1.27]{MR3444227}. This equivalence allows us borrow results from \cite{AGPR} (the reader is also invited to consult \cite[11.8]{ProcesiLG}).

We first start by giving a concrete description of the free symplectic Cayley-Hamilton algebra on a set $S$. But before we proceed, we first need to explicitly define this object. So let us consider the free $A$-algebra with involution in $S$-variables $A\langle S\rangle=A\langle x_s,x_s^*\ | \ s\in S\rangle$. The algebra $\mathcal{F}_{S}(2d)\colonequals A[\SpDet_{(A\langle S \rangle,*)}^{\square,2d}]\otimes_A A\langle S \rangle$ equipped with the universal symplectic determinant law 
\begin{equation*}
    (D^u,P^u)\colon \mathcal{F}_S(2d)\to A[\SpDet_{(A\langle S \rangle,*)}^{\square,2d}]
\end{equation*}
is the free symplectic determinant $A$-algebra of dimension $2d$ on the set $S$ in the sense of \Cref{spdetalg}. For an integer $m\ge 1$, we will write $\mathcal{F}_m(2d)$ for $\mathcal{F}_{\{1,\dots,m\}}(2d)$.
\begin{definition}
    An ideal $I \subseteq \mathcal{F}_S(2d)$ is called a \emph{T-ideal} if for every endomorphism of symplectic determinant $A$-algebras $\big(\varphi_0\colon A[\SpDet_{(A\langle S \rangle,*)}^{\square,2d}]\to A[\SpDet_{(A\langle S \rangle,*)}^{\square,2d}], \ \varphi\colon \mathcal{F}_S(2d)\to \mathcal{F}_S(2d)\big)$, we have $\varphi(I)\subseteq I$. 
\end{definition} 
For instance, $\CH(D^u)$ and $\CH(P^u)$ are $T$-ideals. Now we use \Cref{pfaffianunique} to  express $P^u$ in terms of the coefficients $\Lambda_i^{D^u}$ of the characteristic polynomial of $D^u$. Furthermore, the Newton relations (see \cite[1.11 (ii)]{MR3444227}) allow us to express the $\Lambda_i^{D^u}$ in terms of the trace $\tr\colonequals \Lambda_1^D$, so that $P^u$ can be expressed in terms of the trace alone. In other words, for any commutative $A$-algebra $B$, and  $r\in \mathcal{F}_d(2d)^+\otimes_A B$, we have $P^u_B(r)\in \mathbb{Q}[\tr(r),\tr(r^2),\cdots]$. In particular, for indeterminates $t_1,\dots,t_d$, we have that 
$$ \chi^{P^u}(t,(x_1+x_1^*)t_1+\cdots +(x_d+x_d^*)t_d)\in \mathbb{Q}[\tr(W) \ | \ W\text{ a word in }\{x_1,x_1^*,\dots,x_d,x_d^*\}][t_1,\dots,t_d][t]. $$
We let $\widetilde{\Pf}(x_1,\dots,x_d)\in A[\tr(W) \ | \ W\text{ a word in }\{x_1,x_1^*,\dots,x_d,x_d^*\}]\subset \mathcal{F}_d(2d)$ be the coefficient of $t_1\cdots t_d$ in the polynomial $ \chi^{P^u}((x_1+x_1^*)t_1+\cdots +(x_d+x_d^*)t_d,(x_1+x_1^*)t_1+\cdots +(x_d+x_d^*)t_d)$.

\begin{example}
For $d=4$, using \Cref{exd=2} and the Newton relations, we find that for $M\in M_8(A)^+$,
$$ \Pf(M)= \frac{7}{384}\tr(M)^4-\frac{3}{32}\tr(M)^2\tr(M^2)+ \frac{1}{12}\tr(M)\tr(M^3)+\frac{3}{32}\tr(M^2)^2-\frac{1}{8}\tr(M^4), $$
so the Pfaffian characteristic polynomial is
\begin{align*}
    \Pf_M(t) &= t^4 -\tr(M)t^3 +\frac{1}{8}(\tr(M)^2-\tr(M^2))t^2+\frac{1}{16}(\tr(M^3)+2\tr(M)\tr(M^2))t + \Pf(M).
\end{align*}
\end{example}
The algebra $ \mathcal{F}_S(2d)/\CH(P^u)$ equipped with the corresponding involution and symplectic determinant law is the free symplectic Cayley-Hamilton algebra that we want to describe. Using the process of polarization and specialization (see for example \cite[\S 13]{DeConPro17}), we see that $\CH(P^u)$ is generated by $\widetilde{\Pf}$ as a $T$-ideal. In other words, it is the ideal generated by the elements $\varphi(\widetilde{\Pf})$ for every morphism $$\big(\varphi_0\colon A[\SpDet_{(A\langle d\rangle,*)}^{\square,2d}]\to A[\SpDet_{(A\langle S\rangle,*)}^{\square,2d}],\ \varphi \colon \mathcal{F}_d(2d) \to  \mathcal{F}_S(2d)\big)$$ of symplectic determinant $A$-algebra. 

We will now proceed to provide the desired description. We write $V=A^{2d}$ for the canonical free $A$-module of rank $2d$, which we equip with the canonical symplectic pairing given by $[v,v']=v^\top J v'$.
\\ Recall from the proof of \Cref{spreprepr} that $A[\Rep^{\square,2d}_{(A\langle S\rangle, *)}]=A[M_{2d}^S]=A[\bbX_{k,h}^{(s)}, \ 1\le k,h\le 2d, s\in S]$ with $\bbX_{k,h}^{(s)}$ being the $(k,h)$-coordinate map of the generic matrix indexed by $s$. We introduce the in algebras $R_S(2d)=M_{2d}(A[M_{2d}^S])^{\Sp_{2d}}$ and $T_S(2d)=A[M_{2d}^S]^{\Sp_{2d}}$, where the action of $\Sp_{2d}$ is described in \Cref{subsymprep}. In other words, $R_S(2d)$ is the $A$-algebra of $\Sp(V)$-equivariant polynomial maps
$$ f\colon \End(V)^{S} \longrightarrow \End(V)$$
with $\Sp(V)$ acting on $\End(V)$ by conjugation and on $\End(V)^S$ by diagonal conjugation. Additionally, $T_S(2d)$ is the commutative subalgebra of $R_S(2d)$ of polynomial maps with values in scalar matrices. For an integer $m\ge 1$, we will write $R_m(2d)$ and $T_m(2d)$ for the algebras $R_{\{1,\dots,m\}}(2d)$ and $T_{\{1,\dots,m\}}(2d)$.
\\ We equip $R_S(d)$ with the involution and symplectic determinant law $(D_S,P_S)\colon R_S(2d)\to T_S(2d)$ given by restriction from the standard ones on $(M_{2d}(A[M_{2d}^S]),\mathrm{j})$. 

Recall that $ R_S(2d)$ is equipped with a trace map, and since we are working in characteristic zero, we have that $T_{S}(2d)=\tr(R_S(2d))$. Also remark the coordinate maps $\bbX^{(s)} \colon (M_s)_{s\in S} \mapsto M_s$ are $\Sp(V)$-equivariant. This allows us to give the following description.
\begin{theorem}\label{generatorsofinv}
The algebra $T_S(2d)$ is generated over $A$ by the maps $\tr(W)$, where $W$ is a word in $\bbX^{(s)}, (\bbX^{(s)})^{\mathrm{j}}$ for $s\in S$. The algebra
$R_S(2d)$ is generated over $T_S(2d)$ by the maps $\bbX^{(s)}, (\bbX^{(s)})^{\mathrm{j}}$ for $s\in S$.
\end{theorem}
\begin{proof}
    This follows directly from \cite[Theorem 13.1.4]{AGPR}.
\end{proof}
Consider the morphism symplectic determinant $A$-algebras 
$$\big(\pi_0\colon A[\SpDet_{(A\langle S\rangle, *)}^{\square,2d}]\to T_S(2d),\ \pi\colon \mathcal{F}_S(2d)\to R_S(2d)\big)$$ 
where $\pi_0$ is the map induced by $(D_S,P_S)$, and $\pi_0$ is the map  sending $x_s$ to $\bbX^{(s)}$ for $s\in S$. It is surjective and its kernel $\ker(\pi)$ is a $T$-ideal called the \emph{ideal of trace identities}. Our next step is to give a description of $\ker(\pi)$, which is the content of the following Proposition.

\begin{proposition}\label{kerpi}
The $T$-ideals $\ker(\pi)$ and $\CH(P^u)$ are equal, and the morphism
$$ \pi_0\colon A[\SpDet_{(A\langle S\rangle, *)}^{\square,2d}]\longrightarrow T_S(2d)  $$
is an isomorphism. In particular, $R_S(2d)$ is the free symplectic Cayley-Hamilton $A$-algebra on a set $S$.
\end{proposition}

First note that of an element $f\in \ker(\pi)$, there exists an integer $m\ge 1$ and an embedding $\iota\colon \{1,\dots, m\}\hookrightarrow S$ so that in the commutative diagram
\begin{center}
    \begin{tikzcd}
        \mathcal{F}_{m}(2d) \arrow[d,"\pi"] \arrow[r, "\iota_*"] & \mathcal{F}_S(2d) \arrow[d,"\pi"]
        \\ R_m(2d) \arrow[r] & R_S(2d)
    \end{tikzcd}
\end{center}
we have that $f\in \text{im}(\iota_*)$. Therefore, it suffices to understand $\ker(\pi)$ for $S=\{1,\dots,m\}$. So let us fix such an integer $m\ge 1$.
\\ $ $\\ Using the pairing $[\cdot,\cdot]$, we have an identification $V\otimes_A V\cong\End(V)$ as $\Sp(V)$-modules via the map $u\otimes v \mapsto (x\mapsto [v,x]u) $. We record the following identities:
\begin{equation}\label{coding}
     (u\otimes v)\circ (u'\otimes v')=u\otimes [v,u']v', \quad \tr(u\otimes v)=-[u,v], \quad (u\otimes v)^{\mathrm{j}}=-v\otimes u
\end{equation}
Thus the space of elements of $T_m(2d)$ which are multilinear in the variables $\bbX^{(1)},\dots,\bbX^{(m)}$ can be identified with $[(V\otimes_A V)^*\otimes \cdots \otimes (V\otimes_A V)^*]^{\Sp(V)}$, which is the space of multilinear functions $h\colon V^{\oplus 2m}\to A$ which are invariant under the action of $\Sp(V)$.

\begin{remark} Let $u_i, v_i \in V$ for $1\le i\le m$ arbitrary. Let $M_i \in \End(V)$ be the elements corresponding to $u_i \otimes v_i$ under the above identification. The identities in (\ref{coding}) allow us to express any product $\prod_{i=1}^m [y_i,z_i]$, where $y_1,\dots,y_m,z_1,\dots,z_m$ is a permutation of $u_1,\dots,u_m,v_1,\dots,v_m$, in terms of a product involving the traces of words in $\{M_i,M_i^*\}$ (this product is multilinear in the $M_i$). For example, if $m=6$, we have
$$ [u_1,u_4][v_4,v_6][u_6,v_1][u_2,u_5][u_3,v_5][v_2,v_3]=-\tr(M_1^*M_4M_6)\tr(M_2^*M_5M_3) $$
\end{remark}

We can see $V^{\oplus 2m}$ as the space of $2d \times 2m$ matrices $Y$ with the action of $X \in \Sp(V)$ given by matrix multiplication $XY$. As explained in \cite[§11.5.1]{ProcesiLG} (the result there is over $\bbC$ but extends easily to commutative $\bbQ$-algebras), the mapping $q\colon Y\mapsto Y^{\top}JY$ from the space of $2d\times 2m$ matrices to the space of antisymmetric $2m\times 2m$ matrices of rank $\le 2d$ is the quotient map under the action of $\Sp(V)$. On the level of coordinate rings, we get a surjective morphism
$$ q^\sharp\colon \Sym(\wedge^2(A^{2m})^*)\twoheadrightarrow \Sym\big((V^{\oplus 2m})^*\big)^{\Sp(V)}$$
whose kernel is generated by the Pfaffian of the principal $2d$ minors in the antisymmetric $2m \times 2m$ matrices.

Suppose that $M_i \in \End(V)$ is attached to $u_i\otimes v_i$ ($u_i,v_i\in V$) and let $Y=(u_1,\dots,u_m,v_1,\dots,v_m)$ so that
\begin{align}\label{codingmap}
    Y^{\top}JY = \begin{pmatrix} 0 & [u_1,u_2] & \cdots & [u_1,u_m] & [u_1,v_1] & \cdots & [u_1,v_m]
\\ [u_2,u_1] & 0 & \cdots & [u_2,u_m] & [u_2,v_1] & \cdots & [u_2,v_m]
\\ \cdots & \cdots & \cdots & \cdots & \cdots & \cdots & \cdots
\\ \cdots & \cdots & \cdots & \cdots & \cdots & \cdots & \cdots
\\ [v_1,u_1] & [v_1,u_2] & \cdots & [v_1,u_m] & 0  & \cdots & [v_1,v_m]
\\ \cdots & \cdots & \cdots & \cdots & \cdots & \cdots & \cdots
\\ [v_m,u_1] & [v_m,u_2] & \cdots & [v_m,u_m] & [v_m,v_1]  & \cdots & 0
\end{pmatrix}
\end{align}
A multilinear element in $\ker(q^\sharp)$ corresponds to a linear combination of polynomials of the form
\begin{equation}\label{multidentity}
    [w_1,\dots, w_{2d+2}][y_1,z_1]\cdots[y_t,z_t]
\end{equation}
where $m=d+1+t$, the elements
$$ w_1,w_2,\dots,w_{2d+2},y_1,z_1,\dots, y_t,z_t $$
are given by a permutation of $u_1,\dots,u_m,v_1,\dots, v_m$, and $[w_1,\dots,w_{2d+2}]$ denotes the Pfaffian of the principal minor of $Y^{\top}JY$ corresponding to the rows and columns in which $w_1,\dots,w_{2d+2}$ appear in the entries. We show by induction on $m$ that all these relations are a consequence of the following one:
$$ P_{d+1}(M_1,\dots,M_{d+1})\colonequals [u_1,\dots,u_{d+1},v_1,\dots,v_{d+1}] $$
For $m=d+1$ this is obvious. For $m>d+1$, by performing the operations $M_i\leftrightarrow M_j$ and $M_i\leftrightarrow M_i^{\mathrm{j}}$, which amounts to doing the exchanges $u_i,v_i\leftrightarrow u_j,v_j$ and $u_i\leftrightarrow u_j$, the expression (\ref{multidentity}) can be reduced to one of the form
\begin{equation}\label{multidentity2}
    [u_1,\dots,u_k,v_1,\dots,v_k,v_{k+1},\dots,v_{2d+2-2k}][y_1,z_1]\cdots[y_t,z_t]
\end{equation}
where we may assume that $k<d+1$. We look at the term $[y_i,z_i]$ in which $u_{k+1}$ appears. If it is (up to a sign) of the form $[u_j,u_{k+1}]$, then we introduce the new variable $\overline{M}_{j}=M_j^{\mathrm{j}}M_{k+1}=\overline{u}_j\otimes \overline{v}_j$ with $\overline{u}_j=-v_j$ and $\overline{v}_j=-[u_j,u_{k+1}]v_{k+1}$. Then we have that
\begin{align*}
    [u_1,\dots,u_k,v_1,\dots,v_k,v_{k+1},\dots,v_{2d+2-2k}][u_j,u_{k+1}] &= [u_1,\dots,u_k,v_1,\dots,v_k,[u_j,u_{k+1}]v_{k+1},\dots,v_{2d+2-2k}]
            \\ &= [u_1,\dots,u_k,v_1,\dots,v_k,\overline{v}_j,\dots,v_{2d+2-2k}]
\end{align*}
Thus we eliminated the variable $M_{k+1}$ and $(\ref{multidentity2})$ can be expressed in terms of the $m-1$ variables $M_1,\dots,M_{k-1}$, $M_{k+1},\dots, M_{j-1}, \overline{M}_j,\dots,M_m$. The case where the term $[y_i,z_i]$ is of the form $[v_j,u_{k+1}]$ (up to a sign) is treated similarly by introducing the variable $\overline{M}_j=M_jM_{k+1}$. This shows the claim by induction.
\\ Note that using the identities in $(\ref{coding})$, any map of the form $(M_1,\dots,M_m)\mapsto[y_1,z_1]\cdots [y_t,z_t]$ where $y_j,z_j$ are among the $u_i,v_i$, can be written as a product of $\tr(W)$ for $W$ a word in $M_i,M_i^{\mathrm{j}}$. In particular, we can see $P_{d+1}$ as an element of $\tr(\mathcal{F}_{d+1}(2d))$ such that $\pi(P_{d+1})=0$. In fact, we have proved that up to a scalar, $P_{d+1}$ is the unique multilinear identity in $\tr(\mathcal{F}_{d+1}(2d))$.

On the other hand, consider elements $M_1,\dots,M_d,M_{d+1}\in \End(V) \cong M_{2d}(A)$. The image of $\tr(\widetilde{\Pf}x_{d+1})$ under the map $\mathcal{F}_{d+1}(2d) \to \End(V)$ sending $x_i$ to $M_i$, is zero. By the uniqueness of $P_{d+1}$ discussed above, both identities are equal up to a scalar.

\begin{proof}[Proof of \Cref{kerpi}] Using operators of polarization and restitution, one can show that any two $T$-ideals containing the same multilinear elements coincide (see \cite[§3.1, Proposition 3.1.10]{AGPR}, the arguments work for arbitrary $\bbQ$-algebras). In particular, $\ker(\pi)$ is generated as a $T$-ideal by the set of its multilinear elements. So let $f(x_1,\dots,x_m)\in \ker(\pi)$ be a multilinear element. Seeing $\mathcal{F}_{m}(2d)$ inside $\mathcal{F}_{m+1}(2d)$, we get that $\tr(f(x_1,\dots,x_m)\cdot x_{m+1})$ is a multilinear element in $\tr(\mathcal{F}_{m+1}(2d))$. By the above discussion, this is a linear combination of elements of type
$$N \cdot \tr(\widetilde{\Pf}(W_1,\dots,W_d)\cdot W_{d+1}) $$
where $N,W_1,\dots,W_{d+1}$ are monomials in $x_1,x_1^*,\dots,x_{m+1},x_{m+1}^*$. Now we consider two cases, either the variable $x_{m+1}$ appears in $N$ or in one of the $W_i$. In the first case, we have
$$ N \cdot \tr(\widetilde{\Pf}(W_1,\dots,W_d)\cdot W_{d+1}) =\tr(\widetilde{\Pf}(W_1,\dots,W_{d})Bx_{m+1}) $$
for some expression $B$, and $\widetilde{\Pf}(W_1,\dots,W_{d})B$ is obviously a consequence of the Pfaffian identity. 
\\ Since $\tr(\widetilde{\Pf}(W_1,\dots,W_d)W_{d+1})$ is up to a scalar equal to $P_{d+1}(W_1,\dots,W_{d+1})$ and that the $P_{d+1}$ comes from taking the Pfaffian of the matrix (\ref{codingmap}) in $d+1$ variables, permuting the variables might only change the sign. Hence, we can assume that $x_{m+1}$ appears in $W_{d+1}=B\cdot x_{m+1}\cdot C$. Hence,
\begin{align*}
    N \cdot \tr(\widetilde{\Pf}(W_1,\dots,W_d)W_{d+1}) &= N \cdot\tr(\widetilde{\Pf}(M_1,\dots,M_d)B\cdot x_{m+1}\cdot C)
    \\ &=\tr(CN\cdot \widetilde{\Pf}(W_1,\dots,W_d)B\cdot x_{m+1})
\end{align*}
and $CN\cdot \widetilde{\Pf}(W_1,\dots,W_d)B$ is a consequence of the Pfaffian identity. This shows that $\ker(\pi)=\CH(P^u)$.
\\ Finally, note that $A[\SpDet_{(A\langle S\rangle, *)}^{2d}]\cap \CH(P^u)=\{0\}$ since $\CH(P^u)\subseteq \ker(D^u)$, and so $\pi_0$ is injective. Surjectivity follows from the description of $T_S(2d)$ in \Cref{generatorsofinv}.
\end{proof}

\begin{remark}\label{vacresult}
    We conjecture that the isomorphism $\pi_0\colon A[\SpDet_{(A\langle S\rangle, *)}^{2d}]\xrightarrow{\sim} T_S(2d)$ holds over arbitrary commutative rings $A$ with $2\in A^\times$. This is the analogue of \cite[Theorem 6.1]{vac08}.
\end{remark}

The last ingredient we need for the proof of our result is the functorial Reynolds operator. Given a flat affine group scheme $G$, a functorial Reynolds operator is the datum, for each $G$-module $M$, of a $G$-equivariant homomorphism $\mathcal{R}_G: M\to M^G$ such that $\mathcal{R}_G|_{M^G}=\id$ and such that for every morphism $\varphi: M\to N$ of $G$-modules, we have a commutative diagram
\begin{center}
    \begin{tikzcd}
        M \arrow[r,"\mathcal{R}_G"] \arrow[d,swap,"\varphi"] & M^G \arrow[d,"\varphi_{|M^G}"]
        \\ N \arrow[r,"\mathcal{R}_G"] & N^G
    \end{tikzcd}
\end{center}
In particular, if $M=R$ is an $A$-algebra with trace and involution such that $G$ acts by morphisms respecting these structures, then we have
\begin{equation}\label{reynolds}
    \mathcal{R}_G(r^*)=\mathcal{R}_G(r)^*, \quad \mathcal{R}_G(sr)=s\mathcal{R}_G(r),\quad \mathcal{R}_G(rs)=\mathcal{R}_G(r)s, \quad \tr(\mathcal{R}_G(r))=\mathcal{R}_G(\tr(r))
\end{equation}
for all $s\in R^G$ and $r\in R$.
In \cite{wang2022arithmetic}, the author extends the notion of linear reductivity to arbitrary base rings (see \cite[Definition 3.2]{wang2022arithmetic}). By \cite[Lemma 3.3]{wang2022arithmetic}, $G$ is linearly reductive if and only if a functorial Reynolds operator $\mathcal{R}_G$ exists. Moreover, by \cite[Proposition 3.6]{wang2022arithmetic}, this notion is stable under base change. Therefore, $\Sp_{2d,A}$ is linearly reductive. We note that the $\Sp_{2d, A}$-invariants of $V$ and the $\Sp_{2d,\bbQ}$-invariants of the restriction of $V$ to an $\Sp_{2d,\bbQ}$-module are equal, so we can drop $A$ from the notation when taking invariants.
\begin{theorem}[Converse Cayley-Hamilton Theorem]\label{converseCH}
Let $(R,*,D\colon R\to A,P\colon R^+\to A)$ be a symplectic Cayley-Hamilton $A$-algebra of degree $2d$. Then the universal mapping
\begin{equation*}
    \rho^{u}\colon R \longrightarrow M_{2d}(A[\SpRep_{(R,*,D,P)}^{\square,2d}])
\end{equation*}
restricts to an isomorphism $R\cong M_{2d}(A[\SpRep_{(R,*,D,P)}^{\square,2d}])^{\Sp_{2d}}$.
\end{theorem}

\begin{proof}
The proof is based on \cite{Pro87} (also see the proof of \cite[Theorem 14.2.1]{AGPR}). Let $R$ be as in the statement, then by \Cref{kerpi}, we can present it as $R=R_S(2d)/I$ where $I$ is an ideal of $R_S(2d)$. For ease of notation, let us write $B := M_{2d}(A[M_{2d}^S])$, so that $R_S(2d)= B^{\Sp_{2d}}$. By \cite[Lemma 2.6.2]{AGPR}, we can write $BIB=M_{2d}(J)$ for some ideal $J$ of $A[M_{2d}^S]$. Therefore by linear reductivity of $\Sp_{2d,A}$ and by \cite[Proposition 3.4]{wang2022arithmetic}, we get a surjective map $u\colon R=R_S(2d)/I\twoheadrightarrow M_{2d}(A[M_{2d}^S]/J)^{\Sp_{2d}}$. Note that since $A[\SpDet_{(A\langle S\rangle, *)}^{2d}]=T_S(2d)$, we have that $R_S(2d)=A[\SpRep_{(A\langle S\rangle, *, D_S, P_S)}^{\square,2d}]$. Therefore, we get from the proof of \Cref{representability}, that $A[M_{2d}^S]/J=A[\SpRep_{(R,*,D,P)}^{2d,\square}]$. To prove the theorem, we just have to show that $u$ is injective, which amounts to showing that $BIB\cap R_S(2d)=I$. So let $a=\sum_i a_iu_ib_i\in BIB\cap R_S(2d)$ with $a_i,b_i\in B$, $u_i\in I$, and let $s\in S$ be an element which is independent of $a$, i.e. the variables $\bbX^{(s)}_{h,k}$ for $1\le h,k\le 2d$ do not appear in the expression of $a$. We consider
$$ \tr(a\bbX^{(s)})=\tr(\sum_i a_iu_ib_i\bbX^{(s)})=\tr(\sum_i b_i\bbX^{(s)} a_iu_i). $$
Applying the Reynolds operator $\mathcal{R}_{\Sp_{2d}}$ and using the identities (\ref{reynolds}), we get that
\begin{align*}
    \tr(a\bbX^{(s)})=\tr(\mathcal{R}_{\Sp_{2d}}(\sum_i b_i\bbX^{(s)} a_iu_i))=\tr(\sum_i\mathcal{R}_{\Sp_{2d}}(a_i \bbX^{(s)} b_i)u_i).
\end{align*}
By \Cref{generatorsofinv}, since $\mathcal{R}_{\Sp_{2d}}(a_i \bbX^{(s)} b_i)$ is linear in $\bbX^{(s)}$, we can write $$\mathcal{R}_{\Sp_{2d}}(a_i \bbX^{(s)} b_i)=\sum_j \alpha_{i,j}x\beta_{i,j}+\sum_{k} \alpha_{i,k}' (\bbX^{(s)})^{\mathrm{j}}\beta_{i,k}'+\sum_h \tr(m_{i,h}x)n_{i,h}$$
for $\alpha_{i,j},\beta_{i,j},\alpha_{i,k}',\beta_{i,k}',m_{i,h},n_{i,h}\in R_S(2d)$ independent of $s$. Thus,
\begin{align*}
     \tr\left(a\bbX^{(s)}- \sum_i\left(\sum_j \alpha_{i,j}\bbX^{(s)} \beta_{i,j}+\sum_{k} \alpha_{i,k}' (\bbX^{(s)})^{\mathrm{j}}\beta_{i,k}'+\sum_h \tr(m_{i,h}\bbX^{(s)})n_{i,h}\right)u_i\right)
     \\ =\tr\left(\left(a-\sum_i\left(\sum_j \beta_{i,j} u_i\alpha_{i,j}+\sum_k \alpha_{i,k}'^{*}u_i^*\beta_{i,k}'^*+\sum_h \tr(n_{i,h}u_i)m_{i,h}\right)\right)\bbX^{(s)}\right)=0.
\end{align*}
This implies that $a=\sum_i\left(\sum_j \beta_{i,j} u_i\alpha_{i,j}+\sum_k \alpha_{i,k}'^{*}u_i^*\beta_{i,k}'^*+\sum_h \tr(n_{i,h}u_i)m_{i,h}\right)$.
Indeed, if $\alpha \in B$ is independent of $s$, then for every specialisation map $\varphi\colon  A[M_{2d}^S] \to A[M_{2d}^{S \setminus \{s\}}]$, inducing a morphism of $A[M_{2d}^{S \setminus \{s\}}]$-algebras $\varphi_*\colon B\to M_{2d}(A[M_{2d}^{S \setminus \{s\}}])$ we get $\tr(\alpha\varphi_*(\bbX^{(s)})) = 0$. Since the trace pairing is nondegenerate, we conclude that $\alpha = 0$.
\end{proof}

\begin{corollary}\label{iso0}
Let $(R,*)$ be an $A$-algebra with involution. Then there is an isomorphism
$$ A[\SpDet^{2d}_{(R,*)}]\xrightarrow{\sim} A[\SpRep^{\square,2d}_{(R,*)}]^{\Sp_{2d}} $$
\end{corollary}

\begin{proof}
The proof is based on that of \cite[Proposition 2.3]{ChHDR}. Let us consider the universal representation
$$ \rho^{u}\colon (R,*)\longrightarrow (M_{2d}(A[\SpRep_{(R,*)}^{\square,2d}]),\mathsf{j}), $$
and the universal symplectic determinant law
$$ (D^{u},P^{u})\colon R\otimes_A A[\SpDet^{2d}_{(R,*)}]\longrightarrow A[\SpDet_{(R,*)}^{2d}]. $$
The symplectic determinant law $\left(\det\circ\rho^{u}),\Pf\circ( \rho^{u}\cdot J)\right)$ induces an $A$-algebra map
\begin{align*}
    \theta\colon A[\SpDet_{(R,*)}^{2d}] \longrightarrow A[\SpRep^{\square,2d}_{(R,*)}]^{\Sp_{2d}}
\end{align*}
sending $T^{u}(r)$ to $\tr(\rho^{u}(r))$. If $R$ is the free $A$-algebra with involution on a set $S$, then we have that $A[\SpRep^{\square,2d}_{(R,*)}]^{\Sp_{2d}}=T_S(2d)$. And so we get using \Cref{generatorsofinv} that $\theta$ is surjective. Therefore, by linear reductivity of $\Sp_{2d, A}$ and by \cite[Proposition 3.4]{wang2022arithmetic}, we get  that $\theta$ is surjective in general. 
\\ Since the $A[\SpDet^{\square,2d}_{(R,*)}]$-algebra $R'\colonequals (A[\SpDet^{\square,2d}_{(R,*)}]\otimes_A R)/\CH(P^{u})$ equipped with the corresponding involution and symplectic determinant law is symplectic Cayley-Hamilton, we get by \Cref{converseCH} that there exists a commutative $A$-algebra $B$, a symplectic representation
$$ \rho\colon (R',*)\longrightarrow (M_{2d}(B),\mathsf{j}), $$
and an injective $A$-algebra morphism $u\colon  A[\SpDet^{\square,2d}_{(R,*)}] \hookrightarrow B$ such that $\det\circ\rho= u\circ D^{u}$ and $\Pf\circ(\rho\cdot J) = u\circ P^{u}$. We get by universality an $A$-algebra morphism $u'\colon A[\SpRep^{2d}_{(R,*)}] \to B$ such that $u'\circ \theta= u$. It follows that $\theta$ is injective and therefore an isomorphism.
\end{proof}

\section{Symplectic determinant laws over Henselian local rings}\label{SpDetHenselian}

We fix a Henselian local ring $A$ with maximal ideal $\mathfrak{m}_A$ and residue field $k$, and we suppose that $2\in A^\times$. Let $\overline{k}$ be an algebraic closure of $k$. If $(R,*)$ is an involutive $A$-algebra, we write $\overline{R}=R/\mathfrak{m}_A R$ which is equipped with the involution induced by $*$. If $(D\colon R\to A, P\colon R^+\to A)$ is a symplectic determinant law, we call $(\overline{D}=D\otimes_A k\colon  \overline{R}\to k, \overline{P}=P\otimes_A k\colon  \overline{R}^+\to  k)$ the \emph{residual symplectic determinant law} of $D$.

Let us recall \cite[Definitions 2.18, 2.19]{MR3444227}, but adapted to our setting:
\begin{definition}
    Let $(R,*)$ be an involutive $A$-algebra and let $(D\colon R\to A, P\colon R^+\to A)$ be a $2d$-dimensional symplectic determinant law. 
    \begin{enum}
        \item[(1)] We say that $(\overline{D},\overline{P})$ is \textit{absolutely irreducible} and $(D,P)$ is \textit{residually absolutely irreducible}, if the (unique up to conjugation) semisimple symplectic representation $(\overline{R}\otimes_k \overline{k},*)\to (M_{2d}(\overline{k}),\mathrm{j})$ with symplectic determinant law $(\overline{D} \otimes_k \overline{k},\overline{P} \otimes_k \overline{k})$ is irreducible as a representation (i.e. after forgetting the involution).
        \item[(2)] We say that $(\overline{D},\overline{P})$ is $\textit{split}$ and $(D,P)$ is \textit{residually split}, if $(\overline{D},\overline{P})$ is the symplectic determinant law associated to a symplectic representation $(\overline{R},*)\rightarrow (M_{2d}(k),\mathrm{j})$.
        \item[(3)] We say that $(\overline{D},\overline{P})$ is $\textit{multiplicity free}$ and $(D,P)$ is \textit{residually multiplicity free}, if $\overline{D}$ is the determinant law associated to a direct sum of pairwise non-isomorphic (after forgetting the involution) absolutely irreducible $k$-linear representations of $R$. 
    \end{enum}
\end{definition}

The goal of this section is to describe the symplectic Cayley-Hamilton algebras over $A$ with residually multiplicity free symplectic determinant laws. To illustrate this, we have the following result in the residually split absolutely irreducible case.

\begin{proposition}\label{henselianlemma}
Let $(R,*,D,P)$ be a symplectic Cayley-Hamilton algebra of degree $2d$ such that $(\overline{D},\overline{P})$ is split and absolutely irreducible. Then there exists an isomorphism of involutive algebras $\rho\colon  (R,\sigma)\xrightarrow{\sim}  (M_{2d}(A),\mathrm{j})$ such that $D=\det\circ \rho$ and $P=\Pf\circ (\rho\cdot J)$.
\end{proposition}

\begin{proof}
By \Cref{symplecticreconstructionthmoverfields}, we have that $(\overline{R}/\ker(\overline D),*)\cong (M_{2d}(k),\mathrm{j})$. For $1\le i,j\le 2d$, let $\epsilon_{ij}\in \overline{R}/\ker(\overline{D})$ be the element corresponding under this isomorphism to the matrix with $1$ at the $(i,j)$ entry and $0$ elsewhere. Since $R$ is integral over $A$ (by \Cref{SCHisfinite}), $A$ is Henselian, and $\Rad(R)=\ker(R\rightarrow \overline{R}/\ker(\overline{D}))$ (by \Cref{radR}), we can apply \cite[Lemma 1.8.2]{BC} to find a $*$-stable family of orthogonal idempotents $E_{ii}$ lifting $\epsilon_{ii}$, with $\sum_{i=1}^{2d}E_{ii}=1$. Note that we necessarily have $E_{ii}^*=E_{i+d,i+d}$ for $1\le i \le d$. By \cite[Chapter III, §4, Exercise 5(c)]{Bourbaki}, we can extend this to find elements $E_{ij}\in R$ lifting the $\epsilon_{ij}$ and satisfying $E_{ij}E_{kl}=\delta_{jk}E_{il}$. 

Let us write $e_i=E_{ii}+E_{i+d,i+d}$ for $1\le i \le d$ so that $e_i^*=e_i$. Note that $\Spec(A)$ is connected since $A$ is local, so we can apply \Cref{SCHidempotent} to get that $(D_{e_i}\colon  e_iRe_i\to A, \ P_{e_i}\colon  e_iR^+e_i\to  A)$ is symplectic Cayley-Hamilton of degree $2d_{e_i}$. Reducing the equality $t^{2d_{e_i}}=D_{e_i}(e_it)=D(e_it+1-e_i)\in A[t]$ modulo $\mathfrak{m}_A$, we get that
\begin{equation*}
    t^{2d_{e_i}}=\overline{D}(t(\epsilon_{ii}+\epsilon_{i+d,i+d})+1-\epsilon_{ii}-\epsilon_{i+d,i+d})=t^2\in k[t]
\end{equation*}
so that $d_{e_i}=1$. Thus for $x\in E_{ii}RE_{ii}$, we get that $x+x^*=P_{e_i}(x+x^*)= P_{e_i}(x+x^*)E_{ii}+ P_{e_i}(x+x^*)E_{i+d,i+d}$. This shows that $E_{ii}RE_{ii}$ and $E_{i+d,i+d}RE_{i+d,i+d}$ are free of rank one over $A$.

Now for $x\in E_{ii}RE_{jj}$, we have that $x=E_{ij}(E_{jj}E_{ji}x)\in AE_{ij}$, and so $R=\oplus_{ij} AE_{ij}\cong M_{2d}(A)$.
\\ Since $A$ is a local ring, every automorphism of $M_{2d}(A)$ is inner (see \cite[Remark 3.4.19]{AGPR}). Therefore there exists an invertible matrix $P\in \GL_{2d}(A)$ such that $(M^*)^{\mathrm{j}}=PMP^{-1}$ for all $M \in M_{2d}(A)$. It follows from the fact that $(E_{ii}^*)^{\mathrm{j}}=E_{ii}$ that we have $P=\diag(\lambda_1,\dots,\lambda_{2d})$ with $\lambda_i\equiv 1 \mod \mathfrak{m}_A$. Since $A$ is Henselian and $2\in A^\times$, there exist elements $\lambda_i'\in A^\times$ such that $\lambda_i'^2=\lambda_i$. Letting $Q=\diag(\lambda_1',\dots,\lambda_d')$, we get an isomorphism of involutive $A$-algebras
\begin{equation*}
    (M_{2d}(A),*)\longrightarrow (M_{2d}(A),\mathrm{j})\colon  M\mapsto QMQ^{-1} 
\end{equation*}
which is what we want.
\end{proof}

To generalize this result to the residually multiplicity-free case, we need to develop the theory of symplectic GMAs analogously to \cite[§1.3]{BC}. Since the definition works more generally, we suppose that $A$ is a commutative algebra with $2\in A^\times$.

\begin{definition}\label{definvolutiveGMAtype} A \emph{symplectic GMA type} $\delta=((I_0, I_1,I_2),\sigma,(d_i)_{i \in I})$ of dimension $2d \in \bbZ_{\geq 0}$ consists of a partition of $I=\{1,\dots,r\}$ into three parts $ I_0\sqcup I_1 \sqcup I_2$, a bijection $\sigma\colon I\to I$, and a sequence of positive integers $d_1, \dots, d_r$, such that
\begin{itemize}
    \item $\sigma^2 = \id_I$,
    \item $\sigma(i)=i$ for all $i\in I_0$,
    \item $\sigma(I_1)=I_2$,
    \item $d_1 + \dots + d_r = 2d$,
    \item $d_i$ is even for all $i \in I_0$ and
    \item $d_{\sigma(i)}=d_i$ for all $i \in I$.
\end{itemize}
\end{definition}

To a symplectic GMA type $\delta$ as above, we associate the following matrix:
$$ J_\delta\colonequals  
\begin{pmatrix} J_{\delta}(1,1) & J_\delta(1,2) & \dots & J_\delta(1,r) \\
    J_\delta(2,1) & J_\delta(2,2) & \dots & J_\delta(2,r) \\
    \vdots & \vdots & \ddots & \vdots \\
   J_\delta(r,1) & J_\delta(r,2) & \dots & J_\delta(r,r) \end{pmatrix}\in M_d(\bbZ)$$
where $J_\delta(i,j)\in M_{d_i,d_j}(\bbZ)$ is defined as follows: if $j\neq \sigma(i)$, then $J_\delta(i,j)=0$, if $i\in I_0$ then $J_\delta(i,i)=J$, if $i\in I_1$ then $J_\delta(i,\sigma(i))=-\id$, and if $i\in I_2$ then $J_\delta(i,\sigma(i))=\id$. We define an involution $*_\delta$ on $M_d(A)$ by setting $M^{*_\delta} \colonequals  J_\delta M^\top J_\delta^{-1}$.

\begin{definition}\label{defsymplecticGMA}
Let $(R,*)$ be an involutive $A$-algebra. A \emph{symplectic GMA structure} $\calE = ((e_i)_{i\in I},(\psi_j)_{j\in I_0\sqcup I_1})$ on $(R,*)$ of type $\delta=((I_0, I_1,I_2),\sigma,(d_i)_{i \in I})$ consists of the following data:
\begin{enum}
    \item[(1)] A family of orthogonal idempotents $e_1,\dots, e_r$ of sum $1$ such that $\forall i\in I$, $e_i^*=e_{\sigma(i)}$.
    \item[(2)] For every $i\in I_0\sqcup I_1$, an $A$-algebra isomorphism $\psi_i\colon  e_iRe_i\xrightarrow{\sim} M_{d_i}(A)$.
\end{enum}
For $i\in I_2$, we define $\psi_i\colon e_iRe_i \to  M_{d_i}(A)$ by $\psi_i\colonequals \top \circ \psi_{\sigma(i)}\circ *$. Moreover, we require the following conditions:
\begin{itemize}
    \item $\mathsf{j}\circ \psi_i=\psi_i\circ *$ for all $i\in I_0$.
    \item The \emph{trace map} $T_{\mathcal{E}}\colon R\to A$ defined by $T_{\mathcal{E}}(x) \colonequals  \sum_{i=1}^r \tr(\psi_i(e_ixe_i))$ satisfies $T_{\mathcal{E}}(xy)=T_{\mathcal{E}}(yx)$ for all $x,y\in R$.
\end{itemize}
The triple $(R, *, \calE)$ is called a \emph{symplectic GMA} of type $\delta$.
\end{definition}

\begin{remark} \Cref{defsymplecticGMA} is designed in a way that for $i \in I_1$, we get that the isomorphism $(\psi_i, \psi_{\sigma(i)}) \colon e_i R e_i \oplus e_{\sigma(i)} R e_{\sigma(i)} \xrightarrow{\sim} M_{d_i}(A) \times M_{d_i}(A)$ is compatible with the swap involution on $M_{d_i}(A) \times M_{d_i}(A)$, i.e., $$(\psi_i, \psi_{\sigma(i)}) \circ * = \mathrm{swap} \circ (\psi_i, \psi_{\sigma(i)}).$$
\end{remark}

Let $(R,*,\mathcal{E})$ be a symplectic $\text{GMA}$ of type $\delta$. For $1\le i\le r, \ 1\le k,l\le d_i$, there is a unique element $E_i^{k,l}\in e_iRe_i$ such that $\psi_i(E_i^{k,l})$ is the elementary matrix of $M_{d_i}(A)$ with unique nonzero coefficient at row $k$ and column $l$. Define $E_i \colonequals  E_i^{1,1}$. Now set
$$ \mathcal{A}_{i,j}\colonequals E_iRE_j $$
with $T$ inducing an isomorphism $\mathcal{A}_{i,i}\xrightarrow{\sim} A$ and we will implicitly identify $\calA_{i,i}$ with $A$.

For each triple $1\le i,j,k\le r$, the multiplication induces a map
$$ \varphi_{i,j,k}\colon \mathcal{A}_{i,j}\otimes \mathcal{A}_{j,k}\longrightarrow \mathcal{A}_{i,k} $$
and these satisfy the relations (UNIT), (ASSO) and (COM) of \cite[§1.3.2]{BC}.

If $i\in I_0$, let $p_i\colonequals \psi_i^{-1}(J)$ for $J\in M_{d_i}(A)$, otherwise let $p_i\colonequals e_i$ if $i\in I_1$ and $p_i\colonequals -e_i$ if $i\in I_2$. This is an invertible element of the algebra $e_iRe_i$ and by abuse of notation we denote its inverse in this algebra by $p_i^{-1}$. Then for all $i\in I$, we have $E_i^*=p_iE_{\sigma(i)}p_i^{-1},$ and we can define morphisms
\begin{align*}
    \tau_{i,j}\colon  \mathcal{A}_{i,j}&\longrightarrow  \mathcal{A}_{\sigma(j),\sigma(i)}
    \\ x &\mapsto p_{\sigma(j)}^{-1}x^*p_{\sigma(i)}
\end{align*}
which have the following properties:
\begin{enum}
    \item For all $i,j\in I$, the $A$-linear endomorphism $\tau_{\sigma(j),\sigma(i)}\circ \tau_{i,j}$ of $\mathcal{A}_{i,j}$ is the identity.
    \item For all $i,j\in I$, $\tau_{i,j}$ is an isomorphism of $A$-modules.
    \item For all $i,j,k\in I$, $x\in \mathcal{A}_{i,j}$, and $y\in \mathcal{A}_{j,k}$, we have $\tau_{j,k}(y)\tau_{i,j}(x)=\tau_{i,k}(xy)$ in $\mathcal{A}_{\sigma(k),\sigma(i)}$.
\end{enum}

\begin{example}\label{stdsymplecticGMA}
    We describe in this example what we will call a \emph{standard symplectic GMA} of type $\delta$.
    Let $B$ be a commutative $A$-algebra. Let $A_{i,j}$, $i,j\in I$ be a family of $A$-submodules of $B$ satisfying the following properties:
    $$ \text{ for all }i,j,k\in I, \quad A_{i,i}=A, \quad A_{i,j}=A_{\sigma(j),\sigma(i)}, \quad A_{i,j}A_{j,k}\subseteq A_{i,k} $$
    Then the $A$-submodule
    \begin{align*}
         R\colonequals \begin{pmatrix} M_{d_1}(A_{1,1}) & M_{d_1,d_2}(A_{1,2}) & \dots & M_{d_1,d_r}(A_{1,r}) \\
    M_{d_2,d_1}(A_{2,1}) & M_{d_2}(A_{2,2}) & \dots & M_{d_2,d_r}(A_{2,r}) \\
    \vdots & \vdots & \ddots & \vdots \\
    M_{d_r,d_1}(A_{r,1}) & M_{d_r,d_2}(A_{r,2}) & \dots & M_{d_r}(A_{r,r}) \end{pmatrix}\subseteq M_{2d}(B)
    \end{align*}
    equipped with the involution $*$ defined to be the restriction of $*_\delta$ on $M_{2d}(B)$ to $R$ is an $A$-subalgebra with involution. Following \cite[Example 1.3.4]{BC}, we can equip $(R,*)$ with the structure of a symplectic GMA.
\end{example}
By \cite[§1.3.2]{BC}, we have an isomorphism $e_iRE_i\otimes \mathcal{A}_{i,j}\otimes E_jRe_j\xrightarrow{\sim}e_iRe_j$ such that $\psi_i$ and $\psi_j$ induce a canonical identification $e_iRe_j= M_{d_i,d_j}(\mathcal{A}_{i,j})$. The involution on $R$ induces isomorphisms of $A$-modules
\begin{align*}
    *\colon M_{d_i,d_j}(\mathcal{A}_{i,j}) &\longrightarrow M_{d_j,d_i}(\mathcal{A}_{\sigma(j),\sigma(i)})
    \\ M &\mapsto J_{\delta}(\sigma(j),j)\cdot\tau_{i,j}(M)^\top\cdot J_{\delta}(i,\sigma(i))^{-1}
\end{align*}
The maps $\tau_{i,j}$ for $i,j\in I$ induce a map of $A$-modules
\begin{align*}
    \tau\colon \begin{pmatrix} M_{d_1}(\mathcal{A}_{1,1}) & \dots & M_{d_1,d_r}(\mathcal{A}_{1,r}) \\
    M_{d_2,d_1}(\mathcal{A}_{2,1})  & \dots & M_{d_2,d_r}(\mathcal{A}_{2,r}) \\
     \vdots & \ddots & \vdots \\
    M_{d_r,d_1}(\mathcal{A}_{r,1})  & \dots & M_{d_r}(\mathcal{A}_{r,r}) \end{pmatrix}\rightarrow  \begin{pmatrix} M_{d_1}(\mathcal{A}_{\sigma(1),\sigma(1)}) &  \dots & M_{d_1,d_r}(\mathcal{A}_{\sigma(r),\sigma(1)}) \\
    M_{d_2,d_1}(\mathcal{A}_{\sigma(1),\sigma(2)}) & \dots & M_{d_2,d_r}(\mathcal{A}_{\sigma(r),\sigma(2)}) \\
     \vdots & \ddots & \vdots \\
    M_{d_r,d_1}(\mathcal{A}_{\sigma(1),\sigma(r)}) & \dots & M_{d_r}(\mathcal{A}_{\sigma(r),\sigma(r)}) \end{pmatrix}
\end{align*}
 Therefore we get an isomorphism of involutive $A$-algebras
$$ (R,*) \cong\left( \begin{pmatrix} M_{d_1}(\mathcal{A}_{1,1}) & M_{d_1,d_2}(\mathcal{A}_{1,2}) & \dots & M_{d_1,d_r}(\mathcal{A}_{1,r}) \\
    M_{d_2,d_1}(\mathcal{A}_{2,1}) & M_{d_2}(\mathcal{A}_{2,2}) & \dots & M_{d_2,d_r}(\mathcal{A}_{2,r}) \\
    \vdots & \vdots & \ddots & \vdots \\
    M_{d_r,d_1}(\mathcal{A}_{r,1}) & M_{d_r,d_2}(\mathcal{A}_{r,2}) & \dots & M_{d_r}(\mathcal{A}_{r,r}) \end{pmatrix},*_\delta\right) $$
    where $M^{*_\delta}\colonequals J_\delta\cdot \tau(M)^\top\cdot J_\delta^{-1}$.

\begin{definition}
    Let $B$ be a commutative $A$-algebra and let $(R,*, \calE)$ be an symplectic GMA of type $\delta$.
    A $*$-representation $\rho \colon (R,*) \to (M_d(B),*_\delta)$ is said to be \emph{adapted} to $\calE$, if its restriction to the subalgebra $\bigoplus_{i=1}^r e_i R e_i$ is the composite of $\bigoplus_{i=1}^r \psi_i \colon  \bigoplus_{i=1}^r e_i R e_i \to \bigoplus_{i=1}^r M_{d_i}(A)$ with the diagonal map
    $$ \bigoplus_{i=1}^r M_{d_i}(A) \longrightarrow M_{2d}(B).$$
    We define $\Rep^\square_{\Ad}(R,*,\mathcal{E}) \colon  \CAlg_A \to \Set$ to be the functor associating to an $A$-algebra $B$ the set of adapted representations of $(R,*,\mathcal{E})$ over $B$.
\end{definition}
\begin{remark} By a change of basis of $B^d$, we can achieve that the involution on $M_d(B)$ is the standard one. Better so, we can always change the symplectic GMA structure $\calE$ on $R$ in a way that $*_{\delta}$ will be the standard involution.
\end{remark}
\begin{proposition}\label{representabilityofadrep}
The functor $\Rep_{\Ad}^\square(R,*,\mathcal{E})$ is represented by a commutative $A$-algebra that we denote by $A[\Rep_{\Ad}^\square(R,*,\mathcal{E})]$.
\end{proposition}
Representability follows from Freyd's adjoint functor theorem.
The purpose of the following proof is to give an explicit construction of $A[\Rep_{\Ad}^\square(R,*,\mathcal{E})]$ and to introduce notations, that we will use in the proof of \Cref{solutionembedding}.

\begin{proof}
By \cite[Proposition 1.3.9]{BC}, the datum of an adapted representation $\rho\colon  R\to M_d(B)$ is equivalent to the datum of a family of functions $(f_{i,j}\colon  \mathcal{A}_{i,j}\to  B)_{i,j\in I}$ satisfying the following conditions:
\begin{enum}
    \item[(1)] $f_{i,i}$ is the structure map $A\to B$.
    \item[(2)] The product on $B$ is compatible with the $\varphi_{i,j,k}$, i.e. $f_{i,k}\circ \varphi_{i,j,k}=f_{i,j}\cdot f_{j,k}$. 
\end{enum}
To further ask that the representation $\rho$ respects the involution is equivalent to the following extra condition:
 \begin{enum}
     \item[(3)] $f_{\sigma(j),\sigma(i)}\circ \tau_{i,j}=f_{i,j}$. 
 \end{enum}
Therefore, $\Rep_{\Ad}^\square(R,*,\mathcal{E})$ is represented by the quotient of the $A$-algebra
\begin{align}
    \mathcal{B}\colonequals \Sym_A\left( \bigoplus_{1\le i\neq j\le r} \mathcal{A}_{i,j}\right) \label{symmetricalgebraB}
\end{align} 
by the ideal $J$ generated by $b\odot c-\varphi(b\otimes c)$ for all $\varphi=\varphi_{i,j,k}$, $b\in \mathcal{A}_{i,j}$, $c\in \mathcal{A}_{j,k}$; and by $a-\tau_{i,j}(a)$ for $a\in \mathcal{A}_{i,j}$.
\end{proof}

Given a symplectic GMA $(R,*,\mathcal{E})$ of type $\delta$, we can associate to it a canonical $2d$-dimentional Cayley-Hamilton determinant law $D_{\mathcal{E}}\colon  R\to  A$ such that $\Lambda_{1,A}^{D_{\mathcal{E}}}=T_{\mathcal{E}}$. From the formula defining $D_{\mathcal{E}}$, we can see that $D_{\mathcal{E}}=\det\circ \rho_{\Ad}^{u}$, where $\rho^{u}_{\Ad}\colon  (R,*)\to (M_{2d}(A[\Rep_{\Ad}^\square(R,*,\mathcal{E})]),*_\delta)$ is the universal adapted representation. We can therefore define the polynomial law $P_{\mathcal{E}}\colonequals \Pf\circ (\rho_{\Ad}^{u}\cdot J_\delta)\colon R^+\to A$ so that $(D_{\mathcal{E}},P_{\mathcal{E}})$ is a $2d$-dimensional symplectic determinant law. In what follows, we will give a necessary and sufficient condition on $(R,*)$ so that $(R,*,D_{\mathcal{E}},P_{\mathcal{E}})$ is symplectic Cayley-Hamilton. The following Proposition also provides a solution, in this context, to the embedding problem discussed in \Cref{CH0}.

\begin{proposition}[Solution to the embedding problem]\label{solutionembedding}
Let $(R,*,\mathcal{E})$ be a symplectic $\text{GMA}$ of type $\delta$ and suppose that for all $i\in I_1\sqcup I_2$ and all $x\in \mathcal{A}_{i,\sigma(i)}$ $x^*=-x$. Then the universal adapted representation
\begin{equation*}
    \rho^{u}_{\Ad}\colon  (R,*)\longrightarrow (M_{2d}(A[\Rep_{\Ad}^\square(R,*,\mathcal{E})]),*_\delta)
\end{equation*} is injective.
\end{proposition}
\begin{proof} The reader is encouraged to refer to the proof of \cite[Proposition 1.3.13]{BC}, which serves as the inspiration for our own. However, our case presents more complexity due to the need to account for the involution, resulting in the failure of the constructions in loc.cit. 

Consider the set $\Omega\colonequals \{(i,j)\in I^2 \ | \ i\neq j\}$, where for $x=(i',j')\in \Omega$, we write $i(x)\colonequals i'$ and $j(x)\colonequals j'$. We identify $\mathbb{N}^{\Omega}$ with the set of oriented graphs with set of vertices $I$, where we do not allow edges from a vertex to itself but allow multiples edges between two different vertices. This is done by associating to $t=(t_{i,j})_{(i,j)\in \Omega}$ the graph with $t_{i,j}$ edges from $i$ to $j$. We write $t(i,j)$ with the graph having a unique arrow from $i$ to $j$. 
\\ We have an additive map
\begin{equation*}\deg\colon  \mathbb{N}^\Omega\longrightarrow \mathbb{Z}^I, (t_{i,j})_{(i,j)\in \Omega}\mapsto \big(\sum_{j\neq i} t_{j,i}-t_{i,j}\big)_{i\in I}, \end{equation*}
and the involution $\sigma\colon I\to I$ induces a map of additive monoids
\begin{equation*}
    \sigma\colon  \mathbb{N}^\Omega\longrightarrow \mathbb{N}^\Omega,\ (t_{i,j})_{(i,j)\in \Omega}\mapsto (t_{\sigma(j),\sigma(i)})_{(i,j)\in \Omega}
\end{equation*} 
A sequence $\gamma=(x_1,\dots,x_s)$ of elements of $\Omega$ is called a \textit{path} from $i(x_1)$ to $j(x_s)$ if for all $k\in \{1,\dots,s-1\}$, $j(x_k)=i(x_{k+1})$. In this case, we set $\mathcal{A}_{\gamma}\colonequals \mathcal{A}_{i(x_1),j(x_1)}\otimes \cdots \otimes \mathcal{A}_{i(x_s),j(x_s)}$, and we have a canonical contraction map $\varphi_\gamma\colon  \mathcal{A}_\gamma\to  \mathcal{A}_{i(x_1),j(x_s)}$. We say that $\gamma$ is a cycle if $i(x_1)=j(x_s)$. 
\\ If $(i,j)\in \Omega$, an \textit{extended path} from $i$ to $j$ consists of a sequence of paths $\Gamma=(c_1,\dots,c_r,\gamma)$, where the $c_k$'s are cycles and $\gamma$ is a path from $i$ to $j$. To an extended path $\Gamma=(c_1,\dots,c_r,\gamma)$, we can associate a graph $t(\Gamma)\colonequals (\Gamma_{i',j'})\in \mathbb{N}^\Omega$, where $\Gamma_{i',j'}$ is the number of times that $(i',j')$ appears in the $c_k$'s or $\gamma$. We remark that $\deg(t(\Gamma))=\deg(t(i,j))$.
\\ Recall from the proof of \Cref{representabilityofadrep} that the representing ring  $A[\Rep_{\Ad}^\square(R,*,\mathcal{E})]$ is equal to $\mathcal{B}/J$ with $\mathcal{B}=\Sym_A\left( \bigoplus_{1\le i\neq j\le r} \mathcal{A}_{i,j}\right)$ and $J$ is the ideal generated by the elements of the form $ b\odot c-\varphi(b\odot c) $ for $\varphi=\varphi_{i',j',k'}, b\in \mathcal{A}_{i',j'},c\in \mathcal{A}_{j',k'}$, and by elements of the form $a-\tau_{i',j'}(a)$ for $a\in \mathcal{A}_{i',j'}$. The latter $A$-algebra has a $\mathbb{N}^\Omega$-graduation $\mathcal{B}=\bigoplus_{t\in \mathbb{N}^\Omega}\mathcal{B}_t$, where $\mathcal{B}_t=\bigodot_{(i',j')\in\Omega}\Sym_A^{t_{i',j'}}(\mathcal{A}_{i',j'})$.

Now let us fix some $(i,j)\in\Omega$. For every $t\in \mathbb{N}^\Omega$ with $\deg(t)=\deg(t(i,j))$, there is an $A$-linear map
\begin{equation*}
    \overline{\varphi}_t\colon  \mathcal{B}_t \longrightarrow \mathcal{A}_{i,j}
\end{equation*}
defined in the proof of \cite[Proposition 1.3.13]{BC} as follows: consider an extended path $\Gamma=(c_1,\dots,c_r,\gamma)$ from $i$ to $j$ with $t(\Gamma)=t$ (this extended path exists by \cite[Lemma 1.3.14]{BC}), then the following $A$-linear map:
\begin{align*}
    \varphi_\Gamma\colon  \mathcal{A}_{c_1}\otimes \cdots \otimes \mathcal{A}_{c_r}\otimes \mathcal{A}_\gamma &\longrightarrow \mathcal{A}_{i,j}
    \\ \left(\otimes_{k=1}^r x_k\right)\otimes y &\mapsto \left(\prod_{k=1}^r \varphi_{c_k}(x_k)\right)\varphi_\gamma(y)
\end{align*}
factors through $\mathcal{B}_t$ and the resulting map does not depend on the choice of $\Gamma$.
\\ We introduce an equivalence relation on the set of oriented graphs $\mathbb{N}^\Omega$ by setting $t\sim t'$ if we can write $t=t_1+t_2$ and $t'=\sigma(t_1)+t_2$ for some $t_1,t_2\in \mathbb{N}^\Omega$, in which case we can define an $A$-linear map
\begin{equation*}
     \tau(t',t)= \odot_{i',j'} h_{i',j'}\colon \mathcal{B}_t \longrightarrow \mathcal{B}_{t'}
\end{equation*}
where $h_{i',j'}=\id$ if $(i',j')\in t_2$ and $h_{i',j'}=\tau_{i',j'}$ if $(i',j')\in t_1$. Note that if $t'\sim t''$, then $\tau(t'',t)=\tau(t'',t')\circ \tau(t',t)$. If moreover we have that $\deg(t')=\deg(t(i,j))$, then we can define an $A$-linear map
\begin{equation*}
    \overline{\varphi}_t=\overline{\varphi}_{t'}\circ \tau(t',t)\colon  \mathcal{B}_t\longrightarrow \mathcal{A}_{i,j}
\end{equation*}
To justify our notation, let us prove that $\overline{\varphi}_t$ does not depend on the representative $t'$. For this, we need to show that if $t'\sim t''$ with $\deg(t')=\deg(t'')=\deg(t(i,j))$, then $\overline{\varphi}_{t'}=\overline{\varphi}_{t''}\circ \tau(t'',t')$. 
\\ Let us set $t'=t'_1+t'_2$ and $t''=\sigma(t'_1)+t'_2$. The fact that $\deg(t'_1)=\deg(\sigma(t'_1))$ allows us to write $t'_1=t'_{1,1}+\dots +t'_{1,r}$, with $t'_{1,k}$ having one of these four forms:
\begin{enum}
    \item[(I)] $t(i',\sigma(i'))$ for $i'\in I_1\sqcup I_2$.
    \item[(II)] $t(i',j')+t(j',i')$ for $i',j'\in I_0$.
    \item[(III)] $t(i',j')+t(\sigma(j'),i')$ for $i'\in I_0$ and $j'\in I_1\sqcup I_2$.
    \item[(IV)] $t(i',j')+t(\sigma(j'),\sigma(i'))$ for $i',j'\in I_1\sqcup I_2$ and $j'\neq \sigma(i')$. This concludes the argument. 
\end{enum}
Arguing by induction, we can suppose that $t'_1=t'_{1,1}$. The case (I) is trivial since by our hypothesis $\mathcal{A}_{i',j'}\cap R^+=0$, which means that $\tau_{i',\sigma(i')}$ is the identity on $\mathcal{A}_{i',\sigma(i')}$. The case (II) does not change $t'$, and by \cite[Lemma 1.3.14(ii)]{BC}, we can find an extended path $\Gamma'$ such that $t(\Gamma')=t'$ and $\Gamma'$ has a path containing $((i',j'),(j',i'))$ as a subpath. But we know that for $a\in \mathcal{A}_{i',j'}$ and $b\in \mathcal{A}_{j',i'}$, $\tau(ab)=\tau(b)\tau(a)=\tau(a)\tau(b)=ab$. Therefore, we get the desired equality $\overline{\varphi}_{t'}=\overline{\varphi}_{t''}\circ \tau(t'',t')$. For the case (III), same as before, we can find extended paths $\Gamma'$ and $\Gamma''$ such that $t(\Gamma')=t'$, and $t(\Gamma'')=t''$, and such that they have paths that respectively contain $((\sigma(j'),i'),(i',j'))$ and $((\sigma(j'),\sigma(i')),(\sigma(i'),j'))$ as subpaths. Now for $a\in \mathcal{A}_{\sigma(j'),i'}$ and $b\in \mathcal{A}_{i',j'}$, we have that $ab=\tau(ab)=\tau(b)\tau(a)$ which gives us the desired conclusion. Finally for the case (IV), we have that $t''=t'$, so let us consider an extende path $\Gamma'=(c_1,\dots,c_r,\gamma)$ from $i$ to $j$ such that $t(\Gamma')=t'$. If $(i',j')$ and $(\sigma(j'),\sigma(i'))$ are in the same path then, we can suppose that this path contains $((i',j'),(j',\sigma(j')),(\sigma(j'),\sigma(i')))$ as a subpath. But for $a\in \mathcal{A}_{i',j'}$, $b\in \mathcal{A}_{j',\sigma(j')}$ and $c\in \mathcal{A}_{\sigma(j'),\sigma(i')}$, we have $abc=\tau(abc)=\tau(c)\tau(b)\tau(a)=\tau(c)b\tau(a)$.
\\ If $(i',j')$ and $(\sigma(j'),\sigma(i'))$ are in different paths, we can suppose that $(i',j')$ is in $c_1$,  and even that $c_1=((i',j'),(j',i'))$. Then for $a\in \mathcal{A}_{i',j'}$, $b\in \mathcal{A}_{j',i'}$ and $a'\in \mathcal{A}_{\sigma(j'),\sigma(i')}$, we have that $a'ab=a'\tau(ab)=a'\tau(b)\tau(a)=\tau(a'\tau(b))\tau(a)=b\tau(a')\tau(a)$. This allows us to obtain the equality in this case.

On the other hand, if $t\in \mathbb{N}^\Omega$ is not equivalent to a path of degree $\deg(t(i,j))$, we set the map $\overline{\varphi}_t\colon  \mathcal{B}_t\to  \mathcal{A}_{i,j}$ to be the zero map. Therefore, we get an $A$-linear map
\begin{equation*}
    \overline{\varphi}=\oplus_{t} \overline{\varphi}_t\colon  \mathcal{B}=\oplus_{t} \mathcal{B}_t \longrightarrow \mathcal{A}_{i,j}
\end{equation*}
It is immediate from the definition that $\overline{\varphi}$ vanishes on the elements of the form  $f\odot(a-\tau_{\sigma(j'),\sigma(i')}(a))$ for some $a\in \mathcal{A}_{i',j'}$ and $f\in \mathcal{B}$ (by $A$-linearity we can suppose that $f\in \mathcal{B}_t$ for some $t\in \mathbb{N}^\Omega$). Now let us show that $\overline{\varphi}$ vanishes on elements of the form $ f( b\odot c-\varphi(b\otimes c)) $ for $\varphi=\varphi_{i',j',k'}, b\in \mathcal{A}_{i',j'},c\in \mathcal{A}_{j',k'},f\in \mathcal{B}_t$. Let us assume that $t+t(i',j')+t(j',k')=t_1+t_2$ with $\deg(\sigma(t_1)+t_2)=\deg(t(i,j))$. It suffices to find a graph $t'\sim t+t(i',j')+t(j',k')=t'_1+t'_2$ such that $t'=\sigma(t_1')+t_2'$, $\deg(t')=\deg(t(i,j))$, and $t(i',j')+t(j',k')$ either lies entirely in $t'_1$ or in $t'_2$. For then we can invoke the same argument as in the end of the proof of \cite[Proposition 1.3.13]{BC} to show the vanishing. So let $\Gamma=(c_1,\dots,c_r,\gamma)$ be an extended graph such that $t(\Gamma)=\sigma(t_1)+t_2$. Up to taking $t'=\sigma(t_1)+t_2$, we can suppose that $(i',j')\in t_1$ and $(j',k')\in t_2$. If either $(\sigma(j'),\sigma(i'))$ or $(j',k')$ lies in a cycle $c_k$, then we can take $t'$ to be  $\sigma(t_1)+t_2-c_k+\sigma(c_k)$. Hence we can assume that both $(\sigma(j'),\sigma(i'))$ and $(j',k')$ lie in $\gamma$, and that $\gamma=(\gamma_1,(\sigma(j'),\sigma(i')),\gamma_2,(j',k'),\gamma_3)$ for some paths $\gamma_1,\gamma_2,\gamma_3$. And so if we take $\gamma'=(\gamma_1,\sigma(\gamma_2),(i',j'),(j',k'),\gamma_3)$ and $\Gamma'=(c_1,\dots,c_r,\gamma')$, then $t'=t(\Gamma')$ works.

Finally, we have shown that $\overline{\varphi}\colon \mathcal{B}\rightarrow \mathcal{A}_{i,j}$ descends to an $A$-linear map $\varphi\colon A[\Rep_{\Ad}^\square(R,*,\mathcal{E})]\rightarrow \mathcal{A}_{i,j}$ which can easily be checked to be a section of the map $f_{i,j}\colon \mathcal{A}_{i,j}\rightarrow A[\Rep_{\Ad}^\square(R,*,\mathcal{E})]$ defined in the proof of \Cref{representabilityofadrep}. This gives the injectivity of the universal adapted representation $\rho^{u}_{\Ad}$.
\end{proof}

\begin{corollary}\label{sCHGMA}
    Let $(R,*,\mathcal{E})$ be a symplectic $\text{GMA}$ of type $\delta$. Then $(R,*,D_{\mathcal{E}},P_{\mathcal{E}})$ is symplectic Cayley-Hamilton if and only if for all $i\in I_1\sqcup I_2$ and all $x\in \mathcal{A}_{i,\sigma(i)}$,  we have $x^*=-x$.
\end{corollary}
\begin{proof}
 First, suppose that $i\in I_1\sqcup I_2$ and all $x\in \mathcal{A}_{i,\sigma(i)}x^*=-x$. Then the result follows from the fact that  $\rho_{\Ad}^{u}$ remains injective after every base extension $\otimes_A B$, since the maps $f_{i,j}\colon \mathcal{A}_{i,j}\rightarrow A[\Rep_{\Ad}^\square(R,*,\mathcal{E})]$ are $A$-split injections (by the proof of \Cref{solutionembedding}).
\\Conversely suppose that $(R,*,D_{\mathcal{E}},P_{\mathcal{E}})$ is symplectic Cayley-Hamilton. We can need to show that for all $i\in I_1\sqcup I_2$, the space $\mathcal{A}_{i,\sigma(i)}$ consists only of antisymmetric elements. Let us assume that $i=1$ and $\sigma(i) = 2$. We have the direct sum decomposition $\calA_{1,2}=(\calA_{1,2} \cap R^+)\oplus (\calA_{1,2} \cap R^-)$. Assume $z \in \calA_{1,2} \cap R^+$ and let $A[\lambda_1,\dots,\lambda_{d_1}]$ be a polynomial ring over $A$ generated by the variables $\lambda_1,\dots,\lambda_{d_1}$. Then
$$ x=\lambda_1 (E^{1,1}_1+E^{1,1}_2)+ \cdots+ \lambda_{d_1}(E_1^{d_1,d_1}+E_2^{d_1,d_1})+ z \in (R\otimes_A A[\lambda_1,\dots,\lambda_{d_1}])^+$$
$$ \chi^D(x,t)=(t-\lambda_1)^2\cdots (t-\lambda_{d_1})^2 $$
$$ \chi^P(x,t)=(t-\lambda_1)\cdots (t-\lambda_{d_1}) $$
so $\chi^P(x,x)=(\lambda_1-\lambda_2)\cdots(\lambda_1-\lambda_{d_1})z = 0$ in $R\otimes_A A[\lambda_1,\dots,\lambda_{d_1}]$. We have proved that $\mathcal{A}_{1,2}\cap R^+=0$.
\end{proof}

We now have all the tools to describe residually multiplicity free symplectic Cayley-Hamilton algebras over Henselian local rings.

\begin{theorem}\label{hensCH}
   Suppose that $A$ is a Henselian local ring, and let $(R,*,D,P)$ be a finitely generated $2d$-dimensional symplectic Cayley-Hamilton $A$-algebra with involution. If $D$ is residually multiplicity-free, then $(R,*)$ admits a symplectic $\text{GMA}$ structure $\mathcal{E}$ such that $(D,P)=(D_{\mathcal{E}},P_{\mathcal{E}})$.
\end{theorem}
\begin{proof}
    By our assumptions and \Cref{invAW}, we can write $\overline R/\ker(\overline{D})\cong \prod_{i\in I} M_{d_i}(k)$ for some set $I=\{1,\dots,r\}$. So we can choose central orthogonal idempotents $(\epsilon_i)_{i \in I}$, corresponding to this decomposition, such that $(\overline D_{\epsilon_i},\overline{P}_{\epsilon_i})$ is a $2d_i$-dimensional split and absolutely irreducible symplectic determinant law. Let $\sigma\colon I\to I$ be the unique bijection satisfying $\epsilon_i^*=\epsilon_{\sigma(i)}$ for all $i\in I$. Let $I_0\subseteq I$ be the subset of $\sigma$-fixed indices, let $I_1$ be a system of representatives of $\sigma$-orbits in $I\setminus I_0$, and let $I_2\colonequals I\setminus(I_0\sqcup I_1)$.
\\ We apply \cite[Lemma 1.8.2]{BC} to $\Rad(R)= \ker(R \to \overline R/\ker(\overline D))$ (by \Cref{radR}) to find a $*$-stable family $(e_i)_{1 \leq i \leq r}$ of orthogonal idempotents lifting $(\epsilon_i)_{1 \leq i \leq r}$. We necessarily have $e_i^* = e_{\sigma(i)}$ for all $1 \leq i \leq r$. Since $A$ is local, it is connected, and we get by \Cref{SCHidempotent} symplectic Cayley-Hamilton determinant laws $(D_{e_i},P_{e_i})$  of dimension $2d_i' \leq 2d$. We can verify that $2d_1' + \dots + 2d_r' = 2d$.
        By construction, we have that $\overline{e_i R e_i} = \epsilon_i \overline R \epsilon_i$, so the reduction of $D_{e_i}$ coincides with $\overline D_{\epsilon_i}$ which gives that $d_i=d'_i$. We obtain from \cite[Lemma 2.4 (4)]{MR3444227} that $e_1 + \dots + e_r = 1$.
\\ If $i = \sigma(i)$, we obtain from \Cref{henselianlemma} an isomorphism $\psi_i \colon e_i R e_i \eqto M_{d_i}(A)$, such that $\tr \circ \psi_i = \Lambda^{D_{e_i}}_{1,A}$ and $\mathrm j \circ \psi_i = \psi_i \circ *$.
        If $i < \sigma(i)$, we obtain from \cite[Theorem 2.22 (i)]{MR3444227} an isomorphism $\psi_i \colon e_i R e_i \eqto M_{d_i}(A)$, such that $\tr \circ \psi_i = \Lambda^{D_{e_i}}_{1,A}$ and may define $\psi_{\sigma(i)} \colonequals  \top \circ \psi_i \circ *$. This defines a symplectic \text{GMA} structure $\calE = ((e_i)_{i\in I},(\psi_j)_{j\in I_0\sqcup I_1})$ on $(R,*)$ of type $\delta=((I_0, I_1,I_2),\sigma,(d_i)_{i \in I})$. Finally, we easily check that $T_{\mathcal{E}}=\Lambda_{1,A}^D$ which give that $(D_\mathcal{E},P_{\mathcal{E}})=(D,P)$ by \cite[Proposition 1.27]{MR3444227}.
\end{proof}

\section{Comparison with the GIT quotient}\label{GITcomp}
Throughout \Cref{GITcomp}, we suppose that $A$ is  Noetherian, and consider a finitely generated involutive $A$-algebra $(R,*)$. By \cite[Theorem 9.1.4]{Alper2014}, the canonical map $[\SpRep^{\square, 2d}_{(R,*)}/\Sp_{2d}] \to \SpRep^{\square, 2d}_{(R,*)} \sslash \Sp_{2d}$ is an adequate moduli space.
Since the canonical map $[\SpRep^{\square, 2d}_{(R,*)}/\Sp_{2d}] \to \SpRep^{2d}_{(R,*)}$ is an equivalence of stacks (\Cref{equivofstk}), the map $\phi \colon \SpRep^{2d}_{(R,*)} \to \SpRep^{\square, 2d}_{(R,*)} \sslash \Sp_{2d}$ is an adequate moduli space as well.
By \cite[Theorem 6.3.3]{Alper2014}, $\SpRep^{\square, 2d}_{(R,*)} \sslash \Sp_{2d}$ is of finite presentation over $A$.
\\ The map $\psi^{\square} \colon  \SpRep^{\square, 2d}_{(R,*)} \to \SpDet^{2d}_{(R,*)}$ given by mapping a representation to its symplectic determinant law factors over the stack quotient and thus through a map $\psi \colon \SpRep^{2d}_{(R,*)} \to \SpDet^{2d}_{(R,*)}$, which in turn factors through the adequate moduli space $\phi$. We obtain a commutative diagram
\begin{equation}\label{bigGITdiagram}
    \begin{tikzcd}
        \SpRep^{\square, 2d}_{(R,*)} \ar[r]  & \SpRep^{2d}_{(R,*)} \arrow[d, "\phi"] \arrow[r] \arrow[dr, "\psi"] & \overline{\SpRep}_{(R,*)}^{2d}  \arrow[d] \\ & \SpRep^{\square, 2d}_{(R,*)} \sslash \Sp_{2d} \arrow[r,"\nu"] & \SpDet^{2d}_{(R,*)} 
    \end{tikzcd}
\end{equation}

\subsection{Adequate homeomorphism}

Recall that a morphism of schemes $f\colon X\to Y$ is a \emph{universal homeomorphism} if $f_{Y'}\colon X\times_Y Y'\to Y'$ is a homeomorphism of topological spaces for every morphism of schemes $Y'\to Y$. By \cite[Corollaire 18.12.11]{PMIHES_1967__32__5_0}, $f$ is a universal homeomorphism if and only if it is integral, universally injective, and surjective.

An \emph{adequate homeomorphism} is a universal homeomorphism which is a local isomorphism at all points with residue field of characteristic $0$ (see \cite[Definition 3.3.1]{Alper2014}). By \cite[Proposition 3.3.5]{Alper2014}, for a morphism of rings $A\rightarrow B$  of finite type, the following are equivalent:
\begin{enum}
    \item[(1)] The morphism $\Spec(B)\rightarrow \Spec(A)$ is an adequate homeomorphism.
    \item[(2)] The ideal $\ker(A\rightarrow B)$ is locally nilpotent (i.e., every element is nilpotent), $\ker(A\rightarrow B)\otimes _{\mathbb{Z}}\mathbb{Q}=0$, and for all $A$-algebra $A'$ and $b'\in B\otimes_A A'$ there exists $N>0$ and $a'\in A'$ such that $a'\mapsto b'^N$.
\end{enum}

\begin{theorem}\label{adequate} $\nu$ is a finite adequate homeomorphism.
\end{theorem}

We follow closely the structure of the proof of \cite[Theorem 2.20]{CWE}.

\begin{proof} We first observe, that $\nu$ is a bijection on geometric points, from which it follows that $\nu$ is surjective and universally injective (see \cite[3.5.5]{PMIHES_1960__4__5_0}). Indeed, the geometric points of $\SpRep^{\square, 2d}_{(R,*)} \sslash \Sp_{2d}$ are in bijective correspondence with conjugacy classes of semisimple symplectic representations of $(R,*)$ on $2d$-dimensional vector spaces by \cite[Theorem 15.4]{Pro}; the theorem is stated in characteristic zero but the proof works in arbitrary characteristic.
By \Cref{reconsalgcl} the geometric points of $\SpDet^{2d}_{(R,*)}$ are in bijection with conjugacy classes of semisimple representations as well.

Note that $\nu$ is of finite presentation, so to show that $\nu$ is an isomorphism in neighborhoods of characteristic zero points it suffices (by \cite[Remark 3.3.2]{Alper2014}) to show that $\nu$ induces an isomorphism of local rings at characteristic zero points. These local rings also arise as local rings of the base extension to $\bbQ$, so it is sufficient to show that $\nu \otimes \bbQ$ is an isomorphism, but this is the content of \Cref{iso0}.
Hence, it remains to show that $\nu$ is integral and by \stackcite{01WM}, it suffices to show that $\nu$ is universally closed.

We will apply the valuative criterion for universally closed morphisms in the version of \cite[Remarques 7.3.9 (i)]{PMIHES_1961__8__5_0}. So let $B$ be a complete discrete valuation ring with an algebraically closed residue field and fraction field $K$.
We will show that given a diagram of $A$-schemes
\begin{center}
    \begin{tikzcd}
         \Spec K \arrow[d] \arrow[r, "{\alpha}"] & \SpRep^{\square, 2d}_{(R,*)} \sslash \Sp_{2d}  \arrow[d, "{\nu}"] \\ 
         \Spec B \arrow[r, "{(D,P)}"] & \SpDet^{2d}_{(R,*)},
    \end{tikzcd}
\end{center}
there exists a finite field extension $K''/K$ with $B''$ the integral closure of $B$ in $K''$, and a morphism $f \colon \Spec B'' \to \SpRep^{2d}_{(R,*)}$ such that $\phi \circ f$ fits in the diagram

\begin{center}
    \begin{tikzcd}
\Spec K'' \arrow[r] \arrow[d] & \Spec K \arrow[d] \arrow[r, "\alpha"] & \SpRep^{\square, 2d}_{(R,*)} \sslash \Sp_{2d} \arrow[d, "\nu"]
\\ \Spec B'' \arrow[urr,dashed, "\phi \circ f"]  \arrow[r] & \Spec B \arrow[r] &  \SpDet^{2d}_{(R,*)}.
\end{tikzcd}
\end{center}

This allows us to verify the valuative criterion.
\\ Now let $(D,P)$ be the symplectic determinant of $(R,*)$ associated to the point $\Spec B \to \SpDet^{2d}_{(R,*)}$. Our \Cref{reconsalgcl} together with \cite[2.19 (1)]{CWE} implies that there is a $\overline K$-linear semisimple symplectic representation $\rho \colon (R \otimes_A \overline K, *) \to (M_{2d}(\overline K), \mathsf j)$ such that the corresponding point $\Spec \overline K \to \SpRep^{2d}_{(R,*)}$ lies above $\alpha$:

\begin{center}
    \begin{tikzcd}
    \Spec \overline{K} \arrow[r, "\rho"] \arrow[d] & \SpRep^{2d}_{(R,*)} \arrow[d, "\phi"]
    \\ \Spec K \arrow[r,"\alpha"] & \SpRep^{\square, 2d}_{(R,*)} \sslash \Sp_{2d}
    \end{tikzcd}
\end{center}

By \cite[Theorem 2.12]{MR3444227} and \cite[Lemma 2.8]{MR3444227}, we have $\ker(\rho) \cap (R \otimes_A K) = \ker(D \otimes_A K)$. Hence, the action of $R \otimes_A K$ on $\overline K^n$ factors through $(R \otimes_A K)/\ker(D \otimes_A K)$, which is finite-dimensional over $K$ by \cite[Corollary 2.14]{CWE}. By \Cref{exfiniteext}, there is a finite extension $K'/K$ and a symplectic representation $\rho \colon  R \otimes_A K' \to (M_{2d}(K'), \mathsf j)$, which induces $(D \otimes_A K', P \otimes_A K')$.

Let $B'$ be the integral closure of $B$ in $K'$.
Let $V' \colonequals  (K')^{2d}$ be the $K'$-vector space realizing $\rho$.
Let $L \subseteq V'$ be a $B'$-lattice and as in the proof of \cite[Theorem 2.20]{CWE}, we may assume that $L$ is $R$-stable.
The symplectic bilinear form on $V'$ restricts to a $B'$-bilinear form $\beta \colon  L \times L \to K'$; beware that we do not know a priori whether $\beta$ has values in $B'$. Choose a basis $x_1, \dots, x_{2d}$ of $L$ and let $F$ be the fundamental matrix of $\beta$, i.e. $F_{ij} = \beta(x_i, x_j)$. Letting $\varpi$ be a uniformizer of $B'$, we have $\det(F) = a \varpi^r$ with $a \in (B')^{\times}$ and $r \in \bbZ$.

We find a finite extension $K''/K'$ such that there is an element $z \in K''$ with $z^{4d} = \varpi^{-r}$.
Let $B''$ be the integral closure of $B'$ (and $B$) in $K''$.
The extension $L'' \colonequals  L \otimes_{B'} B''$ is a lattice in $V'' \colonequals  V' \otimes_{K'} K''$ with basis $x_1, \dots, x_{2d}$. Extending $\beta$, we equip it with a $B''$-bilinear form $\beta \colon  L'' \times L'' \to K''$ with fundamental matrix $F$.
The rescaled lattice $zL''$ has basis $zx_1, \dots, zx_{2d}$ and fundamental matrix $z^2F$. It follows, that $\det(z^2F) = a$ and thus $\beta$ is non-degenerate on $zL''$. So there is a representation on the $B''$-lattice $zL''$ compatible with (the involution induced by) $\beta$, which gives $\rho \otimes {K''}$ after extension of scalars. To obtain an actual symplectic representation $R \otimes_A B'' \to (M_{2d}(B''), \mathsf j)$, we use \cite[Corollary 3.5]{MilnorHusemoller} which states that every non-degenerate bilinear form over $B''$ is congruent to the standard symplectic form.
\end{proof}

\begin{remark}
    We have restricted the discussion in this section to symplectic determinant laws in lack of a version of \Cref{iso0} for weak symplectic determinant laws.
    Recall that we have a canonical closed immersion $\SpDet_{2d}^{(R,*)} \hookrightarrow \wSpDet_{2d}^{(R,*)}$.
    The arguments of the proof of \Cref{adequate} show that it is a finite universal homeomorphism.
    In particular the surjection $$A[\wSpDet_{2d}^{(R,*)}] \twoheadrightarrow A[\SpDet_{2d}^{(R,*)}]$$ has nilpotent kernel and we have a canonical isomorphism $(\SpDet_{2d}^{(R,*)})_{\red} \cong (\wSpDet_{2d}^{(R,*)})_{\red}$.
    So we also have a finite universal homeomorphism $\SpRep^{\square, (R,*)}_{2d} \sslash \Sp_{2d} \to \wSpDet_{2d}^{(R,*)}$.
\end{remark}
\subsection{Isomorphism on the multiplicity free locus}\label{subsecisommultfree}
We now prove that $\nu$ is an isomorphism on the multiplicity free locus using the strategy of \cite[Proposition 2.24, Corollary 2.25]{CWE}. So let $(R,*)$ be an involutive $A$-algebra equipped with a symplectic GMA structure $\mathcal{E}=((e_i)_{i\in I},(\psi_j)_{j\in I_0\sqcup I_1})$ of type $\delta((I_0,I_1,I_2),\sigma, (d_i)_{i\in I})$. It is equipped with a symplectic determinant law $(D_{\mathcal{E}},P_{\mathcal{E}})$ induced from the universal adapted representation, and so we have a morphism 
\begin{equation}\label{adtorep}
    \Rep_{\Ad}^{\square}(R,*,\mathcal{E})\longrightarrow \SpRep_{(R,*,D_{\mathcal{E}},P_{\mathcal{E}})}^{\square}
\end{equation}
given by forgetting the adaptation. There is an action of the $A$-group scheme $$ G(\mathcal{E})=\prod_{i\in I_0}\Sp_{d_i}\times \prod_{i\in I_1} \GL_{d_i}\hookrightarrow \Sp_{2d}$$ by conjugation on $\Rep_{(R,*,D_{\mathcal{E}},P_{\mathcal{E}})}^{\square}$. Here for $i\in I_1$, the embedding is given by $\GL_{d_i}\hookrightarrow \Sp_{2d_i}$ via the map $M\mapsto \text{diag}(M,M^{-1})$. The stabilizer of an adaptation is the subgroup $Z(\mathcal{E})\colonequals \prod_{i\in I_0}\mu_2\times \prod_{i\in I_1}\mathbb{G}_m$. Therefore, $Z(\mathcal{E})$ acts on $\Rep^{\square}_{\Ad}(R,*,\mathcal{E})$, and we have the following result. 

\begin{proposition}
    \label{thmequivalenceRepDE}
    Let $(R,*,\calE)$ be a symplectic GMA over $A$.
    Then the natural map
    \begin{align}
        [\Rep^{\square}_{\Ad}(R,*,\calE) / Z(\calE)] \longrightarrow \Rep_{(R,*,D_{\calE}, P_{\calE})} \label{equivalenceRepDE}
    \end{align}
    of algebraic stacks over $A$ is an equivalence.
\end{proposition}

\begin{proof}
    Let $X$ be an $A$-scheme, and let $(V,b,\rho) \in \Rep_{(R,*,D,P)}(X)$. This is the data of a vector bundle $V$ over $X$, of a skew symmetric non-singular bilinear form $b\colon V\times V\to \mathcal{O}_X$, and a morphism of $A$-algebras with involution $\rho\colon (R,*)\to(\Gamma(X,\End_{\mathcal{O}_X}(V)),\sigma_b)$ such that $(D,P)=(\det\circ \rho, \Pf\circ(\rho\cdot J))$.
    The idempotents $e_i$ give rise to a decomposition
    \begin{align}
        V = \bigoplus_{i=1}^r V_i
    \end{align}
    where $V_i = \rho(e_i)V$ is a vector bundle of rank $d_i$.
    The $A$-algebra $e_i R e_i$ acts on $V_i$ via a homomorphism $e_i R e_i \to \End_{\calO_X}(V_i)$ and the action is faithful, since we have an isomorphism $\psi_i : e_i R e_i \to M_{d_i}(A)$ and the determinant law on $e_i R e_i$ is compatible with the deteminant law on $\End_{\calO_X}(V_i)$. It follows, that we have an isomorphism $M_{d_i}(\calO_X) \eqto \End_{\calO_X}(V_i)$. Suppose now, that $i \in I_0$, i.e. that $i$ corresponds to an irreducible symplectic factor of the residual representation.
    The short exact sequence
    \begin{align}
        1 \to \mu_2 \to \Sp_{d_i} \to \PGSp_{d_i} \to 1
    \end{align}
    of algebraic groups over $X$ induces an exact sequence of non-abelian étale Čech cohomology groups
    \begin{align}
        \check H^1_{\et}(X, \mu_2) \to \check H^1_{\et}(X, \Sp_{d_i}) \to \check H^1_{\et}(X, \PGSp_{d_i}) \label{cechsequenceSp}
    \end{align}
    We know by the above considerations that $V_i$ represents the trivial element of $\check H^1_{\et}(X, \PGSp_{2d})$, and so comes from a $\mu_2$-torsor representing an element of $\check H^1_{\et}(X, \mu_2)$. In other words, we have an isomorphism $V_i\cong \calL_i^{\oplus d_i}$ for a line bundle $\calL_i$ given with a trivialization $\phi_i\colon \calL_i\otimes \calL_i\xrightarrow{\sim}\mathcal{O}_X$, and such that the bilinear form $b_i$ comes from the standard symplectic form on $\calL_{i}^{\oplus d_i}$ composed with $\phi_i$.

    We proceed similarly when $i \in I_1$ and use the short exact sequence $1 \to \Gm \to \GL_{d_i} \to \PGL_{d_i} \to 1$ to obtain a line bundle $\calL_i$ such that $V_i\cong \calL_i^{\oplus d_i}$. Since the bilinear form restricts to a perfect pairing $b\colon V_i\times V_{\sigma(i)}\to \mathcal{O}_X$, we get an isomorphism $V_{\sigma(i)}\cong (\calL_i^{-1})^{d_i}$. If $i \in I_0$, we define a $\mu_2$-torsor $\calG_i \colonequals \calI som(\calO_X, \calL_i)$, where the trivialization of $\calO_X^{\otimes 2}$ is the multiplication map of $\calO_X$ and the Isom sheaf $\calI som$ is understood to parametrize isomorphisms of line bundles compatible with the fixed trivialization of their tensor squares. If $i \in I_1$, we define a $\Gm$-torsor $\calG_i \colonequals \calI som(\calO_X, \calL_i)$, where the Isom sheaf parametrizes isomorphisms of line bundles.
    By definition, for all $i \in I_0 \cup I_1$, the line bundle $\calL_i$ is trivialized by pullback along $\calG_i \to X$.
    We obtain a $Z(\calE)$-torsor $\calG := \prod_{i \in I_0 \cup I_1} \calG_i$, and
    the base change of $V$ along $\pi : \calG \to X$ is a trivial vector bundle. Consequently, we get a the symplectic representation $\rho\colon (R,*)\to(\Gamma(\mathcal{G},\End_{\mathcal{O}_{\mathcal{G}}}(\pi^* V)),\mathrm{j})$ which is adapted by the above considerations. Hence, it corresponds to a map $\calG \to \Rep_{\Ad}^{\square}(R, *, \calE)$ which is $Z(\calE)$-equivariant. We thereby have defined a map
    \begin{align}
        \Rep_{(R,*,D_{\calE},P_{\calE})} \longrightarrow [\Rep^{\square}_{\Ad}(R,*,\calE) / Z(\calE)]
    \end{align}
    which can be checked to be quasi-inverse to \eqref{equivalenceRepDE}.
\end{proof}

\begin{corollary}\label{ZEinvariants}
    Let $(R,*,\calE)$ be a symplectic GMA over $A$.
    Then $\Rep^{\square}_{\Ad}(R,*,\calE) \sslash Z(\calE) = \Spec(A)$ and $\Rep_{(R,*,D_{\calE},P_{\calE})} \to \Spec A$ is a good moduli space.
\end{corollary}

\begin{proof}
    Let $\calB$ be the symmetric algebra in \eqref{symmetricalgebraB}, and let $J \subseteq \calB$ be the ideal such that $\calB/J$ represents $\Rep_{\Ad}^\square(R,*,\mathcal{E})$. We claim that $(\calB/J)^{Z(\calE)} = A$. As $Z(\calE)$ is linearly reductive (since $2 \in A^{\times}$), we have $(\calB/J)^{Z(\calE)} = \calB^{Z(\calE)}/J^{Z(\calE)}$. When $i,j \in I_0 \cup I_1$ the action of $Z(\calE)$ on $\calA_{i,j}$ is given by $t_it_j^{-1}$, where $t_i$ is the coordinate corresponding to the $i$-th factor of $Z(\calE)$. When $i \in I_1$ and $j \in I_2$, then $Z(\calE)$ acts on $\calA_{i,j}$ by $t_it_{\sigma(j)}$. We can restrict to these two cases, since $A_{i,j}$ is identified with $A_{\sigma(j), \sigma(i)}$ in $\calB/J$. If $i \in I_0$, the $i$-th factor of $Z(\calE)$ is a $\mu_2$, so we have $t_i^2=1$, otherwise the factor is a $\Gm$ and there is no additional relation. 

    An elementary tensor $e := a_{i_1j_1} \odot \dots \odot a_{i_kj_k}$ with $a_{i_lj_l} \in \calA_{i_lj_l}$ in $\calB$ with either $i_l,j_l \in I_0 \cup I_1$ or $i_l \in I_1$ and $j_l \in I_2$ is $Z(\calE)$-invariant, if and only if every $i \in I_0$ occurs an even number of times and every $i \in I_1$ occurs exactly as often as a second index, as the sum of the number occurrences of $i$ as a first index and the number of occurrences of $\sigma(i)$ as a second index.

    Working modulo $J$, using the multiplication relation and (ASSO), we can multiply a subtensor $a_{ij} \odot a_{jk}$ of $e$ together and obtain an element $a_{ij}a_{jk} \in \calB/J$, which is in the image of $\calA_{ik}$. By lifting $a_{ij}a_{jk}$ to $\calA_{ik}$, we can create a new invariant elementary tensor in $\calB$, which maps to the same element of $\calB/J$ and has length $k-1$; we also call it $e$. By induction and using the involution relation to adjust the position of indices in $I_0$, we can achieve that $e$ contains no index in $I_0$. By using the involution relation again, we can assume that $e$ only contains elements of $\calA_{i,j}$ with $i \in I_1$ and $j \in I_1 \cup I_2$. By the description of invariant elementary tensors above, we can multiply subtensors of the form $a_{ij} \odot a_{jk}$ with $j \in I_1$ together to achieve, that no index in $I_1$ occurs in $e$ as a second index. Now all coordinates $t_i$ of $Z(\calE)$ act by a nonnegative power of $t_i$ on $e$. But this is only possible if $e$ is zero or has length zero.
    
    We conclude that the image of every invariant elementary tensor in $\calB/J$ lies in $\calA_{ii} = A$. The map $A \to \calB/J$ is injective by \Cref{solutionembedding}. So
    \begin{align}
        [\Rep^{\square}_{\Ad}(R,*,\calE) / Z(\calE)] \longrightarrow \Rep^{\square}_{\Ad}(R,*,\calE) \sslash Z(\calE) = \Spec(A) \label{lem2333}
    \end{align}
    is a good moduli space and the claim follows from \Cref{thmequivalenceRepDE}.
\end{proof}

\begin{lemma}\label{lemsh}
    Let $f : X \to Y$ be a universal homeomorphism of locally noetherian schemes which is locally of finite type. Let $x \in X$, $y \colonequals f(x)$ and assume that the map of strict henselizations $\calO_{Y,y}^{\sh} \to \calO_{X,x}^{\sh}$ is an isomorphism. Then there is an open neighborhood $V \subseteq Y$ of $y$ such that $f : f^{-1}(V) \to V$ is an isomorphism.
\end{lemma}

\begin{proof}
    We implicitly fix separable closures and a map $\kappa(y)^{\sep} \to \kappa(x)^{\sep}$ in the statement and throughout the proof.
    Since an étale universal homeomorphism is an isomorphism (this follows e.g. from \stackcite{025G}), we will find an open neighborhood $V$, such that the map $f : f^{-1}(V) \to V$ is étale.
    For this it is by \stackcite{039N} sufficient to show, that $\calO_{Y,y} \to \calO_{X,x}$ is flat and $\frakm_y \calO_{X,x} = \frakm_x$. Since $\calO_{X,x} \to \calO_{X,x}^{\sh}$ is faithfully flat (see \stackcite{07QM}) and $\calO_{Y,y} \to \calO_{X,x}^{\sh}$ is flat, we get that $\calO_{Y,y} \to \calO_{X,x}$ is flat. We wish to show, that the sequence
    \begin{align}
        0 \to \frakm_y \calO_{X,x} \to \calO_{X,x} \to \calO_{X,x}/\frakm_x \to 0 \label{Oxsequence}
    \end{align}
    is exact. By tensoring \eqref{Oxsequence} with $\calO_{X,x}^{\sh}$, we get an exact sequence
    \begin{align*}
        0 \to \frakm_y \calO_{X,x}^{\sh} \to \calO_{X,x}^{\sh} \to \calO_{X,x}^{\sh}/\frakm_x \to 0 \label{Oxshsequence}
    \end{align*}
    since $\frakm_y \calO_{Y,y}^{\sh}$ is the maximal ideal of $\calO_{Y,y}^{\sh}$.
    By faithful flatness of $\calO_{X,x} \to \calO_{X,x}^{\sh}$, we conclude that \eqref{Oxsequence} is exact.
\end{proof}

\begin{theorem}\label{isoonmultiplicityfreelocus}
    There exists an open subscheme $U \subseteq \SpDet^{2d}_{(R,*)}$ with the following two properties:
    \begin{enum}
        \item The set $U$ contains all points $y \in \SpDet^{2d}_{(R,*)}$, such that the symplectic determinant law associated to the map $\Spec(\kappa(y)^{\sep}) \to \SpDet^{2d}_{(R,*)}$ is multiplicity-free.\label{yone}
        \item The map $\nu$ of \eqref{bigGITdiagram} is an isomorphism onto $U$.
    \end{enum}
\end{theorem}

We adopt the strategy of the proof of \cite[Theorem 2.3.3.7]{MR3167286}, using strict henselizations in place of completions of local rings.
We thereby demonstrate that the argument is purely étale-local in nature.
As our GMAs are defined in terms of symplectic determinant laws instead of traces, we need no additional hypothesis on the residue characteristic on $y$ apart from $2 \in \kappa(y)^{\times}$. In fact, the hypothesis that $(2d)! \in \kappa(y)^{\times}$ in \cite[Theorem 2.3.3.7]{MR3167286} is superfluous. 

\begin{proof}
    We will write $X \colonequals \SpRep^{\square, 2d}_{(R,*)} \sslash \Sp_{2d}$ and $Y \colonequals \SpDet^{2d}_{(R,*)}$.
    Let $y \in Y$ as in \eqref{yone}, let $x \colonequals \nu^{-1}(y)$ and fix a map $\kappa(y)^{\sep} \to \kappa(x)^{\sep}$. By \Cref{lemsh}, we only have to show that the map $\calO_{Y,y}^{\sh} \to \calO_{X,x}^{\sh}$ is an isomorphism. We write $V_x \colonequals \Spec(\calO_{X,x}^{\sh})$ and $U_y \colonequals \Spec(\calO_{Y,y}^{\sh})$.

    By \Cref{strongHomTheorem} the universal symplectic determinant law $(D^u, P^u)$ over $\calO(Y)$ descends to the universal symplectic Cayley-Hamilton quotient $E^u$ of $(R,*,D^u, P^u)$.
    The specialization of $(D^u, P^u)$ at $\calO_{Y,y}^{\sh}$ descends to the symplectic Cayley-Hamilton quotient $(E,*)$ of $(R \otimes_{\calO(Y)} \calO_{Y,y}^{\sh}, *)$, and we have $(E,*, D^u, P^u) \cong (E^u,*, D^u, P^u) \otimes_{\calO(Y)} \calO_{Y,y}^{\sh}$. We can apply \Cref{hensCH} to $(E,*,D^u,P^u)$ and obtain a symplectic GMA structure $\calE$ with $(D^u, P^u) = (D_{\calE}, P_{\calE})$. Our strategy is to show that in
    \begin{equation}\label{twosquares}
        \begin{tikzcd}
            \Rep_{(E,*,D^u,P^u)} \ar[r] \ar[d, "\phi_x"] \ar[dd, bend right=40, swap, "\psi_x"] & \Rep_{(E^u,*,D^u,P^u)} \ar[d, "\phi"] \\
            V_x \ar[r] \ar[d, "\nu_x"] & X \ar[d, "\nu"] \\
            U_y \ar[r] & Y
        \end{tikzcd}
    \end{equation}
    the maps $\phi_x$ and $\psi_x$ are adequate moduli spaces, for then it follows from \cite[Main Theorem (5)]{Alper2014} that $\nu_x$ is an isomorphism. To see that $\phi_x$ is an adequate moduli space, we start by showing that the squares in \eqref{twosquares} are cartesian. The outer square is cartesian by the definitions and since the symplectic Cayley-Hamilton ideal commutes with base extensions (\Cref{stabilityCH}). By \stackcite{08HV} the map $\calO(X) \otimes_{\calO(Y)} \calO_{Y,y}^{\sh} \to \calO_{X,x}^{\sh}$ identifies its target with the strict henselization of its source, so we will show that the source is strictly henselian. Indeed by \Cref{adequate} the map $\calO_{Y,y}^{\sh} \to \calO(X) \otimes_{\calO(Y)} \calO_{Y,y}^{\sh}$ is a finite universal homeomorphism, so it is a local map of local rings. By finiteness the residue field of $\calO(X) \otimes_{\calO(Y)} \calO_{Y,y}^{\sh}$ is separably closed and by \stackcite{04GG} (10) it is henselian. This shows, that the bottom square is cartesian, hence the top square is cartesian and we conclude by \cite[Proposition 5.2.9 (1)]{Alper2014} and flatness of $V_x \to X$, that $\phi_x$ is an adequate moduli space.
    \\Now for $\psi_x$, we know by \Cref{thmequivalenceRepDE} that there is an equivalence of stacks
    \begin{align*}
        [\Rep^{\square}_{\Ad}(E,*,\calE) / Z(\calE)] \eqto \Rep_{(E,*,D^u, P^u)}
    \end{align*}
    and by \Cref{ZEinvariants} the $Z(\calE)$-invariants of $\calO_{Y,y}^{\sh}[\Rep^{\square}_{\Ad}(E,*,\calE)]$ coincide with $\calO_{Y,y}^{\sh}$.
    It follows from \cite[Theorem 9.1.4]{Alper2014} that $\psi_x$ is an adequate moduli space.
\end{proof}

\section{Symplectic and orthogonal matrix invariants}
\label{secinvthy}

Throughout \Cref{secinvthy}, we consider a reductive group scheme $G$ over $\bbZ$ with an embedding $G\hookrightarrow \GL_d$. It has an action by conjugaction on $G^m$ and on $M_d^m$ given by by $g \cdot (g_1, \dots, g_m) = (gg_1g^{-1}, \dots, gg_mg^{-1})$. This induces a rational action of $G$ on the affine coordinate ring $\bbZ[G^m]$ (resp. $\bbZ[M_d^m]$) of $G^m$ (resp. $M_d^m$). 
We will use the notation $\bbX^{(i)} \in M_d(\bbZ[M_d^m])$ corresponding to the $i$-th projection map $M_d^m \twoheadrightarrow M_d$, for the generic matrices.
And we also write $\bbX^{(i)} \in G(\bbZ[G^m])$ for the generic group elements. 

Our goal is to extend over $\bbZ$ the main theorem of \cite{Zubkov}, which is stated as follows: 
\begin{theorem}\label{SOinv}
Let $G=\Sp_d$ (for $d$ even) or $\mathrm O_d$, and let $K$ be an algebraically closed field (of characteristic $\neq 2$ in the orthogonal case). Then then invariant algebra $K[M_d^m]^G$ is generated by the elements $$(\bbX^{(1)},\dots_,\bbX^{(m)})\mapsto\sigma_i(Y_{j_1} \cdots Y_{j_s})$$ 
where $Y_{i}$ is either $\bbX^{(i)}$ or the symplectic (or orthogonal) transpose $(\bbX^{(i)})^*$, and $\sigma_i$ is the $i$-th coefficient of the characteristic polynomial.
\end{theorem}

In the following proposition, we use ideas of Donkin (see \cite{Don}) to find generators of the symplectic invariants of several matrices with integral coefficients. Let us mention that we lack a proof of the analogous statement in the orthogonal case.

\begin{proposition}\label{matrixinvariants}
The invariant algebra $\mathbb{Z}[M_{2d}^m]^{\Sp_{2d}}$ is generated by the elements
$$ (\bbX^{(1)},\dots, \bbX^{(m)})\mapsto\sigma_i(Y_{j_1} \cdots Y_{j_s})$$ 
defined in \Cref{SOinv}.
\end{proposition}

\begin{proof}
Let us write $\widetilde{R}=\mathbb{Z}[M_{2d}^m]^G$ and let $R\subseteq \widetilde{R}$ be the subalgebra generated by the functions defined in the statement of the proposition, so we need to show that this inclusion is an equality.
\\ Note that the algebra of regular functions on $m$ matrices has a natural grading 
\begin{equation*}
    K[M_{2d}^m]= \bigoplus_{\alpha\in \mathbb{N}^m} K[M_{2d}^m]_{\alpha}
\end{equation*}
defined by giving to the $(i,j)$-entry $x_{i,j}^{(l)}$ of the $l$-th matrix $X_l$ $(1\le l \le m)$ the degree $(0,..,1,..,0)$ (the $1$ is in the $l$-th position). In particular, the grading on $\mathbb{C}[M_{2d}^m]$ induces a grading on $R$ and $\widetilde{R}$. 
\\ By \cite[§ 3]{Don}, $K[M_{2d}^m]_{\alpha}$ has a good filtration as a $\GL_{2d}$-module. But as mentioned in the proof of \cite[\text{Theorem }3.9]{Donkin94}, the restriction to $\Sp_{2d}$ of a $\GL_{2d}$-module with a good filtration has a good filtration. From \cite[\text{Proposition} 1.2a(iii)]{Donkin90}, we get that $\dim K[M_{2d}^m]_\alpha^G$ is the coefficient of the character of $\nabla(0)$ in the expension of the character of the $G$-module $K[M_{2d}^m]$ as a $\mathbb{Z}$-linear combination of the characters of $\nabla(\lambda)$ for $\lambda\in X^+$ (loc.cit.). In particular $d_\alpha= \dim K[M_{2d}^m]_\alpha^G$ is the same for all algebraically closed fields $K$.
\\Since $\mathbb{C}\otimes_{\mathbb{Z}} R_{\alpha}= \mathbb{C}\otimes_{\mathbb{Z}} \widetilde{R}_{\alpha}= \mathbb{C}[M_{2d}^m]_\alpha$, we get that $\rank_{\mathbb{Z}} R_\alpha= \rank_{\mathbb{Z}} \widetilde{R}_{\alpha}=d_\alpha$. Also by \Cref{SOinv} we have a sequence of morphisms
\begin{center}
    \begin{tikzcd}
K\otimes_{\mathbb{Z}}R_\alpha \arrow[rr, bend left=20, twoheadrightarrow] \arrow[r]
& K\otimes_{\mathbb{Z}} \widetilde{R}_\alpha  \arrow[r] & K[M_{2d}^m]^G_{\alpha}
\end{tikzcd}
\end{center}
where all of the vector spaces have the same dimension $d_\alpha$, so all of the arrows are isomorphisms. In particular, we have $K\otimes_{\mathbb{Z}} R_\alpha \cong K\otimes_{\mathbb{Z}} \widetilde{R}_\alpha$ for every algebraically closed field $K$, and so $R_\alpha=\widetilde{R}_\alpha$ which is enough to conclude.
\end{proof}

The proof the following Theorem follows a similar approach to \Cref{matrixinvariants}, yet it requires a new idea. This is due to the fact that the algebra $\mathbb{Z}[\Sp_{2d}^m]$ lacks a grading by finite dimensional subspaces having a good filtration. Our solution is to use truncation functors (see \cite[§A]{Jantzen2003}) to establish a well behaved filtration on $\mathbb{Z}[\Sp_{2d}^m]$ as an alternative to the grading. 

\begin{theorem}
    
\label{invZ}
The invariant algebra $\mathbb{Z}[\Sp_{2d}^m]^{\Sp_{2d}}$ is generated by the elements
$$ (\bbX^{(1)},\dots ,\bbX^{(m)})\mapsto\sigma_i(Y_{j_1} \cdots Y_{j_s}) $$
defined in \Cref{SOinv}.
\end{theorem}

\begin{proof}
Let $T$ be a maximal torus of $\Sp_{2d}$ and let $(\pi_n)_{n\ge 1}$ be an ascending sequence of finite saturated subsets of $X^+(T)$ such that $\bigcup_{n\ge 1}\pi_n= \pi= X^+(T)$, which is possible since $\Sp_{2d}$ is semisimple. For a field $K$, let $O_{\tau}$ be the truncation functor associated to a finite saturated subset $\tau\subseteq X^+(T^m)$ whose definition and properties we are going to use are given in \cite[§A]{Jantzen2003}. This definition makes sense over $\mathbb{Z}$ for a finite saturated $\tau$ by setting $O_{\tau}(\mathbb{Z}[\Sp_{2d}^m]) \colonequals  O_{\tau}(\mathbb{Q}[\Sp_{2d}^m]) \cap \mathbb{Z}[\Sp_{2d}^m]$, which is a finitely generated free $\bbZ$-module. We have for any field $K$ (\cite[§A.24]{Jantzen2003})
\begin{equation}\label{bctf}
    O_{\tau}(\mathbb{Z}[\Sp_{2d}^m])\otimes_{\mathbb{Z}} K= O_{\tau}(K[\Sp_{2d}^m]) 
\end{equation}
For the cartesian power $\pi^m = X^+(T)^m$, we have $\pi^m = \bigcup_{n \geq 1} \pi_n^m$ and $\pi_n^m$ are finite saturated subsets for the group $\Sp_{2d}^m$.
By definition, we have $O_{\pi^m}(\mathbb{Q}[\Sp_{2d}^m])=\mathbb{Q}[\Sp_{2d}^m]$  and since $O_{\pi^m}(\mathbb{Q}[\Sp_{2d}^m])=\bigcup_{n\ge 1} O_{\pi_n^m}(\mathbb{Q}[\Sp_{2d}^m])$ (\cite[§A.1]{Jantzen2003}), we get that $(O_{\pi_n^m}(\mathbb{Z}[\Sp_{2d}^m]))_{n\ge 1}$ is an ascending filtration of $\mathbb{Z}[\Sp_{2d}^m]$.

Now let $R$ be the subalgebra of $\mathbb{Z}[\Sp_{2d}^m]^{\Sp_{2d}}$ generated by the elements in the statement of the proposition and let $R_n \colonequals  R \cap O_{\pi_n^m}(\mathbb{Z}[\Sp_{2d}^m])$. By \cite[Lemma A.15]{Jantzen2003}, for any field $K$ $O_{\pi_n^m}(K[\Sp_{2d}^m])$ is finite-dimensional and admits a good filtration as an $\Sp_{2d}^m\times \Sp_{2d}^m$-module (for the left action induced by left multiplication by the first factor and inverse right multiplication by the second factor on $\Sp_{2d}^m$) with factors $\nabla(\lambda)\otimes \nabla(-w_0\lambda)$ for $\lambda \in \pi_n^m$. By \cite[Theorem 3.3]{Donkin94}, the tensor product of two induced modules $\nabla(\lambda)\otimes \nabla(\lambda')$ admits a good filtration, hence $O_{\pi_n}(K[\Sp_{2d}^m])$ admits a good filtration as an $\Sp_{2d}^m$-module under conjugation. But by \cite[Lemma I.3.8]{Jantzen2003}, $\nabla(\lambda)= \otimes_i \nabla(\lambda_i)$ for $\lambda=(\lambda_i)_{1\le i \le m}\in X^+(T^m)$, so by the same argument as before, we get that $O_{\pi_n^m}(K[\Sp_{2d}^m])$ admits a good filtration as an $\Sp_{2d}$-module.  It follows from \cite[Lemma B.9]{Jantzen2003} that $O_{\pi_n^m}(\mathbb{Z}[\Sp_{2d}^m])$ admits a good filtration as an $\Sp_{2d}$-module, hence by \cite[Proposition 1.2a (iii)]{Donkin90} 
$$
\rank_{\mathbb{Z}} O_{\pi_n^m}(\mathbb{Z}[\Sp_{2d}^m])^{\Sp_{2d}}= \dim_K O_{\pi_n^m}(K[\Sp_{2d}^m])^{\Sp_{2d}} =: d_n$$
for any field $K$. We have an exact sequence
$$ 0 \to R_n \rightarrow \mathbb{Z}[\Sp_{2d}^m]\rightarrow (\mathbb{Z}[\Sp_{2d}^m]/R) \times (\mathbb{Z}[\Sp_{2d}^m]/O_{\pi_n^m}(\mathbb{Z}[\Sp_{2d}^m])), $$
so tensoring with $\bbQ$ gives an exact sequence
$$ 0 \to R_n \otimes \bbQ \rightarrow \bbQ[\Sp_{2d}^m]\rightarrow (\bbQ[\Sp_{2d}^m]/(R\otimes \bbQ)) \times (\bbQ[\Sp_{2d}^m]/O_{\pi_n^m}(\bbQ[\Sp_{2d}^m])). $$
By \cite[Proposition 3.2]{Zubkov}, we have $R \otimes \bbQ = \bbQ[\Sp_{2d}^m]^{\Sp_{2d}}$, so the kernel of the rightmost arrow is $\bbQ[\Sp_{2d}^m]^{\Sp_{2d}} \cap O_{\pi_n^m}(\bbQ[\Sp_{2d}^m]) = O_{\pi_n^m}(\bbQ[\Sp_{2d}^m])^{\Sp_{2d}}$. Hence $R_n\otimes_{\mathbb{Z}} \mathbb{Q}= O_{\pi_n}(\mathbb{Q}[\Sp_{2d}^m])^{\Sp_{2d}}$, and in particular we get that $\rank_{\mathbb{Z}}R_n=d_n$. We claim that $R_n$ is cotorsion-free in $R$: by definition $R/R_n$ embeds into $\mathbb{Z}[\Sp_{2d}^m]/O_{\pi_n^m}(\mathbb{Z}[\Sp_{2d}^m])$, which is torsion-free since $O_{\pi_n^m}(\mathbb{Z}[\Sp_{2d}^m])$ is a saturated $\bbZ$-submodule of $\mathbb{Z}[\Sp_{2d}^m]$.

Let $K$ be an algebraically closed field.
The top map in the following diagram
\begin{align*}
    \xymatrix{
        R \otimes K \ar[r]^{\cong} & K[\Sp_{2d}^m]^{\Sp_{2d}} \\
        R_n \otimes K \ar@{^{(}->}[u] \ar@{^{(}->}[r] & O_{\pi^m_n}(K[\Sp_{2d}^m])^{\Sp_{2d}} \ar@{^{(}->}[u]
    }
\end{align*}
is an isomorphism by \cite[Proposition 3.2]{Zubkov}. So the bottom map injective. Since $\rank_{\mathbb{Z}}R_n=d_n$, it must be an isomorphism.
We deduce, that in the following diagram,
\begin{center}
    \begin{tikzcd}
    R_n\otimes_{\mathbb{Z}}K \arrow[rr, bend left=20] \arrow[r] &
    O_{\pi_n^m}(\mathbb{Z}[\Sp_{2d}^m])^{\Sp_{2d}} \otimes_{\mathbb{Z}} K  \arrow[r] & O_{\pi_n^m}(K[\Sp_{2d}^m])^{\Sp_{2d}}
    \end{tikzcd}
\end{center}
all maps are isomorphisms. Since this is true for every algebraically closed field $K$, the map $R_n \to O_{\pi_n^m}(\mathbb{Z}[\Sp_{2d}^m])^{\Sp_{2d}}$ of finitely generated free $\bbZ$-modules is an isomorphism.
So we get that $R = \mathbb{Z}[\Sp_{2d}^m]^{\Sp_{2d}}$ as desired.
\end{proof}

\begin{corollary}\label{gspZ}
    The invariant algebra $\bbZ[\GSp_{2d}^m]^{\GSp_{2d}}$ is generated by the functions 
    $$(\bbX^{(1)},\dots_,\bbX^{(m)})\mapsto\sigma_i(Y_{j_1}\cdots Y_{j_s}) \quad \text{ and } \quad (\bbX^{(1)},\dots_,\bbX^{(m)})\mapsto \lambda^{-1}(\bbX^{(i)})$$ 
   where $Y_{i}$ is either $\bbX^{(i)}$ or the symplectic transpose $(\bbX^{(i)})^{\mathrm{j}}$, and $\sigma_i$ is the $i$-th coefficient of the characteristic polynomial.
\end{corollary}

\begin{proof}
    The proof is based on a remark in \cite[§3]{Zubkov}. Consider the canonical morphism of algebraic groups $\pi\colon \Sp_{2d}\times \mathbb{G}_m \to \GSp_{2d}$ which is surjective since it is surjective on geometric points. This surjectivity can be seen by direct computation and more generally follows, since this surjection arises from a canonical decomposition sequence (see \cite[Example 19.25]{Milne}). The same is true for $\pi^{\times m}$ so we get an injection $(\pi^{\times m})^*\colon  \bbZ[\GSp_{2d}^m]\hookrightarrow \bbZ[(\Sp_{2d}\times \mathbb{G}_m)^m]$, since both rings are integral domains and the map induces a dominant morphism on spectra. Note that $\pi$ is equivariant for the action of $\Sp_{2d} \times \bbG_m$. Therefore we get that $$\bbZ[\GSp_{2d}^m]^{\GSp_{2d}} \subseteq \bbZ[\GSp_{2d}^m]^{\Sp_{2d} \times \bbG_m} \hookrightarrow \bbZ[(\Sp_{2d}\times \mathbb{G}_m)^m]^{(\Sp_{2d}\times \mathbb{G}_m)^m} = \bbZ[\Sp_{2d}^m]^{\Sp_{2d}} \otimes \bbZ[\bbG_m^m]$$
    and this map is clearly surjective, so the claim follows from \Cref{invZ}.
\end{proof}

\begin{remark} The statements of \Cref{matrixinvariants}, \Cref{invZ} and \Cref{gspZ} hold after replacing $\bbZ$ by an arbitrary commutative ring $A$. Indeed, since the $\Sp_{2d}$-modules $\bbZ[M_{2d}^m]$, $\bbZ[\Sp_{2d}^m]$ and the $\GSp_{2d}$-module $\bbZ[\GSp_{2d}^m]$ have good filtrations, taking invariants commutes with tensoring with $A$. The same arguments go through for the orthogonal groups $\mathrm O_d$, and the general orthogonal groups $\GO_d$. Consequently, we obtain the same generators with the symplectic similitude character replaced by the orthogonal similitude character.
\end{remark}

\section{Comparison with Lafforgue's pseudocharacters}
\label{secLafforgue}
In \cite[§11]{Laf}, Lafforgue introduced a notion of pseudocharacters for general reductive groups which we recall below in the form of \cite[Definition 4.1]{BHKT}. The reader is invited to consult \cite{BHKT} and \cite{Quast} for applications of this notion in the context of deformation theory. 

For $\GL_n$, Emerson proved Lafforgue's definition is equivalent to Chenevier's notion of determinant laws (see \cite[Theorem 4.1 (ii)]{emerson2023comparison}). We expect that the bijection constructed in \cite{emerson2023comparison} restricts to a bijection between Lafforgue's pseudocharacters for the symplectic groups and symplectic determinant laws over commutative $\bbZ[\tfrac{1}{2}]$-algebras. In this section, we establish this result for reduced $\bbZ[\tfrac{1}{2}]$-algebras and arbitrary $\bbQ$-algebras.

\begin{definition}\label{LafPC} Let $G$ be a reductive $\bbZ$-group scheme, let $\Gamma$ be an abstract group, and let $A$ be a commutative ring. A \emph{$G$-pseudocharacter} $\Theta$ of $\Gamma$ over $A$ is a sequence of ring homomorphisms $$\Theta_m \colon  \bbZ[G^m]^G \longrightarrow \map(\Gamma^m,A)$$ for each $m \geq 1$, satisfying the following conditions:
\begin{enum}
    \item For all $n,m \geq 1$, each map $\zeta \colon \{1, \dots, m\} \to \{1, \dots,n\}$, every $f \in \bbZ[G^m]^G$, and all $\gamma_1, \dots, \gamma_n \in \Gamma$, we have
    $$ \Theta_n(f^{\zeta})(\gamma_1, \dots, \gamma_n) = \Theta_m(f)(\gamma_{\zeta(1)}, \dots, \gamma_{\zeta(m)}) $$
    where $f^{\zeta}(g_1, \dots, g_n) = f(g_{\zeta(1)}, \dots, g_{\zeta(m)})$.
    \item For all $m \geq 1$, all $\gamma_1, \dots, \gamma_{m+1} \in \Gamma$, and every $f \in \bbZ[G^m]^G$, we have
    $$ \Theta_{m+1}(\hat f)(\gamma_1, \dots, \gamma_{m+1}) = \Theta_m(f)(\gamma_1, \dots, \gamma_m\gamma_{m+1}) $$
    where $\hat f(g_1, \dots, g_{m+1}) = f(g_1, \dots, g_mg_{m+1})$.
\end{enum}
\end{definition}
We denote the set of $G$-pseudocharacters of $\Gamma$ over $A$ by $\PC^G_{\Gamma}(A)$.
If $f \colon A \to B$ is a ring homomorphism, then there is an induced map $f_* \colon \PC^G_{\Gamma}(A) \to \PC^G_{\Gamma}(B)$.
This defines a functor $\PC^G_{\Gamma} \colon \CAlg_{\bbZ} \to \Set$, which is representable by a commutative ring $\mathscr{L}^G_{\Gamma} = \bbZ[\PC^G_{\Gamma}]$ (see \cite[Theorem 2.15]{Quast}).

A representation $\rho : \Gamma \to G(A)$ gives rise to a $G$-pseudocharacter $\Theta_{\rho}$, which depends only on $\rho$ up to $G(A)$-conjugation. Here $(\Theta_{\rho})_m : \bbZ[G^m]^G \to \map(\Gamma^m, A)$ is defined by 
$$ (\Theta_{\rho})_m(f)(\gamma_1, \dots, \gamma_m) := f(\rho(\gamma_1), \dots, \rho(\gamma_m)) $$
The definition of $G$-pseudocharacter can be brought into a more convenient and practical form.
Let $\calF \colonequals  \{\mathrm{FG}(m) \mid m \geq 1\}$ be the category of finitely generated free groups $\mathrm{FG}(m)$ on $m$ letters.
Then the associations $\bbZ[G^{\bullet}]^G \colon  \mathrm{FG}(m) \mapsto \bbZ[G^m]^G$ and $\map(\Gamma^{\bullet}, A) \colon \mathrm{FG}(m) \mapsto \map(\Gamma^m, A)$ give rise to functors $\calF \to \CAlg_{\bbZ}$. There is a natural bijection 
$$ \PC^G_{\Gamma}(A) \cong \Nat(\bbZ[G^{\bullet}]^G, \map(\Gamma^{\bullet}, A)) $$
for any commutative ring $A$ (see \cite[Proposition 2.14]{Quast}).

\subsection{Comparison for $\Sp_{2d}$} For $m \geq 1$, the $\Sp_{2d}$-module $\bbZ[\Sp_{2d}^m]$ under diagonal conjugation has a good filtration and $H^i(\Sp_{2d}, \bbZ[\Sp_{2d}^m]) = 0$ for all $i > 0$ \cite[§B.9]{Jantzen2003}. In particular for any homomorphism of commutative rings $A \to B$, we have
$$ B[\Sp_{2d}^m]^{\Sp_{2d}} \cong \bbZ[\Sp_{2d}^m]^{\Sp_{2d}} \otimes_{\bbZ} B \cong (\bbZ[\Sp_{2d}^m]^{\Sp_{2d}} \otimes_{\bbZ} A) \otimes_A B \cong A[\Sp_{2d}^m]^{\Sp_{2d}} \otimes_A B $$

Now we are in shape to define a comparison map in one direction.

\begin{proposition}\label{LafChenMap} Let $\Theta^u \in \PC_{\Gamma}^{\Sp_{2d}}(\mathscr{L}^{\Sp_{2d}}_{\Gamma})$ be the universal $\Sp_{2d}$-pseudocharacter and let $C$ be a commutative $\mathscr{L}^{\Sp_{2d}}_{\Gamma}$-algebra. Using the isomorphism $C[\Sp_{2d}^m]^{\Sp_{2d}} \cong \mathscr{L}^{\Sp_{2d}}_{\Gamma}[\Sp_{2d}^m]^{\Sp_{2d}} \otimes_{\mathscr{L}^{\Sp_{2d}}_{\Gamma}} C$, the map $\Theta^u_m$ induces a homomorphism $\Theta^u_{m,C} \colon  C[\Sp_{2d}^m]^{\Sp_{2d}} \to \map(\Gamma^m, C)$ for all $m \geq 1$. We define the maps
\begin{align*}
    D_C \colon  C[\Gamma] \to C, \quad &\sum_{i=1}^m c_i \gamma_i &\mapsto \quad &\Theta^u_{m,C} \left(\det\big(\sum_{i=1}^m c_i \bbX^{(i)} \big)\right)(\gamma_1, \dots, \gamma_m) \\
    P_C \colon  C[\Gamma]^+ \to C, \quad &\sum_{i=1}^m c_i (\gamma_i+\gamma_i^{-1}) &\mapsto \quad &\Theta^u_{m,C} \left(\Pf\big(\sum_{i=1}^m c_i (\bbX^{(i)}+(\bbX^{(i)})^{-1})J \big)\right)(\gamma_1, \dots, \gamma_m)
\end{align*}
Then $D$ is a $\mathscr{L}^{\Sp_{2d}}_{\Gamma}$-valued $2d$-dimensional $*$-determinant law, and $P$ is a $d$-homogeneous polynomial law with $P^2 = D|_{\mathscr{L}^{\Sp_{2d}}_{\Gamma}[\Gamma]^+}$ and $P(1)=1$. In particular this defines natural map $\PC_{\Gamma}^{\Sp_{2d}}(A) \to \SpDet_{\bbZ[\tfrac{1}{2}][\Gamma]}^{2d}(A)$ for every commutative $\bbZ[\tfrac{1}{2}]$-algebra $A$.
\end{proposition}

\begin{proof} The way the maps are defined is functorial, so clearly $D$ and $P$ are polynomial laws. We check the multiplicativity of $D$ by noticing that 
\begin{equation*}
    \Theta^u_{m+m',C} \left(\det\big(\sum_{i=1}^m c_i \bbX^{(i)} \big)\right) \Theta^u_{m+m',C} \left(\det\big(\sum_{j=1}^{m'} c_j' \bbX^{(m+j)} \big)\right) = \Theta^u_{m+m',C} \left(\det\big(\sum_{i=1}^m \sum_{j=1}^{m'} c_i c_j' \bbX^{(i)} \bbX^{(m+j)} \big)\right).
\end{equation*}
Now define
$$ \mu \colonequals  \det\big(\sum_{i=1}^m \sum_{j=1}^{m'} c_i c_j' \bbX^{(i)}\bbX^{(m+j)} \big) \quad\quad \mu' \colonequals  \det\big(\sum_{i=1}^m \sum_{j=1}^{m'} c_i c_j' \bbX^{(i + (j-1)m)} \big), $$
so that
\begin{align*}
    \phantom{=} \Theta^u_{m+m',C} \left(\mu\right)(\gamma_1, \dots, \gamma_m, \gamma_1', \dots, \gamma_{m'}') 
    =\Theta^u_{mm',C} \left(\mu' \right) (\gamma_1\gamma_1', \gamma_1\gamma_2', \dots, \gamma_m\gamma_{m'}')
\end{align*}
holds by a suitable substitution in an $\calF\text{-}\bbZ$-algebra (see \cite[§2.4]{Quast}). The homogeniety of $D$ and $P$, the $*$-invariance of $D$, and the equalities $P^2 = D|_{C[\Gamma]^+}$ and $P(1) = 1$ follow by a similar substitution. The fact that $\CH(P)\subseteq \ker(D)$ follows from the surjection $C[M_{2d}^m]^{\Sp_{2d}}\twoheadrightarrow C[\Sp_{2d}^m]^{\Sp_{2d}}$, and so any relation that holds on $C[M_{2d}^m]^{\Sp_{2d}}$ also holds on $C[\Sp_{2d}^m]^{\Sp_{2d}}$.
\end{proof}

\begin{lemma}\label{injrem} The map $\PC_{\Gamma}^{\Sp_{2d}}(A) \to \SpDet^{2d}_{A[\Gamma]}(A)$ defined in \Cref{LafChenMap} is injective.
\end{lemma}

\begin{proof} Indeed, the map $\PC_{\Gamma}^{\Sp_{2d}}(A) \to \PC^{\Gamma}_{\GL_{2d}}(A)$ induced by the standard embedding $\Sp_{2d} \hookrightarrow \GL_{2d}$ is injective, since the maps $\bbZ[\GL_{2d}^m]^{\GL_{2d}} \twoheadrightarrow \bbZ[\Sp_{2d}^m]^{\Sp_{2d}}$ are surjective (by \Cref{invZ}). The forgetful map $\SpDet_{2d}^{\Gamma}(A) \to \Det_{2d}^{\Gamma}(A)$ is injective by \Cref{pfaffianunique}. Since we have a bijection $\PC^{\Gamma}_{\GL_{2d}}(A) \xrightarrow{\sim} \Det_{2d}^{\Gamma}(A)$ by \cite[Theorem 4.1 (ii)]{emerson2023comparison}, the claim follows.
\end{proof}

\begin{proposition}\label{lafforguecomp} Let $A$ be either a reduced commutative $\bbZ[\tfrac{1}{2}]$-algebra or an arbitrary commutative $\bbQ$-algebra. Then the map $\PC_{\Gamma}^{\Sp_{2d}}(A) \to \SpDet_{2d}^{\Gamma}(A)$ defined in \Cref{LafChenMap} is bijective. In particular, we have canonical isomorphisms $\PC_{\Gamma}^{\Sp_{2d}}[\tfrac{1}{2}]_{\red} \cong (\SpDet_{2d}^{\Gamma})_{\red}$, and $\PC^{\Gamma}_{\Sp_{2d} \bbQ} \cong \SpDet_{2d, \bbQ}^{\Gamma}$.
\end{proposition}

\begin{proof} First, assume that $A$ is reduced with $2 \in A^{\times}$. By \Cref{injrem}, it is enough to show surjectivity. If  $(D,P) \in \SpDet_{2d}^{\Gamma}(A)$, we know by \cite[Theorem 4.1 (ii)]{emerson2023comparison} that there is some $\Theta \in \PC^{\Gamma}_{\GL_{2d}}(A)$ that maps to $D$. So it is enough to show that for all $m \geq 1$, $\Theta_m$ factors through $\bbZ[\Sp_{2d}^m]^{\Sp_{2d}}$. In particular, the following claim holds: if $A \to B$ is an injective homomorphism and $\PC_{\Gamma}^{\Sp_{2d}}(B) \to \SpDet_{2d}^{\Gamma}(B)$ is a bijection, then $\PC_{\Gamma}^{\Sp_{2d}}(A) \to \SpDet_{2d}^{\Gamma}(A)$ is a bijection. This reduces the proof of the proposition to the case of an algebraically closed field, as we know explain. First embed $A \hookrightarrow \prod_{\frakp} \overline{\Quot(A/\frakp)}$, where $\frakp$ varies over all prime ideals of $A$. Now if $A$ is an algebraically closed field, then by \Cref{reconsalgcl} there is a semisimple representation $\rho \colon \Gamma \to \Sp_{2d}(A)$ that induces $(D,P)$. The $\Sp_{2d}$-pseudocharacter induced by $\rho$ is necessarily mapped to $(D,P)$.

The comparison map sends a Lafforgue pseudocharacter associated to a representation $\rho$ to the symplectic determinant associated to $\rho$. Therefore, we have a diagram \begin{center}
    \begin{tikzcd}
         \SpRep_{2d}^{\square, \bbQ[\Gamma]} \sslash \Sp_{2d} \arrow[dr] \arrow[d] \\
         \PC_{\Sp_{2d}, \bbQ}^{\Gamma} \arrow[r] & \SpDet_{2d}^{\bbQ[\Gamma]}.
    \end{tikzcd}
\end{center}
The left vertical map is an isomorphism by \cite[Proposition 2.11 (i)]{emerson2023comparison}, and the right map is an isomorphism by \Cref{iso0}. It follows that the comparison map is an isomorphism over $\bbQ$.
\end{proof}

\begin{remark}
    If in \Cref{lafforguecomp} $\Gamma$ is finitely generated, then it follows from the 2-out-of-3 property for adequate homeomorphisms, \cite[Proposition 2.11 (ii)]{emerson2023comparison} and \Cref{adequate}, that the map $\PC_{\Gamma}^{\Sp_{2d}}[\tfrac12] \to \SpDet_{2d}^{\Gamma}$ is an adequate homeomorphism.
\end{remark}

\begin{remark}
    \Cref{lafforguecomp} leads to a new proof of the reconstruction theorem \cite[Theorem 4.5]{BHKT} for Lafforgue's $\Sp_{2d}$-pseudocharacters over an algebraically closed field of characteristic $\neq 2$. Indeed, in that case complete reducibility of representations is equivalent before and after composition with the standard representation $\Sp_{2d} \hookrightarrow \GL_{2d}$, see \cite[Exemple 3.2.2 (b)]{SerreCompleteReducibility}.
\end{remark}

\subsection{Comparison for $\GSp_{2d}$}

By our discussion on invariant theory \Cref{gspZ}, we see that the similitude character of a $\GSp_{2d}$-pseudocharacter $\Theta \in \PC_{\GSp_{2d}}^{\Gamma}(A)$ can be recovered as $\lambda_{\Theta} := \Theta_1(\lambda) : \Gamma \to A^{\times}$.

\begin{proposition}\label{LafChenMapGSp}
    Let $\Theta^u \in \PC^{\Gamma}_{\GSp_{2d}}(\mathscr{L}_{\GSp_{2d}}^{\Gamma})$ be the universal $\GSp_{2d}$-pseudocharacter and let $C$ be a commutative $\mathscr{L}_{\GSp_{2d}}$-algebra. $\Theta^u_m$ induces a homomorphism $\Theta^u_{m,C} \colon  C[\GSp_{2d}^m]^{\GSp_{2d}} \to \map(\Gamma^m, C)$ for all $m \geq 1$.
    Let $\lambda_C : \Gamma \to C^{\times}$ be the specialization of the universal similitude character $\lambda_{\Theta^u} : \Gamma \to (\mathscr{L}_{\GSp_{2d}}^{\Gamma})^{\times}$ at $C$. We define maps
    \begin{align*}
        D_C \colon  C[\Gamma] \to C, \quad &\sum_{i=1}^m c_i \gamma_i &\mapsto \quad &\Theta^u_{m,C} \left(\det\big(\sum_{i=1}^m c_i \bbX^{(i)} \big)\right)(\underline \gamma) \\
        P_C \colon  C[\Gamma]^+ \to C, \quad &\sum_{i=1}^m c_i (\gamma_i+ \lambda_C(\gamma_i) \gamma_i^{-1}) &\mapsto \quad &\Theta^u_{m,C} \left(\Pf\big(\sum_{i=1}^m c_i (\bbX^{(i)}+\lambda_C(\gamma_i)(\bbX^{(i)})^{-1})J \big)\right)(\underline \gamma)
    \end{align*}
    where $\underline \gamma := (\gamma_1, \dots, \gamma_m) \in \Gamma^m$.

    Then $D\colon \mathscr{L}_{\GSp_{2d}}^{\Gamma}[\Gamma] \to \mathscr{L}_{\GSp_{2d}}^{\Gamma}$ is a $2d$-dimensional $*$determinant law with respect to the involution on $\mathscr{L}_{\GSp_{2d}}^{\Gamma}[\Gamma]$ given by $\gamma^* = \lambda_{\Theta^u}(\gamma)\gamma^{-1}$ for $\gamma\in \Gamma$, and $P$ is a $d$-homogeneous polynomial law with $P^2 = D|_{\mathscr{L}_{\GSp_{2d}}^{\Gamma}[\Gamma]^+}$ and $P(1)=1$. In particular this defines natural map $\PC^{\Gamma}_{\GSp_{2d}}(A) \to \GSpDet^{\Gamma}_{2d}(A)$ for every commutative $\bbZ[\tfrac{1}{2}]$-algebra $A$.
\end{proposition}

\begin{proof}
    This follows by a similar computation as in \Cref{LafChenMap}.
\end{proof}

\begin{lemma}\label{injGSp}
    The map $\PC^{\Gamma}_{\GSp_{2d}}(A) \to \GSpDet_{2d}^{\Gamma}(A)$ defined in \Cref{LafChenMapGSp} is injective.
\end{lemma}

\begin{proof}
    The map $\PC^{\Gamma}_{\GSp_{2d}}(A) \to \PC^{\Gamma}_{\GL_{2d}}(A) \times \Hom(\Gamma, A^{\times})$ induced by the standard representation $\GSp_{2d} \to \GL_{2d}$ and the similitude character is injective, since the maps $\bbZ[\GL_{2d}^m]^{\GL_{2d}} \otimes \bbZ[\mathbb{G}_m^m] \twoheadrightarrow \bbZ[\GSp_{2d}^m]^{\GSp_{2d}}$ are surjective by \Cref{gspZ}. The map $\GSpDet_{2d}^{\Gamma}(A) \to \Det_{2d}^{\Gamma}(A) \times \Hom(\Gamma, A^{\times})$ forgetting the Pfaffian is injective by \Cref{pfaffianunique}. The claim follows, since we have a bijection $\PC^{\Gamma}_{\GL_{2d}}(A) \times \Hom(\Gamma, A^{\times}) \to \Det_{2d}^{\Gamma}(A) \times \Hom(\Gamma, A^{\times})$ by \cite[Theorem 4.1 (ii)]{emerson2023comparison}.
\end{proof}

\begin{proposition}\label{lafforguecompGSp} Let $A$ be either a reduced commutative $\bbZ[\tfrac{1}{2}]$-algebra or an arbitrary commutative $\bbQ$-algebra. Then the map $\PC^{\Gamma}_{\GSp_{2d}}(A) \to \GSpDet_{2d}^{\Gamma}(A)$ defined in \Cref{LafChenMapGSp} is bijective. In particular we have canonical isomorphisms $\PC^{\Gamma}_{\GSp_{2d}}[\tfrac{1}{2}]_{\red} \cong (\GSpDet_{2d}^{\Gamma})_{\red}$ and $\PC^{\Gamma}_{\GSp_{2d}, \bbQ} \cong \GSpDet_{2d, \bbQ}^{\Gamma}$.
\end{proposition}

\begin{proof}
    The proof of \Cref{lafforguecomp} applies by invoking \Cref{injGSp}, \cite[Theorem 4.1 (ii)]{emerson2023comparison} and \Cref{reconsalgcl} and \Cref{iso0} with $\Sp_{2d}$-conjugation replaced by $\GSp_{2d}$-conjugation.
\end{proof}

\bibliographystyle{alpha}
\bibliography{literature}

\end{document}